\documentclass[11pt,leqno]{amsart}
\usepackage{url}
\usepackage{xspace}
\usepackage{amsfonts,amssymb}
\usepackage{amsthm}
\usepackage[all]{xy}
\usepackage{textcomp}

\usepackage{mathrsfs}
\usepackage{stmaryrd}
\usepackage{bbm}
\usepackage{oldgerm}
\usepackage[T1]{fontenc}
\usepackage[latin1]{inputenc}
\theoremstyle{definition}
\newtheorem{defn}[subsection]{Definition}
\newtheorem{rem}[subsection]{Remark}

\theoremstyle{plain}\newtheorem{lemma}[subsection]{Lemma}

\newtheorem{prop}[subsection]{Proposition}
\newtheorem{thm}[subsection]{Theorem}
\newtheorem{cor}[subsection]{Corollary}

\numberwithin{equation}{subsection}

\newcommand{\bem}{\begin{bmatrix}}
\newcommand{\enm}{\end{bmatrix}}

\newcommand{\ee}{\mathbf{e}}
\newcommand{\ff}{\mathbf{f}}
\newcommand{\ra}{\rightarrow}

\newcommand{\hra}{\hookrightarrow}

\newcommand{\xra}{\xrightarrow}

\newcommand{\Z}{\mathbb{Z}}
\newcommand{\C}{\mathbb{C}}
\newcommand{\F}{\mathbb{F}}

\newcommand{\Q}{\mathbb{Q}}

\newcommand{\bB}{\mathbb{B}}

\newcommand{\cA}{\mathcal{A}}

\newcommand{\cF}{\mathscr{F}}

\newcommand{\cJ}{\mathcal{J}}
\newcommand{\cL}{\mathcal{L}}

\newcommand{\cM}{\mathcal{M}}
\newcommand{\cO}{\mathcal{O}}
\newcommand{\cV}{\mathcal{V}}

\newcommand{\fd}{\mathfrak{d}_L} 
\newcommand{\fp}{\mathfrak{p}}  

\newcommand{\fX}{\mathfrak{X}}
\newcommand{\fY}{\mathfrak{Y}}

\newcommand{\vk}{\underline{k}} 
\newcommand{\vx}{\underline{x}} 
\newcommand{\vdeg}{\underline{\deg}} 

\newcommand{\taml}{\Gamma_1} 
\newcommand{\prtoX}{\pi} 
\newcommand{\prtoY}{\alpha} 

\newcommand{\IHMV}{Z} 
\newcommand{\fHMV}{\mathfrak{Z}} 

\newcommand{\Cb}{\overline{C}}
\newcommand{\Db}{\overline{D}}

\newcommand{\an}{\mathrm{an}}
\newcommand{\can}{\mathrm{can}}
\newcommand{\univ}{\mathrm{univ}}
\newcommand{\cl}{cl^{+}} 
\newcommand{\pr}{\mathrm{pr}}

\newcommand{\omegab}{\underline{\omega}}

\newcommand{\one}{\underline{1}} 

\newcommand{\codim}{{\operatorname{codim }}}

\newcommand{\Fr}{{\operatorname{Fr }}}

\newcommand{\Ker}{{\operatorname{Ker}}}

\newcommand{\ord}{{\operatorname{ord }}}

\newcommand{\Res}{{\operatorname{Res }}}

\newcommand{\Spec}{{\operatorname{Spec }}}

\newcommand{\Sym}{{\operatorname{Sym }}}

\newcommand{\Proj}{{\operatorname{Proj}}}

\newcommand{\GL}{{\operatorname{GL}}}
\newcommand{\Gal}{{\operatorname{Gal}}}

\newcommand{\gera}{{\frak{a}}}

\newcommand{\gerd}{{\frak{d}}}

\newcommand{\gerl}{{\frak{l}}}
\newcommand{\germ}{{\frak{m}}}
\newcommand{\gern}{{\frak{n}}}

\newcommand{\gerp}{{\frak{p}}}
\newcommand{\gerq}{{\frak{q}}}

\newcommand{\gert}{{\frak{t}}}

\newcommand{\gerL}{{\frak{L}}}
\newcommand{\gerM}{{\frak{M}}}
\newcommand{\gerN}{{\frak{N}}}

\newcommand{\gerS}{{\frak{S}}}
\newcommand{\gerT}{{\frak{T}}}
\newcommand{\gerU}{{\frak{U}}}
\newcommand{\gerV}{{\frak{V}}}

\newcommand{\gerX}{{\frak{X}}}
\newcommand{\gerY}{{\frak{Y}}}
\newcommand{\gerZ}{{\frak{Z}}}

\newcommand{\ua}{{\underline{a}}}
\newcommand{\ub}{{\underline{b}}}

\newcommand{\uA}{{\underline{A}}}

\newcommand{\calA}{{\mathcal{A}}}

\newcommand{\calF}{{\mathcal{F}}}

\newcommand{\calM}{{\mathcal{M}}}

\newcommand{\calO}{{\mathcal{O}}}

\newcommand{\calR}{{\mathcal{R}}}
\newcommand{\calS}{{\mathcal{S}}}

\newcommand{\calU}{{\mathcal{U}}}
\newcommand{\calV}{{\mathcal{V}}}
\newcommand{\calW}{{\mathcal{W}}}

\def\AA{\mathbb{A}}
\def\BB{\mathbb{B}}
\def\CC{\mathbb{C}}

\def\GG{\mathbb{G}}

\def\II{\mathbb{I}}

\def\QQ{\mathbb{Q}}
\def\RR{\mathbb{R}}
\def\SS{\mathbb{S}}

\newcommand{\arr}{{\; \rightarrow \;}}

\newcommand{\ol}{{\mathcal{O}_L}}

\newcommand{\Spf}{\operatorname{Spf }}

\newcommand{\Xrig}{{\gerX_\text{\rm rig}}}

\newcommand{\Yrig}{{\gerY_\text{\rm rig}}}

\newcommand{\rig}{{\operatorname{rig}}}
\newcommand{\spe}{{\operatorname{sp }}}

\newcommand{\Xbar}{\overline{X}}
\newcommand{\Ybar}{\overline{Y}}

\newcommand{\pe}{(\varphi,\eta)}

\newcommand{\Wpe}{W_{\varphi,\eta}}

\newcommand{\Qbar}{\overline{Q}}
\newcommand{\Pbar}{\overline{P}}

\newcommand{\Qpbar}{\overline{\QQ}_p}

\newcommand{\tYrig}{\tilde{\gerY}_{\rm{rig}}}
\newcommand{\tXrig}{\tilde{\gerX}_{\rm{rig}}}

\newcommand{\Yrigtau}{\gerY^{|\tau|\leq 1}_{\rm{rig}}}
\newcommand{\Xrigtau}{\gerX^{|\tau|\leq 1}_{\rm{rig}}}

\newcommand{\HMV}{Z}

\newcommand{\tHMV}{\tilde{Z}}

\newcommand{\tHMVAN}{\tilde{\gerZ}_{\rm rig}}

\newcommand{\tHMVANtau}{\tilde{\gerZ}_{\rm rig}^{|\tau|\leq 1}}

\newcommand{\phix}{\varphi_{\vx}}
\newcommand{\etax}{\eta_{\vx}}
\newcommand{\Wx}{W_{\vx}}

\newcommand{\gen}{\mathrm{gen}}
\newcommand{\bfa}{\mathbf{a}}
\newcommand{\sing}{\mathrm{sing}}
\newcommand{\region}{\prtoY^{-1}(\Sigma)}
\newcommand{\uuA}{\underline{\underline{A}}}
\begin{document}

\begin{abstract}
We extend the modularity lifting result of \cite{Ka} to allow Galois representations with some ramification at $p$. We also prove modularity mod  $5$ of certain Galois representations. We use these results to prove  new cases of the strong Artin conjecture over  totally real fields in which $5$ is unramified. As an ingredient of the proof, we provide a general result on the automatic analytic continuation of overconvergent $p$-adic Hilbert  modular forms of finite slope which substantially generalizes a similar result in \cite{Ka}.
\end{abstract}

\title[Modularity lifting results in parallel weight one: the tamely ramified case]{Modularity lifting results in parallel weight one and applications to the Artin conjecture: the tamely ramified case}
\author{Payman L Kassaei, Shu Sasaki, Yichao Tian}
\maketitle

\section{Introduction}

The work of Buzzard and Taylor \cite{BT} beautifully combined methods of Wiles and Taylor (\`a la Diamond \cite{Diamond}) with a geometric analysis of overconvergent $p$-adic  modular forms to prove a modularity lifting result for geometric representations of $\Gal(\overline{\QQ}/\QQ)$ which are split and unramified at $p$. Combined with the works of  Shepherd-Barron-Taylor and Dickinson, these ideas came together in \cite{BDST} where a program laid out by R. Taylor came to fruition in proving many cases of the Artin conjecture over $\QQ$. Later, Buzzard generalized the results of \cite{BT}, allowing the Galois representation to have ramification at $p$, and these results were used by Taylor to prove more cases of the Artin conjecture.

In \cite{Ka},  the first-named author proved a generalization of  the main result of Buzzard-Taylor \cite{BT} over a totally real field in which $p$ is unramified. In \cite{Pilloni}, Pilloni proved a generalization of  this result in the case where  $p$ is slightly ramified in $F$. In both works the Galois representation is assumed unramified at $p$.  In this work, we generalize the result of \cite{Ka} allowing some ramification at $p$ for the Galois representation  in question. 

To explain our method, let us first give a brief description of the proof in \cite{Ka}. At the heart of the argument in \cite{Ka} lies the proof that a collection of weight one specilaizations of Hida families are all classical Hilbert modular forms by analytically extending them from their initial domains of definition to the entire Hilbert modular variety $\Yrig$ of level $\Gamma_1(N) \cap \Gamma_0(p)$ (where $N$ is prime to $p$). This analytic continuation is done in two steps. First, one proves that every finite slope overconvergent  Hilbert modular form will {\it automatically} extend to a region $\calR$ inside $\Yrig$; a region considerably larger than the locus of definition of the canonical subgroup, $\calV_{can}$, over which one could previously perform automatic analytic continuation (see \cite{GK}). In the second step, using a gluing argument, one shows that a linear combination of the weight one forms at hand extends to a yet bigger region $\calR \subset \Yrigtau\subset \Yrig$  which contains a region saturated with respect to the  forgetful map $\pi: \Yrig \arr \Xrig$ to the Hilbert modular variety $\Xrig$ of level $\Gamma_1(N)$. It is, then, proved that this linear combination  descends under the map $\pi$ to a region $\Xrigtau=\pi(\Yrigtau) \subseteq \Xrig$, over which a rigid-analytic Koecher principle can be applied to extend the descended form to the entire $\Xrig$. The point is that while $\calR$ and $\Yrigtau$ are not large enough to make the application of the Koecher principle possible, their images under $\pi$ are so. 

In the setting of this article, we can no longer descend the forms to level $\Gamma_1(N)$, and, hence, we must prove a more optimal  analytic continuation at the level of $\Yrig$ itself. In this work, we prove that every finite slope overconvergent Hilbert modular form extends to a region $\Sigma$ which is vastly larger than $\calR$: a region only slightly smaller than the tube of the ``generic'' locus of the special fibre $\Ybar$  (in codimension \!<\! 2), defined in terms of the stratification of $\Ybar$ studied in \cite{GK}. 

\begin{thm}
Any overconvergent Hilbert modular form of finite slope (at all primes above $p$) extends analytically to
\[
\Sigma=\bigcup_{\codim(W)=0,1}]W^\gen[',
\]
where $W$ runs over strata of $\Ybar$ (defined in \cite{GK}) of codimension $0,1$, $W^{gen}$ is the generic locus of $W$ defined in \ref{Defn:generic-part}, $]W^{gen}[$ is the tube of $W^{gen}$ in $\Yrig$, and $]W^{gen}[^\prime\subset\ ]W^{gen}[$ is defined in \ref{Defn:Sigma}.\end{thm}

This provides the strongest automatic analytic continuation result for overconvergent Hilbert modular forms in the literature so far, and is of independent interest besides the applications explored in this article.   This is done in \S \ref{section: analytic continuation}.

To explain the second step of the proof, let us assume that the prime $p$ is inert in  $F$ for simplicity. In this case, the collection of weight one overconvergent forms consists of two forms, say $f,g$, each of which can be extended to $\Sigma$ by the first step. Let $w$ be the Atkin-Lehner involution on $\Yrig$. To further extend $f$, we want to glue $f$ defined on $\Sigma$ to a scalar multiple of $w(g)$ which is defined on $w^{-1}(\Sigma)$. This would extend $f$ to $\Sigma \cup w^{-1}(\Sigma)$. The main difficulty in this step, compared to the elliptic case, is that $\Sigma \cap w^{-1}(\Sigma)$  has many  connected components which need to be controlled. A detailed analysis of these connected components is carried out in \S \ref{section: gluing} and utilizes the geometry of intersection of {\it irreducible components} of strata on $\Ybar$ provided by \cite{GK}.
Another complication, essentially due to the existence of totally positive units in $F$, is that $f$ may, in fact, not be a scalar multiple of $w(g)$, and a replacement for $g$ has to be constructed.

The third step of the analytic continuation requires a close study of the stratification on $\Ybar$ which is carried out in \S \ref{sect-geom-Y}. Using these results, one shows that the complement of $\Sigma \cup w^{-1}(\Sigma)$ in $\Yrig$ has codimension at least $2$ in reduction mod $p$. It follows that $f$ automatically extends from $\Sigma \cap w^{-1}(\Sigma)$ to the entire $\Yrig$ by the Koecher principle. We should mention that there are smaller choices of $\Sigma$ that make this last step work fine. However, the region $\Sigma$ defined above is chosen such that the connected components of $\Sigma \cap w^{-1}(\Sigma)$ are well-behaved: we show that every such connected component contains  a region which is saturated under $\pi$; this is essential for carrying out the gluing  process described above. The above analytic continuation results allow us to prove:

\begin{thm}\label{Theorem: Main}
Let $p>2$ be a
prime number, and $F$ a totally real field in which $p$ is unramified. For any prime ideal $\gerp|p$, let
$D_\gerp$ denote a decomposition group of $Gal(\overline{\QQ}/F)$ at $\gerp$, and $I_\gerp$ the inertia subgroup.  Let $\rho: Gal(\overline{\QQ}/F) \arr \GL_2(\calO)$ be a continuous representation, where $\calO$ is the ring of integers in a finite extension of $\QQ _p$, and $\germ$ its maximal ideal. Assume

\begin{itemize}

\item $\rho$ is unramified outside a finite set of primes,

\item For every prime $\gerp | p$, we have 
\[
\rho_{|_{D_\gerp}} \cong \alpha_\gerp \oplus \beta_\gerp,
\]
 where $\alpha_\gerp,\beta_\gerp: D_\gerp \arr \calO^\times$ are characters distinct modulo $\germ$, and $\alpha_\gerp(I_\gerp)$ and $\beta_\gerp(I_\gerp)$ are finite, and $\alpha_\gerp/\beta_\gerp$ is tamely ramified,

\item $\overline{\rho}:=(\rho\ {\rm mod}\ \germ)$ is ordinarily modular, i.e., there exists a classical Hilbert modular form $g $  of parallel weight $2$ such that ${\rho}\equiv \rho_g(\mathrm{mod}\  \germ)$ and $\rho_{g}$ is potentially ordinary and potentially Barsotti-Tate at every prime of $F$ dividing $p$,

\item $\overline{\rho}$ is absolutely irreducible when restricted to $Gal(\overline{\QQ}/F(\zeta_p))$.

\end{itemize}

Then, $\rho$ is isomorphic to $\rho_f$,  the Galois representation
associated to a Hilbert modular eigenform $f$ of weight
$(1,1,\cdots,1)$ and level $\Gamma_{1} (Np)$, for some integer $N$
prime to $p$.

\end{thm}

\begin{rem} We can handle the case where $\alpha_\gerp/\beta_\gerp$ has arbitrary ramification provided there is an integral model of the Hilbert modular variety of level $\Gamma_1(Np^m)$ (for $m>1$) such that the forgetful map  to level $\Gamma_1(N,p)$ is flat.
\end{rem}
In the last two sections of the paper, we apply our results to prove certain cases of the strong Artin conjecture over totally real fields. We first prove modularity of certain Galois representations mod $5$ following a method of Taylor. Combining these results with Theorem \ref{Theorem: Main}, we prove new  icosahedral cases of the strong Artin conjecture over  $F$ (generalizing the main theorem of Taylor in \cite{T:02} to $F$ while removing conditions at $3$ and $5$ given there). 

\begin{thm} Let $F$ be a totally real field in which $5$ is unramified.
Let
\[
\rho: \mathrm{Gal}(\overline{\mathbb{Q}}/F)\rightarrow GL_2(\mathbb{C})
\]
be a totally odd and continuous representation satisfying the following conditions:
\begin{itemize}
\item $\rho$ is totally odd;
\item $\rho$ has the projective image $A_5$;
\item for every place $\mathfrak{p}$ of $F$ above $5$, the projective image of the decomposition group at $\mathfrak{p}$ has order 2.
\end{itemize}
Then, there exists a holomorphic Hilbert cuspidal eigenform $f$ of weight 1 such that $\rho$ arises from $f$ in the sense of Rogawski-Tunnell, and the Artin $L$-function
$L(\rho, s)$ is entire.
\end{thm}
 
 After a first version of this article was submitted to Arxiv, Pilloni-Stroh announced a result generalizing the work of Pilloni \cite{Pilloni}.

\subsection{Notation} \label{Subsection: notation}

Let $p$ be a prime number. Let $L/\mathbb{Q}$ be a totally real field of degree $g>1$ in which $p$ is unramified, $\mathcal{O}_L$ its ring of integers, and $\gerd_L$ its different ideal.    For a  prime ideal $\gerp$  of $\ol$ dividing $p$, let $\kappa_\gerp = \ol/\gerp$, a finite field of order $p^{f_\gerp}$. Let $\kappa$ be a finite field containing an isomorphic copy of all $\kappa_\gerp$ which is generated by their images. We identify $\kappa_\gerp$ with a subfield of $\kappa$ once
 and for all. Let $\mathbb{Q}_\kappa$ be the fraction field of $W(\kappa)$.
 We fix embeddings $\mathbb{Q}_\kappa \subset \QQ_p^{\rm ur}
 \subset \Qpbar$. Let $v_p$ the $p$-adic valuation on $\Qpbar$ so that $v_p(p)=1$.

Let $\SS=\{\gerp|p\}$ be the set of prime ideals of $\ol$ dividing $p$. Let 
\[
\mathbb{B}={\rm Emb}(L,\mathbb{Q}_\kappa)=\textstyle\coprod_{\gerp \in \SS}
\mathbb{B}_\mathfrak{p},
\]
where $\mathbb{B}_\mathfrak{p}= \{\beta\in\mathbb{B}\colon
\beta^{-1}(pW(\kappa)) = \mathfrak{p}\}$, for every prime ideal $\mathfrak{p}$ dividing $p$. If $\gert | (p)$ is an ideal, we define $\BB_\gert=\cup_{\gerp|\gert} \BB_\gerp$. We also set $\gert^\ast=(p)/\gert$.

Let $\sigma$ denote the
Frobenius automorphism of $\mathbb{Q}_\kappa$, lifting $x \mapsto
x^p$ modulo $p$. It acts on $\mathbb{B}$ via $\beta \mapsto \sigma
\circ \beta$, and transitively on each $\mathbb{B}_\mathfrak{p}$.
For $S \subseteq \mathbb{B}$, we let $S^c=\BB-S$, and
$\ell(S)=\{\sigma^{-1}\circ\beta\colon \beta \in S\}$, the left shift of $S$. Similarly, define the right shift of $S$, denoted $r(S)$.
We denote by $|S|$ the  cardinality of $S$. The decomposition \[\mathcal{O}_L
\otimes_\mathbb{Z} W(\kappa)=\bigoplus_{\beta \in \mathbb{B}}
W(\kappa)_\beta,\] where $W(\kappa)_\beta$ is $W(\kappa)$ with the
$\mathcal{O}_L$-action given by $\beta$, induces a decomposition,
\[M=\bigoplus_{\beta\in \mathbb{B}} M_\beta,\] on any $\mathcal{O}_L
\otimes_\mathbb{Z} W(\kappa)$-module $M$.

\

Let $N\geq 4$ be an integer prime to $p$.  Throughout the paper, we will fix a finite extension $K$ of $\QQ_\kappa$ with residue field $\kappa_K$, and ring of integers $\calO_K$.


\subsection{The Hilbert modular variety of level $\taml(N) \cap \Gamma_0(p)$}

We denote by $X$ the Hilbert modular scheme of level $\taml(N)$ over $\Spec(\calO_K)$, and $Y$ the Hilbert modular scheme over $\Spec(\calO_K)$ of level $\taml(N)\cap\Gamma_0(p)$. The base extensions of $X,Y$ to $\Spec(K)$ are denoted, respectively, $X_K,Y_K$. The connected components of $X$ and $Y$ are both in natural bijection with the set of strict ideal class group $\cl_L$. For a representative $(\gera,\gera^+)$ of $\cl_L$, we denote the corresponding connected component of $X$ (respectively, $Y$) by $X_\gera$ (respectively, $Y_\gera$). We apply the same convention to other variants of $X,Y$ appearing in this paper. Let $\Xbar$, $\Ybar$ be, respectively, the special fibres of $X$ and $Y$, $\fX$ and $\fY$, the formal completions of $X$ and $Y$ along their special fibres, and $\fX_{\rig}$ and $\fY_{\rig}$,  the associated rigid analytic generic fibres over $K$. By abuse of notation, we always denote by  $\prtoX$ the natural forgetful projection from  level $\taml(N)\cap\Gamma_0(p)$ to $\taml(N)$ in various settings, e.g., $\prtoX: Y\ra X$, $\prtoX: \fY\ra \fX$, $\prtoX: \fY_{\rig}\ra \fX_{\rig}$.

Let $\cA^{\univ}$ be the universal abelian scheme  over $X$, and $e: {X}\ra \cA^{\univ}$ be  the unit section.  We put ${\omega}=e^*\Omega_{\cA^{\univ}/X}$. This is a locally free $(\cO_{X}\otimes_{\Z}\cO_L)$-module of rank 1, so that  we have a canonical decomposition $\omega=\oplus_{\beta\in \bB}\omega_{\beta}$ according to the natural action of $\cO_{L}$ on $\omega$. For any $\vk=(k_{\beta})_{\beta\in \bB}\in \Z^{\bB}$, we put
\[
\omega^{\vk}=\otimes_{\beta\in \bB}\omega_{\beta}^{\otimes k_{\beta}}.
\]
We still denote by $\omega$, $\omega_{\beta}$ and $\omega^{\vk}$ their pull-backs to $Y$ via $\prtoX$ or the corresponding sheaves in the formal or rigid analytic settings.

 We choose toroidal compactifications $\tilde{X}$ and $\tilde{Y}$ based on a common fixed choice of rational polyhedral cone decompositions. We still denote by $\prtoX$ the natural map $\tilde{Y}\ra \tilde{X}$. There is a semi-abelian scheme $\tilde{\cA}$ over $\tilde{X}$ extending the universal abelian scheme $\cA^{\univ}$. We denote by $\tilde{\fX}$ and  $\tilde{\fY}$ the corresponding formal completions along the special fibres, and by $\tilde{\fX}_{\rig}=\tilde{X}_{K}^{\an}$ and $\tilde{\fY}_{\rig}=\tilde{Y}_{K}^{\an}$  the associated rigid analytic spaces. For any $\vk\in \Z^{\bB}$, the line bundle $\omega^{\vk}$ on $X$ (resp. on $Y$) extends to a line bundle over $\tilde{X}$ (resp. $\tilde{Y}$), still denoted by $\omega^{\vk}$, by using the semi-abelian scheme $\tilde{\cA}$ over $\tilde{X}$ (resp. $\tilde{Y}$).

  Let $\tilde{\fX}^{\ord}_{\rig}$ be the ordinary locus of $\fX_{\rig}$, i.e., the quasi-compact admissible open subdomain of $\tilde{\fX}_{\rig}$, where the rigid semi-abelian scheme $\tilde{\cA}^{\an}$ has good ordinary reduction or specializes to cusps,   and $\fX_{\rig}^{\ord}= \tilde{\fX}_{\rig}^{\ord}\cap \fX_{\rig}$. Similarly, let $\fY_{\rig}^{\ord}$ be the locus of $\prtoX^{-1}(\fX^{\ord}_{\rig})$ where the finite flat subgroup $H$ is of multiplicative type,  and $\tilde{\fY}^{\ord}_{\rig}\subset \tilde{\fY}_{\rig}$ be the union of $\fY_{\rig}^{\ord}$ together with the locus of $\tilde{\fY}_{\rig}$ with reduction to unramified cusps.

\subsection{The Hilbert modular variety of level $\Gamma_1(Np)$. } Recall our fixed choice of $K$, a finite extension of $\QQ_\kappa$.  We assume that $K$ contains the $p$-th roots of unity. Let $\HMV_K$ be the Hilbert modular variety of level $\taml(Np)$ over $\Spec(K)$, i.e., the scheme  that represents the functor attaching to a $K$-scheme $S$, the set of isomorphic classes of the 4-tuples, $(A, \lambda, i_N, P)$, where
   \begin{itemize}
   \item $A$ is an abelian scheme of dimension $g$ over $S$ with real multiplication by $\cO_L$;

   \item $\lambda:A\ra A^{\vee}$ is a prime-to-$p$ polarization compatible with the action of $\cO_L$;

   \item $i_N: \mu_N\otimes_{\Z}\fd^{-1}\hra A[N]$ is an $\cO_L$-equivariant  closed immersion of group schemes;

 \item $P: S\ra A[p]$ is a section of the finite flat group scheme $A[p]$ of order $p$, such that the $\cO_L$-subgroup generated by $P$  is a free $(\cO_L/p)$-module  of rank $1$.
   \end{itemize}

We will often use the abbreviation $\uA=(A,\lambda,i_N)$, and denote a $S$-valued point of $\HMV_K$ simply by $(\uA,P)$.
Let $H=(P)$ be the $\cO_L$-subgroup of $A[p]$ generated by $P$. We have a canonical decomposition $H=\prod_{\gerp\in \SS}H[\gerp]$ corresponding to the decomposition $\cO_L/p=\prod_{\gerp \in \SS} \cO_L/\gerp$. We denote by $P_{\gerp}$ the image of $P$ under the natural projection $H\ra H[\gerp]$. Then $P_{\gerp}$ is a generator of $H[\gerp]$ as $\cO_L$-module, and we have $P=\prod_{\gerp\in \SS}P_{\gerp}$. Let
\[
\prtoY: \HMV_K\ra Y_K
\]
be the  map  given by $(\uA,P)\mapsto (\uA, (P))$. The map $\prtoY$ is finite \'etale of degree $\prod_{\gerp\in \SS}(p^{f_{\gerp}}-1)$.

 We choose the same rational polyhedral cone decomposition as that of ${Y}$ to construct a toroidal compactification $\tHMV_K$ of $\HMV_K$. The morphism $\prtoY$ extends to a finite flat map ${\prtoY}:\tHMV_K\ra \tilde{Y}_K$. For any $\vk=(k_{\beta})_{\beta\in \bB}\in \Z^{\bB}$, we define $\omega^{\vk}$ on $\tHMV_K$ to be the pullback of $\omega^{\vk}$ on $\tilde{Y}_K$ under $\prtoY$.

   For uniformity of notation, we will denote $\tHMV_K^{\rm an}$, the rigid analytification of $\tHMV_K$,  by $\tHMVAN$. We will continue to denote by ${\prtoY}: \tHMVAN\!=\! \tHMV_K^{\rm an} \ra \tilde{Y}_K^{\an}\!=\!\tilde{\fY}_{\rig}$, the analytification of $\prtoY:\tHMV_K\ra \tilde{Y}_K$. We set $\tHMVAN^{\ord} =\prtoY^{-1}(\tilde{\fY}_{\rig}^{\ord})$. The analytification of the sheaf $\omega^{\vk}$ on $\tHMVAN$ will be denoted by the same symbol.

\subsubsection{An integral model of $\HMV_K$}

We will need  the integral model of $\HMV_K$ over $\cO_K$ defined in \cite{Pa}. Let $H\subset \cA^{\univ}[p]$ be the universal isotropic $(\cO_L/p)$-cyclic subgroup over $Y$. We have a canonical decomposition $H=\prod_{\gerp\in \SS}H[\gerp]$, where each $H[\gerp]=\Spec(\cO_{H[\gerp]})$ is a scheme of one-dimensional $(\cO_L/\gerp)$-vector spaces over $Y$. By Raynaud's classification of such  group schemes \cite[1.4.1]{Ray}, there exist invertible sheaves $\cL_{\beta}$  over $Y$, for each  $\beta\in \BB_{\gerp}$, together with $\cO_Y$-linear morphisms $\Delta_{\beta}: \cL_{\beta}^{\otimes p}\ra\cL_{\sigma\circ\beta}$ and $\Gamma_{\beta}:\cL_{\sigma\circ\beta}\ra \cL_{\beta}^{\otimes p}$ such that $\Delta_{\beta}\circ \Gamma_{\beta}$ and $\Gamma_{\beta}\circ\Delta_{\beta}$ are both multiplication by $p$, and  the $\cO_Y$-algebra $\cO_{H[\gerp]}$ is isomorphic to
 \[
\Sym_{\cO_Y}(\oplus_{\beta\in \bB}\cL_{\beta})/((1-\Delta_{\beta})\cL^{\otimes p}:\beta\in \bB_{\gerp}).
\]
In fact, $\cL_{\beta}$ is the direct summand of the augmentation ideal $\cJ_{H[\gerp]}\subset \cO_{H[\gerp]}$, where $(\cO_L/\gerp)^{\times}$ acts via the Teichm\"uller character $\chi_{\beta}:(\cO_L/\gerp)^{\times}\ra W(\kappa_K)^{\times}$. Now we consider the closed subscheme $H'[\gerp]$ of $H[\gerp]$ defined by the equation
\[
(\otimes_{\beta\in\bB}\Delta_{\beta}-1)(\otimes_{\beta\in \bB}\cL_{\beta}^{\otimes (p-1)}).
\]

 Explicitly, let $U=\Spec(R)$ be an affine open subset of $Y$ such that
\[
H[\gerp]|_U=\Spec(R)[T_{\beta}:\beta\in \bB_{\gerp}]/(T_{\beta}^p-y_{\sigma\circ\beta}T_{\sigma\circ\beta}:\beta\in\bB_{\gerp}),
\] with $y_{\beta}\in R$, for $\beta\in \bB_{\gerp}$. Then, we have
\[H'[\gerp]|_U=\Spec(R[T_{\beta}:\beta\in\bB_{\gerp}])/(T_{\beta}^p-y_{\sigma\circ\beta}T_{\sigma\circ\beta}, \prod_{\beta\in\bB_{\gerp}}T_{\beta}^{p-1}-\prod_{\beta\in\bB_{\gerp}}y_{\beta}).\]
From this local description, we see that $H'[\gerp]$ is a finite flat scheme over $Y$ of rank $p^{f_{\gerp}}-1$. We set $H':=\prod_{\gerp\in\SS}H'[\gerp]$.

\begin{prop}[\cite{Pa}, 5.1.5, 2.3.3]
 The $Y$-scheme $H'\ra Y$ represents the functor which associates to each $Y$-scheme $S$, the set of $(\cO_L/p)$-generators of $H\times_{Y}S$ in the sense of Drinfeld-Katz-Mazur \cite[1.10]{KM}. Consequently, the scheme $H'$ is an integral model of $\HMV$ over $\cO_K$, which is finite flat of degree $\prod_{\gerp\in \SS}(p^{f_{\gerp}}-1)$ over $Y$.
\end{prop}

In the sequel, we will define $\IHMV:=H'$, and call $\prtoY$ the natural map $\IHMV \ra Y$. For any $\vk$ as above, we let $\omega^{\vk}$ denote the pullback of $\omega^{\vk}$ on $Y$ under $\prtoY:\IHMV \ra Y$; it is an integral model for the restriction to $\HMV_K$ of the sheaf $\omega^{\vk}$ defined above on $\tHMV_K$.

We let  $\fHMV$ denote the formal completion of $\IHMV$ along its special fibre, and  $\fHMV_{\rig}$ its associated  rigid generic fibre. Therefore, $\fHMV_{\rig}$ is the quasi-compact admissible open subdomian of $\tHMVAN$ where the universal HBAV has good reduction.

\subsection{Atkin-Lehner automorphisms} \label{sect-AL-inv}  We first define these automorphisms over $Y$. Let $S$ be an $\cO_K$-scheme, and consider an $S$-valued point $Q$ of $Y$ corresponding to $(\uA, H)=(A,i_N,\lambda,H)$. We have a canonical decomposition $H=\prod_{\gerq\in \SS}H[\gerq]$. For any $\gerp\in \SS$, we put
\[
w_{\gerp}(Q)=(\uA/H[\gerp], H'),
\]
Where $\uA/H[\gerp]$ denotes the quotient of $A$ by $H[\gerp]$ along with its induce PEL data (as defined in \cite[\S 2.1]{GK}), and $H'=\prod_{\gerq}H'[\gerq]$, with $H'[\gerp]=A[\gerp]/H[\gerp]$ and $H'[\gerq]=H[\gerq]$, under the identification $(A/H[\gerp])[\gerq] \cong A[\gerq]$ for $\gerq\neq \gerp$. The automorphisms $w_{\gerp}$ for $\gerp\in\SS$ commute with each other.   For any $T\subset \SS$, we define $w_T=\prod_{\gerp\in T}w_{\gerp}$, and $w=w_{(p)}$.
 In view of
$$
w^2(Q)=(\prod_{\gerp\in \SS}w_{\gerp}^2)(Q)=(A,\lambda, pi_N, H),
$$
 we see that each $w_T$ is automorphism on $Y$. The automorphisms $w_T$ extend naturally to the chosen toroidal compactification $\tilde{Y}$.
We still denote by $w_T$ the automorphisms induced on $\tilde{Y}_K$, $\tilde{\Ybar}$, $\tYrig$, etc.

 We now turn to $Z_K$. Let us fix an element $\zeta_\gerp$ which is an $\calO_L$-generator of $(\GG_m \otimes \gerd_L^{-1})[\gerp]$.    We define an automorphism $w_\gerp$ on $\HMV_K$ for $\gerp\in\SS$ as follows. Let $x=(\uA, P)$ be a point of $\HMV_K$ with values in a $K$-scheme $S$. For each $\gerp \in \SS$,  the prime-to-$p$ polarization $\lambda$ induces a perfect Weil pairing
\[
\langle\cdot,\cdot\rangle_{\gerp}\colon A[\gerp]\times A[\gerp]\ra (\GG_m \otimes \gerd_L^{-1})[\gerp]
\]
Let $\phi_{\gerp}: A\ra A'=A/H[\gerp]$ be the canonical isogeny, and $\hat{\phi}_{\gerp}:A'\ra A/A[\gerp]$ be the canonical isogeny with kernel $A[\gerp]/H[\gerp]$.
  Since $H[\gerp] $ is (automatically) isotropic, the Weil pairing $\langle\cdot,\cdot\rangle_{\gerp}$ induces a perfect duality pairing
\[
\langle\cdot,\cdot \rangle_{\phi_{\gerp}}\colon \Ker(\phi_{\gerp})\times \Ker (\hat{\phi}_{\gerp})\ra (\GG_m \otimes \gerd_L^{-1})[\gerp].
\]
  We define $w_{\gerp}(x)=(\uA^\prime,Q)=(A', \lambda', i'_N, Q)$, where $\lambda'$ and $i_N'$ are respectively the induced polarization and $\taml(N)$-structure on $A'$, and  $Q=\prod_{\gerq\in \SS}Q_{\gerq}\in A'[p](S)$ is given as follows: For $\gerq\neq \gerp$, we put
$$
Q_{\gerq}=P_{\gerq}\in A'[\gerq]\simeq A[\gerp],
$$
 and $Q_{\gerp}$ is the unique point of $\Ker(\hat{\phi}_{\gerp})(S)$ such that $\langle P_{\gerp}, Q_{\gerp}\rangle_{\phi_{\gerp}}=\zeta_\gerp$.

We see easily that $w_{\gerp}$ and $w_{\gerq}$ commute for $\gerp,\gerq\in \SS$, so that $w_T=\prod_{\gerp | T }w_{\gerp}$ is well-defined for any $T \subset \SS$. We let $w=w_{(p)}$. Note that $w^2(A, \lambda,i_N,P)=(A, \lambda, pi_N, -P)$, since
\[
\langle P_{\gerp}, Q_{\gerp}\rangle_{\phi_{\gerp}}\langle Q_{\gerp}, P_{\gerp}\rangle_{\hat{\phi}_{\gerp}}=1
\] for any $\gerp\in \SS$. Since a certain power of $w^2$ is the identity,  it follows that each $w_T$ is an automorphism. The automorphism $w_T$ extends to the toroidal compacification $\tilde{\HMV}_K$. Via the natural projection $\prtoY: \tilde{\HMV}_K\ra \tilde{Y}_K$, the automorphisms $w_T$ on $\tilde{\HMV}_K$ and $\tilde{Y}_K$ are compatible.

\section{Preliminaries on the Geometry of $\Ybar$}\label{sect-geom-Y}

In this section, we prove some results on the geometry of the special fibre of $Y$, which will be useful in the analytic continuation process of later sections (see Theorem \ref{thm-ac}).



\subsection{Directional degrees and valuations} We recall the notions of directional degrees introduced in \cite{Ti,PS}, and  directional  valuations defined in \cite{GK}, and discuss the relationship between the two notions. Let $Q=(\uA, H)$ be a rigid point of $\tYrig$ defined over a  finite extension $K'/K$.

\textit{Case 1.  $\uA$ has good reduction over $\cO_{K'}$.} Then  $\uA$ and $H$ can be  defined over $\cO_{K'}$. Let $\omega_{H}$ be the module of invariant differential 1-forms of $H$. We have a canonical decomposition
$
\omega_H =\bigoplus_{\beta\in\BB}\omega_{H,\beta},
$
 where each $\omega_{H,\beta}$ is a torsion $\cO_{K'}$-module generated  by one element. So there exists $a_\beta\in \cO_{K'}$ such that $\omega_{H,\beta}\simeq \cO_{K'}/(a_\beta)$. We set
\[
\deg_{\beta}(Q)=\deg_{\beta}(H)=v_p(a_{\beta})\in \Q\cap [0,1].
\]

\textit{Case 2.  $A$ has semi-stable reduction over $\cO_{K'}$.} Then we have a canonical decomposition $H\simeq \prod_{\gerp\in \SS}H[\gerp]$ of group schemes over $K'$. Recall that $A[\gerp]$ has a unique maximal $\cO_L$-subgroup, denoted by $A[\gerp]^{\mu}$, which  extends to a group scheme over $\cO_{K'}$ of multiplicative type.   We put, for all $\beta\in \BB_{\gerp}$,
\[
\deg_{\beta}(Q)=\begin{cases}
1  &\text{if } H[\gerp]=A[\gerp]^{\mu},\\
0& \text{otherwise}.
\end{cases}
\]
 We get therefore  a parametrization of $\tYrig$ by the partial degrees \cite{GK,PS,Ti}:
\[
\vdeg=(\deg_{\beta})_{\beta\in\bB}: \tYrig\ra [0,1]^{\bB}.
\]

For an ideal $\gert$ of $\cO_L$ dividing $(p)$, we have  $\BB=\BB_{\gert}\coprod\BB_{\gert^\ast}$, and a decomposition of rigid analytic spaces
\[\tYrig-\Yrig=\coprod_{(p)\subset \gert\subset \cO_L}(\tYrig-\Yrig)_{\gert},\]
where $(\tYrig-\Yrig)_{\gert}$ consists of points $Q\in \tYrig-\Yrig$ with $\deg_{\beta}(Q)=1$ for $\beta\in \BB_{\gert}$ and $\deg_{\beta}(Q)=0$ for $\beta\in \BB_{\gert^\ast}$.

Using the additivity of partial degrees \cite[3.6]{Ti}, we have
\begin{equation}\label{equ-deg-wp}
\deg_{\beta}(w_{\gerp}(Q))=
\begin{cases}
1-\deg_{\beta}(Q)&\text{if }\beta\in\BB_{\gerp}\\
\deg_{\beta}(Q)&\text{if } \beta\notin \BB_{\gerp}
\end{cases}
\end{equation}
for any $Q\in \tYrig$ and $\gerp \in \SS$.

 \begin{defn} \label{Definition: admissible domains} Let $Q\in \tYrig$. For any $\gerp \in \SS$, we define $\deg_\gerp(Q)=\sum_{\beta\in\BB_\gerp} \deg_\beta(Q)$. If $\underline{I}$ is a multiset of intervals indexed by $\SS$, $\underline{I}=\{I_\gerp\subset [0,f_\gerp]: \gerp\in\SS\}$, and $\calV \subset \tYrig$, we define an admissible open of $\Yrig$
\[
\calV{\underline I}=\{Q \in \calV: \deg_\gerp(Q) \in I_\gerp, \forall \gerp \in \SS\}.
\]
Also, for an interval $I \subset [0,g]$, define $\tYrig I=\{Q\in \tYrig: \deg(Q)=\sum_{\gerp \in \SS} \deg_\gerp(Q) \in I \}$.

\end{defn}

We now discuss the relation to partial valuations $\nu_\beta(Q)$ defined in \cite[\S 4.2]{GK}, using the mod-$p$ geometry of $\tYrig$.  The valuations and partial degrees are related as follows
\[
\nu_\beta(Q)=1-\deg_\beta(Q).
\]

\begin{rem}\label{Remark: change of notation}
In this paper, we have decided to use the partial degrees which are more intrinsically defined in terms of the subgroup $H$. Since we refer often to the results of \cite{GK}, and \cite{Ka}, the reader should be mindful of the slight change in the notation. In those references, intervals are formed using directional valuations rather than directional degrees. The two notions are simply related by interchanging $\underline{I}$ with $\underline{I}^w$ obtained by replacing any interval $I_\gerp=[a,b]$, with $I_\gerp^{w_\gerp}:=[f_\gerp-b,f_\gerp-a]$ (similarly, for open, half-open intervals).
\end{rem}

\subsection{ Mod-$p$ geometry: stratifications}\label{subsect-Wx}

Recall first Goren-Oort's stratification on $\Xbar$ defined in \cite{GO}. For each $\beta\in \BB$, let $h_{\beta}\in \Gamma(\Xbar, \omega_{\sigma^{-1}\circ\beta}^{\otimes p}\otimes\omega_{\beta}^{-1})$ be the $\beta$-th partial Hasse invariant, and $Z_{\beta}\subseteq \Xbar$ be the closed subscheme defined by the vanishing of $h_{\beta}$. For any subset $\tau\subseteq \BB$, we put $Z_{\tau}=\bigcap_{\beta\in \tau}Z_{\beta}$, and $W_{\tau}=Z_{\tau}-\bigcup_{\tau\subset \tau'}Z_{\tau'}$. Each $W_{\tau}$ is a locally closed reduced subvariety of $\Xbar$, equidimensional of codimension $|\tau|$, and  $\{W_{\tau}, \tau\subseteq \BB\}$ form a stratification of $\Xbar$. For any point $\Qbar \in \Xbar$, we put 
\begin{equation}\label{Defn:tau-Qbar}
\tau(\Qbar)=\{\beta\in \BB: h_{\beta}(\Qbar)=0\}.
\end{equation}
Then, we have $\Qbar\in W_{\tau}$ if and only if $\tau(\Qbar)=\tau$.

Similarly, Goren and Kassaei \cite{GK} defined a stratification  on $\Ybar$. Following \emph{loc. cit.}, we say a pair $(\varphi,\eta)$ of subsets of $\BB$ is admissible if $\ell(\varphi^c)\subset \eta$. For each admissible pair, they defined a locally closed subset  $W_{\varphi,\eta}$  equidimensional of codimension $|\varphi|+ |\eta|-g$ in  $\Ybar$. They proved that if $Z_{\varphi,\eta}$ denotes  the closure of $W_{\varphi,\eta}$ in $\Ybar$, then   
\[
Z_{\varphi,\eta}=\bigcup_{(\varphi',\eta')\geq (\varphi,\eta)}W_{\varphi',\eta'},
\]
where $(\varphi',\eta')\geq (\varphi,\eta)$ means that $\varphi'\supset \varphi$ and $\eta'\supset \eta$ \cite[2.5.1]{GK}.
Then, 
\[
\{W_{\varphi,\eta}: (\varphi,\eta) \text{ admissible}\}
\]
form a stratification of $\Ybar$ \cite[2.5.2]{GK}. By [\emph{loc. cit.} 2.6.4], we have 
\begin{equation}\label{E:relation-strat}
\pi(W_{\varphi,\eta})=\bigcup_{\substack{\varphi\cap \eta\subset \tau\\
\tau\subset (\varphi\cap\eta)\cup (\varphi^c\cap \eta^c)}}W_{\tau}.
\end{equation}

The relationship between Goren-Kassaei's stratification and the directional degrees can be explained as follows (\cite[Thm 4.3.1]{GK}). Let $\spe: \Yrig\ra \Ybar$ denote the specialization map.  
For a rigid point $Q\in \Yrig$, we put 
\begin{gather}\label{eq: deg vs pe}
\varphi(\Qbar)=\{\beta\in \BB, \deg_{\sigma^{-1}\circ\beta}(Q)\in (0,1]\},\quad
\eta(\Qbar)=\{\beta\in \BB, \deg_{\beta}(Q)\in [0,1)\}.
\end{gather}
The pair $(\varphi(\Qbar),\eta(\Qbar))$ is always admissible, and $\Qbar=\spe(Q)$ lies in $W_{\varphi,\eta}$ if and only if $(\varphi(\Qbar),\eta(\Qbar))=(\varphi,\eta)$. Hence,  the strata $\{W_{\varphi,\eta}\}$ correspond bijectively to the faces of the hypercube $[0,1]^{\BB}$: for any stratum $W_{\varphi,\eta}$, we have 
\[
\spe^{-1}(W_{\varphi,\eta})=\{Q\in \Yrig: (\varphi(\Qbar),\eta(\Qbar))=(\varphi,\eta)\}=\Yrig\bfa,
\]
where $\bfa\subset [0,1]^{\BB} $ is the face consisting of $(x_{\beta})_{\beta\in \BB}\in [0,1]^\BB$ with $x_{\beta}=1$ for $\beta\in \eta^c$, $x_{\beta}=0$ for $\beta\in \ell(\varphi^c)$, and $x_{\beta}\in(0,1)$ for $\beta\in\eta\cap \ell(\varphi)$, and $\Yrig\bfa=\vdeg^{-1}(\bfa)$. Note that the codimension of $W_{\varphi,\eta}$ in $\Ybar$ equals to the dimension of $\bfa$.

Let $T\subset \SS$, and $w_{T}$ be the Atkin-Lehner  automorphism on $\Ybar$ defined earlier. The above description of the stratification in terms of directional degrees shows that\begin{equation}\label{E:action-of-w_T}
w_T(W_{\varphi,\eta})=W_{w_{T}(\varphi,\eta)},
\end{equation}
where  $(\varphi',\eta')=w_T(\varphi,\eta)$ is an admissible pair whose $\gerp$-component is given by
$$
(\varphi'_{\gerp}, \eta'_{\gerp})=
\begin{cases}
(r(\eta_{\gerp}), \ell(\varphi_{\gerp}))&\text{if }\gerp\in T,\\ 
(\varphi_{\gerp}, \eta_{\gerp}) &\text{if } \gerp \notin T.\end{cases}
$$

\begin{defn}\label{Defn:generic-part}
For a stratum $W_{\varphi,\eta}$ of $\Ybar$, the generic part of  $W_{\varphi,\eta}$ is defined to be 
\[
W_{\varphi,\eta}^{\rm gen}=\pi^{-1}(W_{\varphi\cap \eta})\cap W_{\varphi,\eta}.
\]


\end{defn}

\begin{lemma} \label{Lemma: divisors} For each admissible pair $(\varphi,\eta)$, $W_{\varphi,\eta}^{\rm gen}$ is an open dense subset of $W_{\varphi,\eta}$. Moreover, if $\varphi^c\cap \eta^c\neq \emptyset$, the complement of $W_{\varphi,\eta}^{\rm gen}$ in $W_{\varphi,\eta}$ is a divisor. 
\end{lemma}

\begin{proof}
By \eqref{E:relation-strat}, we have $W_{\varphi,\eta}^{\rm gen}=\{\Qbar\in W_{\varphi,\eta}: h_{\beta}(\Qbar)\neq 0, \forall \beta\in \varphi^c\cap \eta^c\}$. Therefore, if $\varphi^c \cap \eta^c \neq \emptyset$, the  complement of $W^{\rm gen}_{\varphi,\eta}$ is the locus where at least one of the sections $\{\prtoX^\ast(h_\beta)| \beta \in \varphi^c \cap \eta^c\}$ vanishes. 
\end{proof}


Consider now codimension $0$ strata of $\Ybar$. They correspond to vertex points of $[0,1]^\BB$. For a vertex point $\vx=(x_{\beta})\in \{0,1\}^\BB$, we denote by $\Wx=W_{\phix,\etax}$ the stratum of codimension $0$ attached to $\vx$, where 
\[
\phix=\{\beta\in \bB: \vx_{\sigma^{-1}\circ\beta}=1\},\quad 
\etax=\{\beta\in \bB: \vx_\beta=0\}.
\]
 Then, by \cite[4.3.1]{GK},  we have
\begin{equation}\label{equ-sp-inverse}
\vdeg^{-1}(\vx)=
\begin{cases}
\spe^{-1}(W_{\vx})\cup(\tYrig-\Yrig)_{\gert}&\text{if $\vx=\vx_{\gert}$ for some $\gert|(p)$;}\\
 \spe^{-1}(\Wx) &\text{otherwise.}
\end{cases}
\end{equation}
Here,  $\vx_{\gert}$ is the vertex point whose $\beta$-th component is $1$ if $\beta\in \BB_{\gert}$, and $0$ otherwise.  For a subset $T\subset \SS$, let $w_T(\vx)$ be the vertex point whose $\beta$-component equals to $1-x_{\beta}$ for $\beta\in \cup_{\gerp\in T}\BB_{\gerp}$, and to $x_{\beta}$ otherwise. It is clear that $w_{T}(\Wx)=W_{w_T(\vx)}$.


\begin{defn}
Let $\vx \in [0,1]^\BB$, and $t \subset \SS$. We say $\vx$ is \emph{$T$-ordinary}, if for every prime ideal $\gerp \in T$,  we have either $\vx_\beta=1$ for all $\beta\in\BB_\gerp$,  or $\vx_\beta=0$ for all $\beta\in\BB_\gerp$.
\end{defn}










\begin{prop}\label{Proposition: codim 2} 
Let $\vx$ be a vertex point of $[0,1]^{\BB}$, and $W_{\vx}$ the corresponding stratum of $\Ybar$.  Let $T$ be the set of all primes $\gerp \in \SS$ such that $\vx$ is not $\gerp$-ordinary. Then, the open dense subset 
\[
W_{\vx}^{\rm gen}\cup w_T^{-1}(W^{\rm gen}_{w_T(\vx)})
\]
of  $W_{\vx}$ has a complement of codimension $2$.
\end{prop}

\begin{proof} Let $\Wx^{\rm sp}=\Wx-\Wx^{\rm gen}$. Then, the complement of $W_{\vx}^{\rm gen}\cup w_{T}^{-1}(W^{\rm gen}_{w_{T}(\vx)})$ in $\Wx$ equals
\[
\Wx^{\rm sp} \cap w_T^{-1}(W^{\rm sp}_{w_{T}(\vx)}).
\]
By Lemma \ref{Lemma: divisors}, the closed subschemes $\Wx^{\rm sp}$ and $W^{\rm sp}_{w_T(\vx)}$ of $\Wx$ are divisors. To prove  the claim, it is enough to show that no irreducible component of $\Wx^{\rm sp}$ coincides with an irreducible component of $w_T^{-1}(W^{\rm sp}_{w_T(\vx)})$.

Let $Z_{\vx}$ be the closure of $\Wx$ in  $\Ybar$. In the notation of  \cite[2.5.1]{GK}, we have
\[
Z_{\vx}=Z_{\phix,\etax}=\bigcup_{(\varphi,\eta)\geq (\phix,\etax)}W_{\varphi,\eta}.
\]
 From the proof of Lemma \ref{Lemma: divisors}, the Zariski closure of $\Wx^{\rm sp}$ in $Z_{\vx}$  is given by the vanishing of at least one of the sections 
\[
\{ \prtoX^{*}(h_{\beta}): \beta \in \varphi_{\vx}^c \cap \eta_{\vx}^c \}
\]
on $Z_{\vx}$. Similarly, using Definition \ref{Defn:generic-part}, we see that the Zariski closure of $w_T^{-1}(W_{w_\gert(\vx)}^{\rm sp})$ in $Z_{\vx}$  is defined by the vanishing of at least one of the sections 
\[
\{ w_T^\ast\prtoX^\ast(h_{\beta}): \beta \in r_T(\eta_{\vx})^c \cap \ell_T(\varphi_{\vx})^c \}
\]
on $Z_{\vx}$.

\begin{lemma}\label{lem-ss}
Let notation be as in Proposition \ref{Proposition: codim 2}. The closure in $Z_{\vx}$ of each irreducible component of $\Wx^{\rm sp}$ passes through a point in $W_{\bB,\bB}$.
\end{lemma}

We assume this Lemma for the moment, and finish the proof of \ref{Proposition: codim 2} as follows.  Let  $\Qbar$ be a  point in $W_{\bB,\bB}\cap Z_{\vx}$. By Stamm's theorem \cite{St}, we have an isomorphism
\[
\widehat{\cO}_{\Ybar, \Qbar}\simeq \kappa_K[[x_{\beta},y_{\beta}:\beta\in {\bB}]]/(x_\beta y_{\beta})_{\beta\in\bB},
\]
where $x_{\beta},y_{\beta}$ are specifically defined local parameters.
Similarly, we can write
\[
\widehat{\cO}_{\Ybar, w_T(\Qbar)}\simeq \kappa_K[[x^\prime_{\beta},y^\prime_{\beta}:\beta\in {\bB}]]/(x^\prime_\beta y^\prime_{\beta})_{\beta\in\bB},
\]
such that $w_T^\ast (x^\prime_\beta)=y_\beta$, and $w_T^\ast (y^\prime_\beta)=x_\beta$ for all $\beta \in \ell_T(\phi_{\vx})^c\cap r_T(\eta_{\vx})^c (\subset \cup_{\gerp \in T} \BB_\gerp)$, as in \cite[2.7.2]{GK}. It follows from \cite[2.5.2]{GK}(4) that
\[
\widehat{\cO}_{Z_{\vx},\Qbar}=\kappa_K[[z_{\beta}: \beta\in\BB]],
\]
where, $z_{\beta}=x_{\beta}$, if $\beta \not \in \eta_{\vx}$, and $z_\beta=y_\beta$, otherwise. 
 
 By Lemma \ref{lem-ss} above, to complete the proof of \ref{Proposition: codim 2}, we just need to show that for each $\beta \in \varphi_{\vx}^c\cap \eta_{\vx}^c$, and each $\beta^\prime \in r_T(\eta_{\vx})^c \cap \ell_T(\varphi_{\vx})^c$,  the closed subschemes defined by $\prtoX^\ast(h_{\beta})$ and $w_T^\ast\prtoX^\ast(h_{\beta^\prime})$ in $\Spec(\widehat{\cO}_{Z_{\vx},\Qbar})$ have no common irreducible components. 
 
  The Key Lemma 2.8.1 in \cite{GK} implies that for all $\beta \in \phi_{\vx}^c\cap \eta_{\vx}^c$, we have
\begin{eqnarray}\label{Equation: lhs}
\prtoX^*(h_{\beta})=ux_{\beta}+vy^{p}_{\sigma^{-1}\circ\beta}
\end{eqnarray}
in the local ring $\widehat{\cO
}_{\Ybar,\Qbar}$, where $u,v$ are units in $\widehat{\cO}_{\Ybar,\Qbar}$. Specializing to the local ring ${\cO}_{Z_{\vx},\Qbar}$, we obtain
\[
\prtoX^*(h_{\beta})=uz_{\beta}+vz^{p}_{\sigma^{-1}\circ\beta}.
\]
Similarly, for all $\beta^\prime \in \ell_T(\phi_{\vx})^c\cap r_T(\eta_{\vx})^c$,
\[
 \prtoX^*(h_{\beta^\prime})=u^\prime x^\prime_{\beta^\prime}+v^\prime (y^\prime_{\sigma^{-1}\circ{\beta^\prime}})^p,
 \]
 in the local ring $\widehat{\cO
}_{\Ybar,w_\gert(\Qbar)}$, with $u^\prime, v^\prime$  units in $\widehat{\cO}_{\Ybar,w_\gert(\Qbar)}$. This implies that
\begin{eqnarray}
w_T^\ast\prtoX^\ast(h_{\beta^\prime})=w_T^\ast(u^\prime)y_{\beta^\prime}+w_T^\ast(v^\prime)(x_{\sigma^{-1}\circ{\beta^\prime}})^p\\
\label{Equation: rhs} =w_T^\ast(u^\prime)z_{\beta^\prime}+w_T^\ast(v^\prime)(z_{\sigma^{-1}\circ{\beta^\prime}})^p,
\end{eqnarray}
in the local ring ${\cO}_{Z_{\vx},\Qbar}$.
Comparing Equations \ref{Equation: lhs} and  \ref{Equation: rhs}, and noting that 
\[
(\phi_{\vx}^c\cap \eta_{\vx}^c)\cap (\ell_T(\phi_{\vx})^c\cap r_T(\eta_{\vx})^c)=\emptyset,
\]
it is now immediate to see that $\prtoX^*(h_{\beta})$ and $w_T^*(\prtoX^*(h_{\beta^\prime}))$  cut out two  irreducible and distinct divisors in  $\Spec(\widehat{\cO}_{Z_{\vx},\Qbar})$ . Now it remains to prove Lemma \ref{lem-ss}.

\end{proof}

\begin{proof}[Proof of  \ref{lem-ss}]
Let $C$ be an irreducible component of $\Wx^{\rm sp}$, and $\Cb$ be its closure in $Z_{\vx}$.
 Since $Z_{\vx}$ is smooth \cite[2.5.2.]{GK}(4), there exists a unique irreducible component $\Db$ of $Z_{\vx}$ containing $\Cb$. By \cite[2.6.4]{GK}(1), $\prtoX(\Db)$ is an irreducible component of
 \[
 \prtoX(Z_{\vx})=\prtoX(Z_{\phix,\etax})=Z_{\varphi_{\vx}\cap\eta_{\vx}}.
\]
 
 Since $\prtoX$ is proper, $\prtoX(\Cb)$ is an irreducible component of $Z_{\varphi_{\vx}\cap\eta_{\vx}}$. In particular, $\prtoX(\Cb)$ intersects $W_\BB$. 
 On the other hand, since $C \subset \bigcup_{\beta \in \varphi_{\vx}^c \cap \eta_{\vx}^c} Z_{\vx} \cap \prtoX^{-1}(Z_{\{\beta\}})$, there is $\beta_0 \in \varphi_{\vx}^c\cap\eta_{\vx}^c$ such that $\Cb \subset \prtoX^{-1}(Z_{\{\beta_0\}})$. Let
 \[
 E:=\Cb \cap Z_{\BB-\{\beta_0\},\BB-\{\beta_0\}}.
 \]
 Since $\Cb$ is a divisor in $Z_{\varphi_{\vx}\cap\eta_{\vx}}$, and  $Z_{\BB-\{\beta_0\},\BB-\{\beta_0\}}$ is purely 2-dimensional, every connected component of $E$ has dimension one.

By \cite[2.6.16]{GK}, if $\Qbar \in Z_{\BB-\{\beta_0\},\BB-\{\beta_0\}}$, then $\tau(\prtoX(\Qbar)) \supseteq \BB-\{\beta_0\}$. It follows that 
\[
\prtoX(E) \subset W_\BB \cap \prtoX(\Cb),
\]
which is a finite set, say $\{\Pbar_1,\cdots,\Pbar_r\}$. Therefore,
\[
E \subset \bigcup_{i=1}^r \prtoX^{-1}(\Pbar_i) \cap Z_{\BB-\{\beta_0\},\BB-\{\beta_0\}}
\]
By \cite[]{GK}, each $\prtoX^{-1}(\Pbar_i) \cap Z_{\BB-\{\beta_0\},\BB-\{\beta_0\}}$ is homeomorphic to $\AA^1$. Since $E$ is purely one-dimensional, it follows that there is a point $\Pbar \in \{\Pbar_1,\cdots,\Pbar_r\}$, such that $\prtoX^{-1}(\Pbar) \cap Z_{\BB-\{\beta_0\},\BB-\{\beta_0\}}\subset E$.  Let $\Qbar$ be he unique point in 
\[ 
\prtoX^{-1}(\Pbar)-\prtoX^{-1}(\Pbar) \cap Z_{\BB-\{\beta_0\},\BB-\{\beta_0\}}=\prtoX^{-1}(\Pbar) \cap Z_{\BB,\BB}.
\]
Since $E$ is Zariski closed in $\Ybar$, we deduce that $\Qbar \in E \cap W_{\BB,\BB}\subset \Cb \cap W_{\BB,\BB}$.

\end{proof}

\section{Geometric Hilbert modular forms}\label{Section: GHMF} 
For $\vk\in \Z^{\bB}$, the space of geometric Hilbert modular forms of level $\taml(Np)$ and weight $\vk$ is defined as
 \[
 \cM_{\vk}(\taml(Np);K):=H^0(\HMV_K,\omega^{\vk})=H^0(\tHMV_K,\omega^{\vk}),
 \]
where the last equality is because of the classical Koecher principle \cite[4.9]{Ra}. We will denote the subspace of cusp forms by $\calS_{\vk}(\taml(Np);K)$. The space of overconvergent Hilbert modular forms of level $\taml(Np)$ and weight $\vk$ over $K$ is
\[
\cM^{\dagger}_{\vk}(\taml(Np);K)=\varinjlim_{\cV\supset \tHMVAN^{\ord}}H^0(\cV,\omega^{\vk}),
\]
where $\cV$ runs through the strict neighborhoods of $\tHMVAN^{\ord}$ in $\tHMVAN$.  We will denote the subspace of cusp forms by $\calS^\dagger_{\vk}(\taml(Np);K)$.

 We have a  natural injection $\cM_{\vk}(\taml(Np);K)\hra \cM^{\dagger}_{\vk}(\taml(Np);K)$ sending cusp forms to cusp forms. Both the source and target of this injection are equipped with an action of Hecke operators, and the injection is compatible with the Hecke action. The overconvergent modular forms in the image of $\cM_{\vk}(\taml(Np);K)$ are called {\it{classical}}.

 We also define the space of geometric Hilbert modular forms of level $\Gamma_1(N) \cap \Gamma_0(p)$ and weight $\vk$
 \[
\calM_{\underline{k}}(\Gamma_1(N) \cap \Gamma_0(p);K):=H^0(Y_K,\omega^{\underline{k}}),
 \]
with the subspace of cusp forms denoted by $\calS_{\underline{k}}(\Gamma_1(N) \cap \Gamma_0(p);K)$. Similarly, one can define the space of overconvergent Hilbert modular forms of level $\Gamma_1(N) \cap \Gamma_0(p)$ and weight $\vk$ as
\[
\cM^{\dagger}_{\vk}(\Gamma_1(N) \cap \Gamma_0(p);K)=\varinjlim_{\cV\supset \tYrig^{\ord}}H^0(\cV,\omega^{\vk}),
\]
where $\cV$ runs through the strict neighborhoods of $\tYrig^{\ord}$ in $\tYrig$.  The subspace of cusp forms is denoted by $\calS^\dagger_{\vk}(\Gamma_1(N) \cap \Gamma_0(p);K)$. When $\underline{k}$ corresponds to parallel weight $k$, we replace $\vk$ with $k$ in the above notation.


\begin{rem} \label{Remark: HMF} Geometric Hilbert modular forms are not exactly automorphic forms for $\Res_{L/\QQ} \GL_{2,\QQ}$. In level $\Gamma_1(M)$, these automorphic forms can be identified as the space of geometric Hilbert modular forms invariant under the natural action of $\ol^{\times,+}/\calO_{L,1,M}^2$ (induced by a natural action on $\uA=(\uA,\lambda,i_M)$ via the polarization $\lambda$). Here, $\ol^{\times,+}$ denotes the group of totally positive units of $L$, and $\calO_{L,1,M}^2$ the group of units congruent to $1$ mod $M$. See \cite[Remark 1.11.8]{KL} for more details.

In the second part of this paper, where applications to the Artin conjecture are discussed, we will be working with spaces of automorphic forms for $\Res_{L/\QQ} \GL_{2,\QQ}$. The results of the first part of the paper will be applied via the above identification.

\end{rem}
\subsection{The operators $U_T$.} Let $T \subset \SS$, and set $\gert:=\prod_{\gerp \in T} \gerp$.  We will describe the action of the  $U_T$-operators. Let $\HMV^{T}_{K}$ be the scheme which classifies triples $(\uA,P,D)$ over $K$-schemes, where $(\uA,P)$ is classified by $\HMV_K$,  and $D\subset A[\gert]$ is a closed finite flat $\cO_L$-subgroup which is an $\cO_L/\gert$ module, free of rank $1$  as abelian fppf-sheaves, and such that $(P)\cap D=0$. Let $\tHMV^T_K$ denote the toroidal compactification of $\HMV^T_K$, obtained using our fixed choice of collection of cone decompositions. Let $\tHMVAN^T$ denote the associated rigid analytic variety.  There are  two finite flat morphisms
\begin{eqnarray}
\pi_{1,T},\pi_{2,T}:\tHMVAN^T \arr \tHMVAN, \label{Equation: pi one and two}
\end{eqnarray}
defined by
$\pi_{1,T}(\uA,P,D)=(\uA,P)$ and $\pi_{2,T}(\uA,P,D)=(\uA/D,\overline{P})$ on the non-cuspidal part, where $\overline{P}$ is the image of $P$ in $A/D$.

We have the same setup over $Y$, i.e., for every $T \subset \SS$, a rigid analytic variety $\tYrig^T$ defined over $K$, and two finite flat morphisms
\begin{equation}\label{Equ:Hecke-Y}
\pi_{1,T},\pi_{2,T}:\tYrig^T \arr \tYrig,
\end{equation}
defined similarly. See \cite[\S 4]{Ka} for details. When $T=\{\gerp\}$ is a singleton, we write $\tHMVAN^{\gerp}$ for $\tHMVAN^T$, and $\tYrig^{\gerp}$ for $\tYrig^T$, and denote $\pi_{1,T}, \pi_{2,T}$, respectively, with $\pi_{1,\gerp},\pi_{2,\gerp}$. If $T=\BB$, we use $\tHMVAN^{(p)}$ for $\tHMVAN^T$, and $\tYrig^{(p)}$ for $\tYrig^T$, and denote $\pi_{1,T}, \pi_{2,T}$, respectively, with $\pi_{1,p},\pi_{2,p}$.


Recall  the standard construction of overconvergent correspondences explained in Definition 2.19 of \cite{KaShimura}: Let $T \subset \SS$, and $\calV_1$ and $\calV_2$ be admisible opens of $\tHMVAN$ such that $\pi_{1,T}^{-1}({\calV_1})\subset \pi_{2,T}^{-1}(\calV_2)$. We define an operator
\[
U_T: \omega^{\vk} (\calV_2)\ra \omega^{\vk}(\calV_1),
\]
via the formula
\begin{equation}\label{Equ:U_p}
U_T(f)=\frac{1}{p^{\sum_{\gerp \in T}f_\gerp}} (\pi_{1,T})_{*}(res({\rm pr}^*\pi_{2,T}^*(f))),
\end{equation}
where $res$ is restriction from $\pi_{2,T}^{-1}(\calV_2)$ to $\pi_{1,T}^{-1}({\calV_1})$,  $(\pi_{1,T})_{*}$ is the trace map associated with the finite flat map $\pi_{1,T}$, and ${\rm pr}^*: \pi_{2,T}^* \omega^{\underline{k}} \arr \pi_{1,T}^* \omega^{\underline{k}} $ is a morphism of sheaves on $\tHMVAN^T$, which at $(\uA,P,D)$ is induced by ${\rm pr}^*: \Omega_{A/D} \arr \Omega_A$ coming from the natural projection ${\rm pr}: A \arr A/D$. If $f$ is a bounded section of $\omegab^{\vk}$ over $\calV_2$, then $U_T(f)$ is also bounded over $\calV_1$, and we have the estimation
\begin{equation}\label{Equ:estimation of U_p}
|U_T(f)|_{\calV_{1}}\leq p^{\sum_{\gerp\in T}f_{\gerp}}|f|_{\calV_2}.
\end{equation}
When $T=\{\gerp\}$ is a singleton, we denote $U_T$ by $U_\gerp$.

For $0<r<1$, let $\tHMVAN[g-r,g]=\prtoY^{-1}(\tYrig[g-r,g])$ (Definition \ref{Definition: admissible domains}). Then, as $r\in\QQ$ goes to $0$, the $\tHMVAN[g-r,g]$ form a basis of strict neighborhoods of $\tHMVAN^{\ord}$.  By \cite[\S 5]{Ka}, and considering Remark \ref{Remark: change of notation}, for $r$ close to $0$, there exists a $0<r'<r$ such that 
\[
\pi_{1,T}^{-1}(\tHMVAN[g-r,g]) \subset \pi_{2,T}^{-1}(\tHMVAN[g-r',g]). 
\]
Taking sections of $\omegab^{\vk}$, we get a completely continuous operator $U_T$ on $\cM_{\vk}^{\dagger}(\taml(Np);K)$ by the formula \eqref{Equ:U_p}. It is immediate to see that the operators $U_T$ for $T \subset \SS$ commute with each other, and $U_T$ is the composite of  all the $U_{\gerp}$'s with $\gerp\in T$. We put also $U_p=U_{\SS}$, the composition of all the $U_\gerp$'s.

\subsection{The operators $w_T$.} \label{Subsection: U} For any $\calU \subset \tHMVAN$, and any $T \subset \SS$, define
\[
w_T:\omega^{\underline{k}}(\calU) \arr \omega^{\underline{k}}(w_T^{-1}(\calU)),
\]
by $w_T(f)={\rm pr}^* w_T^* (f)$. As usual, we set $w_\gerp:=w_{\{\gerp\}}$, and $w=w_\SS$.  The $w_\gerp$ operators commute with each other and the operator $w_T$ is the composite of all $w_\gerp$ for $\gerp \in T$.

\subsection{Principal diamond operators}\label{Subsection: geometric diamond} Let $\gerp \in \SS$. For any point $P$ of $A[p]$,  we can write $P=P^\gerp \times P_\gerp$, where $P^\gerp \in A[p/\gerp]$ and $P_\gerp \in A[\gerp]$. Let $f\in \cM^{\dagger}_{\vk}(\taml(Np);K)$. We say that $\psi_{p,f}: (\calO_L/\gerp)^\times \ra \CC_p^\times$ is the {\it $(\calO_L/\gerp)^\times$-character} of $f$, if, for all $j \in (\calO_L/\gerp)^\times$, we have 
\[
f(\uA,P^\gerp,jP_\gerp)=\psi_{\gerp,f}(j)f(\uA,P).
\]
The $(\calO_L/p)^\times$-character of $f$ is defined to be the character of $(\calO_L/p)^\times$ which is the product of $(\calO_L/\gerp)^\times$-characters of $f$ for all $\gerp|p$.

\section{Statement of the main result}

In what follows, the  Artin reciprocity map is normalized so that any uniformizer is sent to an arithmetic Frobenius. The following is the main classicality result of this work.
  
\begin{thm}\label{Theorem: classical} Let $N \geq 4$ be an integer not divisible by  a prime $p\neq 2$. Consider a collection of elements of  $\cM^{\dagger}_k(\taml(Np);K)$ of parallel weight $k$,
\[
\{f_T: T \subset \SS\},
\]
which are all Hecke eigenforms. Let  $\psi_{\gerp,T}$ denote the $(\calO_L/\gerp)^\times$-character of $f_T$ at a prime $\gerp \in \SS$. Assume we are  given a collection of Hecke characters of finite order for $L$,
\[
\{\chi_{T}: T \subseteq \SS\}.
\]
Assume that the following conditions hold:

\begin{enumerate}
\item Fix $\gerp \in \SS$. Then, either for all $\gerp \not\in T \subset \SS$, the conductor of $\chi_{T\cup\{\gerp\}}/\chi_T$ is $\gerp\gern_T$ for some $\gern_T |N$, or for all $\gerp \not\in T \subset \SS$, $\chi_{T\cup\{\gerp\}}/\chi_T=1$. We say that the collection of characters is ramified at $\gerp$ in the first case and unramified at $\gerp$ in the second case. Furthermore, in all cases, we have 
\[
\psi_{\gerp,T}^{-1}=\psi_{\gerp,T\cup\{\gerp\}}=(\chi_{T\cup\{\gerp\}}/\chi_T)|_{\calO_\gerp^\times} \mod 1+\gerp\calO_\gerp,
\]
where in the last equality $\chi_{T\cup\{\gerp\}}/\chi_T$ is considered as a character of $\AA_L^\times/L^\times$;
\item each $f_T$ is normalized;
\item  if $T \subset \SS$, and $\gerp \not\in T$, then, for any ideal $\germ\subset \calO_L$ prime to $\gerp N$, we have: 
\[
c(f_T,\germ)=(\chi_{T\cup\{\gerp\}}/\chi_T)(\germ) c(f_{T\cup\{\gerp\}},\germ);
\]
\item $c(f_T,\gerq)=0$ for all prime ideals $\gerq|N$.
\item $c(f_T,\gerp)\neq 0$ for all $T \subset \SS$;
\item  If $\gerp \not \in T \subset \SS$ and $\chi_{T\cup\{\gerp\}}/\chi_T=1$, then  $c(f_T,\gerp) \neq c(f_{T\cup\{\gerp\}},\gerp)$.

\end{enumerate}
Then, all $f_T$'s are classical.
\end{thm}

The rest of this section will be devoted to the proof of this Theorem,  which will follow directly from Theorems \ref{Theorem: gluing} and \ref{thm-ac}.

\section{Analytic continuation}\label{section: analytic continuation}

\subsection{Admissible domains of $\tHMVAN$} For a multiset of intervals $\underline{I}$ as in Definition \ref{Definition: admissible domains}, set
\[
\tHMVAN \underline{I}=\prtoY^{-1}(\tYrig \underline{I}).
\]
For any admissible open subset $\calW \subset \tHMVAN$, define $\calW\underline{I}=\calW \cap \tHMVAN \underline{I}$.



\subsection{Canonical locus of Goren-Kassaei} For $\gerp \in \SS$, we put 
\[
\cV_{\gerp}=\begin{cases}
\{Q\in \tYrig: p\deg_{\sigma^{-1}\circ\beta}(Q)+\deg_{\beta}(Q)>1\;\text{for } \beta\in \BB_{\gerp}\}&\text{if }f_{\gerp}>1,\\
\{Q\in \tYrig: \deg_{\beta}(Q)>0\;\text{for }\beta\in \BB_{\gerp}\} &\text{if }f_{\gerp}=1.
\end{cases}
\]
Following \cite{GK}, we put 
\[
\cV_{\can}=\bigcap_{\gerp\in \SS}\cV_{\gerp}.
\]
This is a slight enlargement of the canonical locus defined by Goren-Kassaei in \cite{GK}. Note that the definition here of $\cV_{\can}$ is slightly different from that in \emph{loc. cit.} only at primes $\gerp$ with $f_\gerp=1$.

\begin{lemma}[Goren-Kassaei]\label{Lemma:Canonical} 
Let $f\in \calM^{\dagger}_{\vk}(\Gamma_1(Np);K)$ such that $U_{\gerp}(f)=a_{\gerp}f$ with $a_{\gerp}\in \C_p^{\times}$ for each $\gerp|p$. Then $f$  extends analytically to a section  section of $\omegab^{\vk}$  over $\prtoY^{-1}(\cV_{\can})$.
\end{lemma}
This lemma is essentially proved in \cite{GK}. To extend $f$ along the direction of primes of degree $1$, one uses Lubin-Katz's classical theory of canonical subgroups for one-dimensional formal groups. For the extension of $f$ along other directions,   one can find a detailed proof  for forms of level $\Gamma_0(p)\cap \Gamma_1(N)$ in \cite[Prop. 4.14]{Ti}.The same arguments work in the  $\Gamma_1(Np)$-case with $\cV_{\can}$ replaced by $\prtoY^{-1}(\cV_{\can})$.

\begin{defn}\label{subsection:no-etale} 
Let $W_{\varphi, \eta}$ be a  Goren-Kassaei stratum of $\Ybar$. For $\fp\in \SS$, we say that $W_{\varphi,\eta}$ is \emph{not   \'etale at $\fp$} if $(\varphi_{\gerp},\eta_{\gerp})\neq (\emptyset, \BB_{\gerp})$; equivalently, for every $\Qbar=(\uA, H)\in W_{\varphi,\eta}$, the $\gerp$-component $H[\gerp]$ of the finite group scheme $H$ is not  \'etale. We say that $W_{\varphi,\eta}$ \emph{is nowhere \'etale}, if $W_{\varphi,\eta}$ is not \'etale at any $\gerp\in \SS$. 
\end{defn}

The nowhere \'etale strata have a characterization in terms of partial degrees:
\emph{a stratum $W_{\varphi,\eta}$ is nowhere \'etale if and only if  we have $\deg_{\gerp}(Q)=\sum_{\beta\in \BB_{\gerp}}\deg_{\beta}(Q)>0$, for every $Q\in \spe^{-1}(W_{\varphi,\eta})$ and every $\gerp\in\SS$.}

\begin{defn}\label{Defn:goodness}
Let $W_{\varphi,\eta}$ be a stratum  of codimension $1$. Then, $\ell(\varphi)\cap \eta=\{\beta_0\}$,  where $\beta_0$ is the unique element of $\BB$ such that $\deg_{\beta_0}(Q)\in (0,1)$ for one (hence, for all) rigid point $Q\in  \spe^{-1}(W_{\varphi,\eta})$. We say that $W_{\varphi,\eta}$ is \emph{bad}, if $\deg_{\sigma\circ\beta_0}(Q)=0$ for one (hence, for all) $Q\in  \spe^{-1}(W_{\varphi,\eta})$; otherwise, we say that $W_{\varphi,\eta}$ is \emph{good}. In particular, if $\sigma\circ\beta_0=\beta_0$, i.e. the prime $\gerp_0\in \SS$ with $\beta\in \BB_{\gerp_0}$  has degree $1$, then $W_{\varphi,\eta}$ is good.
\end{defn}

\begin{defn}\label{Defn:Sigma} 
(i) Let $W=W_{\varphi,\eta}$ be a  nowhere \'etale stratum  of codimension $0$ or $1$. We will define an admissible open subset $]W^{\gen}['$ of $\tYrig$ as follows:
\begin{itemize}
\item[(Case 1)] $\codim(W)=0$. We put $]W^\gen['=\spe^{-1}(W^\gen)$, where $W^{\gen}$ is the generic part of $W$ introduced in Definition \ref{Defn:generic-part}.

\item[(Case 2)] $\codim(W)=1$. Let $\beta_0$ be as in Definition \ref{Defn:goodness}, $\gerp_0\in \SS$ be  such that $\beta_0\in \BB_{\gerp_0}$, and $\eta_{\gerp_0}=\eta\cap\BB_{\gerp_0}$.

\begin{itemize}
\item[(Case 2a)] $W$ is good. In this case,
 we put $]W^\gen['=\spe^{-1}(W^\gen)$.

\item[(Case 2b)] $W$ is bad  and $\eta_{\gerp_0}=\BB_{\gerp_0}$, or, equivalently, for any rigid point $Q$  in $\spe^{-1}( W)$, we have $\deg_{\beta}(Q)=0$ for all $\beta\in \BB_{\gerp_0}-\{\beta_0\}$. In this case, we have $W=W^\gen$, and we set 
$$
]W^\gen['=
\{Q\in \spe^{-1}(W): \deg_{\beta_0}(Q)\in(\sum_{i=1}^{f_{\gerp_0}-1}\frac{1}{p^i},1)\}.
$$ 

\item[(Case 2c)] $W$ is bad and  $\eta_{\gerp_0}\neq \BB_{\gerp_0}$, or, equivalently, there exists an integer $1\leq j\leq f_{\gerp_0}$ such that for any rigid point $Q=(\uA,H)\in \spe^{-1}(W)$, we have $\deg_{\sigma^i\circ\beta_0}(Q)=0$ for $1\leq i\leq j$ and $\deg_{\sigma^{j+1}\circ\beta_0}(Q)=1$. We set $\delta_j:=\sum_{i=1}^{j}1/p^i$, and define
\[
]W^{\gen}['=\{Q\in \spe^{-1}(W^\gen):\deg_{\beta_0}(D)\leq \delta_{j},\quad \forall (\uA,H,D)\in \pi_{1,p}^{-1}(Q)\},
\]
where $\pi_{1, p}\colon \tYrig^{(p)}\ra \tYrig$ is defined in \eqref{Equ:Hecke-Y}.
\end{itemize}
\end{itemize}

(ii) We define an admissible domain of $\tYrig$ as follows:
\begin{equation}\label{Equ:defn-sigma}
\Sigma=\bigcup_{\codim(W)=0,1}]W^\gen[',
\end{equation}
where $W$ runs though all the nowhere  \'etale strata of codimension $0$ and $1$.
\end{defn}

\begin{rem}
The definition of $]W^\gen['$ in Case 2(c) will be justified by Lemma~\ref{L:BK-module} below.
\end{rem}

We can now state the main result of this section.

\begin{thm}\label{Thm:AnaCont} 
Let $f\in \calM^{\dagger}_{\vk}(\Gamma_1(Np);K)$ be a simultaneous eigenform for the operators $\{U_{\gerp}: \gerp\in \BB\}$ with non-zero eigenvalues. Then, $f$ extends analytically to a  section of $\omegab^{\vk}$ over $\region$.
\end{thm}

This theorem is an immediate consequence of the following

\begin{prop}\label{Prop:Sigma-U_p} 
 The image of $\Sigma$ under the Hecke correspondence $U_p$ is contained in the canonical locus $\cV_{\can}$, i.e., we have 
\[\pi_{1,p}^{-1}(\Sigma)\subset \pi_{2,p}^{-1}(\cV_{\can}),\]
where $\pi_{i,p}:\tYrig^{(p)}\ra \tYrig$ for $i=1,2$ are the maps defined in \eqref{Equ:Hecke-Y}. 
\end{prop}

Indeed, let $f$ be a finite slope eigenform of level $\Gamma_1(Np)$ as in Theorem \ref{Thm:AnaCont}, and $\lambda_p\neq 0$ be the eigenvaule of $U_p$. By Lemma \ref{Lemma:Canonical}, $f$ extends to a bounded section of $\omegab^{\vk}$ over $\prtoY^{-1}(\cV_{\can})$.  Proposition \ref{Prop:Sigma-U_p} implies that $f$ extends to $\region$ using the functional equation $f=\frac{1}{\lambda_{p}}U_p(f)$. The boundedness of $f$ over $\region$ follows from the estimation \eqref{Equ:estimation of U_p}.\\

To prove Proposition \ref{Prop:Sigma-U_p}, we need some preparation.

\subsection{Directional Hodge Heights} Let $\uA\in \Xrig$ be a rigid point, i.e. a HBAV defined over the ring of integers $\cO_{K'}$ of a finite extension $K'/K$. We denote by $\uA_0$ its base change over $\cO_{K'}/p$. For each $\beta\in \BB$, let $h_{\beta}\in \Gamma(\Xbar,\omegab_{\sigma^{-1}\circ\beta}^{p}\otimes \omegab_{\beta}^{-1})$ be the $\beta$-th partial Hasse invariant, and $h_{\beta}(\uA_0)$ its evaluation at $\uA_0$. Using a basis of $\omegab_{A/\cO_{K'}}$, we may represent $h_{\beta}(\uA_0)$ as an element of $\cO_{K'}/p$. In \cite[4.6]{Ti}, we defined the $\beta$-th directional Hodge height of $\uA$, denoted by $w_{\beta}(\uA)$, as the truncated $p$-adic valuation of $h_{\beta}(\uA_0)$. This is a well-defined rational number in the interval $[0,1]$. We may extend the definition of $w_{\beta}(\uA)$ to whole $\tXrig$ by setting $w_{\beta}(\uA)=0$ if $\uA$ is a rigid point in $\tXrig-\Xrig$. In \cite[4.2]{GK} the same notion is defined under the name of  {\it valuations on $\Xrig$} and the notation $\nu_\beta$ is used instead of $w_\beta$. The following Lemma will play a crucial role in the proof of Proposition~\ref{Prop:Sigma-U_p}.

\begin{lemma}[Goren-Kassaei]\label{Lemma:partial-Hodge}  Let $Q=(\uA,H)\in \tYrig$ be a rigid point. Then, for each $\beta\in \BB$, we have 
\[
w_{\beta}(\uA)\geq \min\{p\deg_{\sigma^{-1}\circ\beta}(H),1-\deg_{\beta}(H)\}.
\]
Moreover, the equality holds if $p\deg_{\sigma^{-1}\circ\beta}(H)\neq 1-\deg_{\beta}(H)$.  
 \end{lemma}

\begin{proof}
Note that if $\beta\not\in \varphi(\Qbar) \cap \eta(\Qbar)$, then the statement is obvious using (\ref{eq: deg vs pe}). For the remaining cases, the statement  follows directly from Goren-Kassaei's ``Key Lemma'' in \cite{GK}.  For another proof, see also \cite[4.5]{Ti}.
\end{proof}

\begin{proof}[Proof of Proposition~\ref{Prop:Sigma-U_p}]
  Let $\pi_{1,p},\pi_{2,p}:\tYrig^{(p)}\ra \tYrig$ be as in \eqref{Equ:Hecke-Y}. To prove \ref{Prop:Sigma-U_p}, it suffices to show that for every  rigid point $Q=(\uA,H)$ of  $\Sigma\subset \tYrig$, we have $\pi_{2,p}(\pi_{1,p}^{-1})(Q)\subset \cV_{\can}$. Note that
$$
\pi_{1,p}^{-1}(Q)=\{(\uA,H,D): \text{$D\subset A[p]$ is free of rank $1$ over $\cO_L/p$ and $D\cap H\neq 0$}\}.
$$
By definition of $\cV_{\can}$, we have to show that if $(\uA,H,D)\in \pi_{1,p}^{-1}(Q)$, and  $\beta \in \BB_\gerp \subset \BB$ such that $f_\gerp\neq 1$, then we have 
$p\deg_{\beta}(\uA[p]/D)+\deg_{\sigma\circ\beta}(\uA[p]/D)>1$; equivalently, for all $(\uA,H,D)$ and $\beta$ as above, we must show that:
\begin{equation}\label{Equ:can-test}
p\deg_{\beta}(D)+\deg_{\sigma\circ\beta}(D)<p,
\end{equation}
since $\deg_{\beta}(\uA[p]/D)=1-\deg_{\beta}(D)$ for $\beta\in \BB$. 

\begin{lemma}\label{lemma:degree} Let $Q=(\uA,H)$ be a rigid point of $\tYrig$ whose specialization $\Qbar$ is generic in the sense of Definition \ref{Defn:generic-part}. Let $(\uA, H,D)\in \pi_{1,p}^{-1}(Q)$,    $\gerp\in \SS$, and $\beta\in \BB_{\gerp}$.
\begin{itemize}
\item[(1)] We have 
\[\deg_{\beta}(H)+\deg_{\beta}(D)\leq \sum_{i=0}^{f_{\gerp}-1}\frac{1}{p^i}.\]
In particular, 
we have $\deg_{\beta}(D)<1$ if $\deg_{\beta}(H)>\sum_{i=1}^{f_{\gerp}-1}\frac{1}{p^i}$, and $
\deg_{\beta}(D)\leq \sum_{i=1}^{f_{\gerp}-1}\frac{1}{p^i},
$ if $\deg_{\beta}(H)=1$.

\item[(2)] If $\deg_{\beta}(H)=1$, then we have $\deg_{\beta}(D)\leq \sum_{i=1}^{f_{\gerp}-1}1/p^i$ and $\deg_{\sigma^{-1}\circ\beta}(D)=0$.

\item[(3)] If $\deg_{\sigma^{-1}\circ\beta}(H)=0$ and $\deg_{\beta}(D)<1$, then $\deg_{\sigma^{-1}\circ\beta}(D)=0$.
 
\end{itemize}
\end{lemma}
\begin{proof}
 The second part of (1) is clearly a consequence of the first part. Suppose $(\uA,H,D)$ are defined over $\cO_{K'}$, where $K'/K$ is a finite extension. Consider the homomorphism $D[\gerp]\ra A[\gerp]/H[\gerp]$ of finite flat group schemes of $1$-dimensional $\F_{p^{f_{\gerp}}}$-vector spaces, which is generically an isomorphism. By Raynaud's theory, one gets  \cite[3.7]{Ti}
\[\sum_{i=0}^{f_{\gerp}-1}p^{f_{\gerp}-1-i}\deg_{\sigma^{i}\circ\beta}(D)\leq \sum_{i=0}^{f_{\gerp}-1}p^{f_{\gerp}-1-i}\deg_{\sigma^{i}\circ\beta}(A/H).\]
Statement (i) follows  from the fact that $\deg_{\gamma}(A/H)=1-\deg_{\gamma}(H)$ for all $\gamma\in \BB$.

For statement (2), the claim on $\deg_{\beta}(D)$ follows from (1). The assumption on $\deg_{\beta}(H)$  implies that $\beta\not\in  \eta(\Qbar)$ (by (\ref{eq: deg vs pe})), whence $\beta \not\in\varphi(\Qbar)\cap\eta(\Qbar)$. Since the specialization $\Qbar$ of $Q$ is generic, it follows that  $w_{\beta}(A)=0$. Applying Lemma~\ref{Lemma:partial-Hodge} to $(\uA, D)$, we get $\deg_{\sigma^{-1}\circ\beta}(D)=0$ in view of $\deg_{\beta}(D)<1$. 

For (3),  we note that the assumption on $\deg_{\sigma^{-1}\circ\beta}(H)$  implies that $\beta\not\in  \varphi(\Qbar)$. The rest of the argument follows as in part (2).

\end{proof}

We assume that  $Q=(\uA,H)\in ]W^{\gen}['\subseteq \Sigma$, where $W=\Wpe$ is a nowhere \'etale stratum of codimension $0$ or $1$. Let $(\uA,H,D)\in \pi_{1,p}^{-1}(Q)$ be a rigid point of $\tYrig^{(p)}$ above $Q$.
\begin{lemma}\label{lemma:case1}
Under the above notations, let $\gerp\in \SS$ be such that $\deg_{\beta}(H)\in \{0,1\}$, for all $\beta\in \BB_{\gerp}$.  Then,  we have $\deg_{\beta}(D)=0$ for all $\beta\in \BB_{\gerp}$ unless $\deg_{\beta}(H)=1$ and $\deg_{\sigma\circ\beta}(H)=0$, in which case, we have 
$$
\frac{1}{p}\leq \deg_{\beta}(D)\leq \sum_{i=1}^{f_{\gerp}-1}\frac{1}{p^i}.
$$
In particular, \eqref{Equ:can-test} holds for every $\beta\in \BB_{\gerp}$.
\end{lemma}

\begin{proof}
By Lemma~\ref{lemma:degree}(2), if $\deg_{\sigma\circ\beta}(H)=1$, then we have $\deg_\beta(D)=0$.  We now prove that $\deg_{\beta}(D)=0$ if $\deg_{\beta}(H)=0$. The  assumption that $W_{\varphi,\eta}$ is nowhere \'etale  implies  that $\deg_{\beta}(H)$ can not be  $0$ for all $\beta\in \BB_{\gerp}$.  There exists an integer $n\geq 1$ such that $\deg_{\sigma^{i}\circ\beta}(H)=0$ for $0\leq i\leq n-1$  and $\deg_{\sigma^{n}\circ\beta}(H)=1$. Lemma~\ref{lemma:degree}(2) implies that $\deg_{\sigma^{n-1}\circ\beta}(D)=0$. If $n=1$, then we are done; if $n\geq 2$, then it follows from Lemma~\ref{lemma:degree}(3) that  $\deg_{\sigma^{i}\circ\beta}(D)=0$ for $0\leq i\leq n-2$.  

It follows from the above that $\deg_\beta(D)=0$ except if $\deg_{\beta}(H)=1$ and $\deg_{\sigma\circ\beta}(H)=0$.   Then,  $\deg_{\beta}(D)\leq \sum_{i=1}^{f_{\gerp}}1/p^i$ follows from Lemma~\ref{lemma:degree}(1). It remains to show that $\deg_{\beta}(D)\geq 1/p$. By Lemma~\ref{Lemma:partial-Hodge},  we get $w_{\sigma\circ\beta}(A)=1$. Applying the same lemma to $(\uA, D)$, we obtain
\[
1=w_{\sigma\circ\beta}(A)\geq\min\{1-\deg_{\sigma\circ\beta}(D), p\deg_{\beta}(D)\}.
\]
Since  $\deg_{\sigma\circ\beta}(D)=0$ by the first part of the proof, it follows that $\deg_{\beta}(D)\geq 1/p$. 

 \end{proof}

We now come back to the proof of \ref{Prop:Sigma-U_p}.
 By the discussion above,  it suffices to check \eqref{Equ:can-test} for all $\beta\in \BB$. Assume first that $\codim(W)=0$. In this case, Lemma~\ref{lemma:case1} applies to every prime $\gerp\in \SS$. Hence,   \eqref{Equ:can-test} holds for every $\beta\in \BB$.

Assume now  $\codim(W)=1$. Let $\beta_0$ the unique element of $\BB$  with $\deg_{\beta_0}(H)\in (0,1)$, and $\gerp_0\in \SS$ be such that $\beta_0\in \BB_{\gerp_0}$. Let $\gerp\in\SS$. If $\gerp\neq\gerp_0$, then, $\deg_{\beta}(H)\in \{0,1\}$ for $\beta\in \BB_{\gerp}$. By Lemma~\ref{lemma:case1},  the relation \eqref{Equ:can-test} holds for $\beta\in \BB_{\gerp}$. Consider now the case  $\gerp=\gerp_0$. We can assume $f_{\gerp_0}>1$. We distinguish two cases:

(a) Assume $\deg_{\sigma^{-1}\circ\beta_0}(H)=\deg_{\sigma\circ\beta_0}(H)=1 $. We deduce first from Lemma~\ref{lemma:degree} (respectively, parts (1) and (2)) that $\deg_{\sigma^{-1}\circ\beta}(D)<1$, and $\deg_{\beta_0}(D)=0$. This implies that the relation \eqref{Equ:can-test} is satisfied for $\beta=\sigma^{-1}\circ\beta_0, \beta_0$.  For the rest of the $\beta$'s the result follows from an argument exactly as in the proof of Lemma \ref{lemma:case1}.

(b) Assume $\deg_{\sigma^{-1}\circ\beta_0}(H)=0$ and $\deg_{\sigma\circ\beta_0}(H)=1$.  Lemma~\ref{lemma:degree}(2) implies that  $\deg_{\beta_0}(D)=0$. Applying  Lemma~\ref{lemma:degree}(3), we further find that $\deg_{\sigma^{-1}\circ\beta_0}(D)=0$. This gives the relation \eqref{Equ:can-test} for $\beta=\sigma^{-1}\circ\beta_0, \beta_0$.  For the rest of the $\beta$'s, we could use an argument similar to the proof of Lemma \ref{lemma:case1}, if we prove that $\deg_\beta(D)=0$ if $\deg_\beta(H)=0$. One can find an $\ell  \geq 1$, such that $\deg_{\sigma^i \circ \beta}(H)=0$ for $0\leq i \leq \ell-1$, and $\deg_{\sigma^\ell \circ \beta}(H) \neq 0$. If the latter value is $1$, the proof of {\it loc. cit} works. If not, then $\sigma^\ell \circ \beta=\beta_0$, whence $\deg_{\sigma^\ell \circ \beta}(D) =0 \neq 1$. Successive application of Lemma~\ref{lemma:degree}(3) gives the desired result.

 Now, we assume that $W$ is bad, i.e., $\deg_{\sigma\circ\beta_0}(H)=0$. Note that in this case  for any $D$ as above, we have $\deg_{\beta_0}(D)<1$ (by Lemma ~\ref{lemma:degree}(1) in case 2b, and by definition in case 2c). We consider the following two cases:



 (c)  $\deg_{\sigma^{-1}\circ\beta_0}(H)=1$. First note that using the same argument as in the proof of Lemma \ref{lemma:case1}, we can show that $\deg_\beta(D)=0$ whenever $\deg_\beta(H)=0$, and also verify \eqref{Equ:can-test} for $\beta \neq \sigma^{-2}\circ\beta_0,\sigma^{-1}\circ\beta_0,\beta_0$. By Lemma \ref{lemma:degree}(1), we have $\deg_{\sigma^{-2}\circ\beta_0}(D)=0$, which implies \eqref{Equ:can-test} for $\beta=\sigma^{-2}\circ\beta_0$. One also verifies \eqref{Equ:can-test} for $\beta=\beta_0$, since $\deg_{\sigma\circ\beta_0}(D)=0$ (owing to $\deg_{\sigma\circ\beta_0}(H)=0$), and $\deg_{\beta_0}(D)<1$ (by definition of $]W^{\rm gen}['$) . It remains to deal with $\beta=\sigma^{-1}\circ\beta_0$. Using Lemma \ref{Lemma:partial-Hodge}, we have
$$
w_{\beta_0}(A)=1-\deg_{\beta_0}(H)\geq \min\{p\deg_{\sigma^{-1}\circ\beta_0}(D), 1-\deg_{\beta_0}(D)\}.
$$
 The relation \eqref{Equ:can-test} is trivially true if $p\deg_{\sigma^{-1}\circ\beta_0}(D)\leq 1-\deg_{\beta_0}(D)$. Assume then $p\deg_{\sigma^{-1}\circ\beta_0}(D)>1-\deg_{\beta_0}(D)$. In this case, we have $\deg_{\beta_0}(D)=1-w_{\beta_0}(A)=\deg_{\beta_0}(H)$. By Lemma \ref{lemma:degree}(i), we have $\deg_{\beta_0}(D)=\deg_{\beta_0}(H)\leq \frac{1}{2}\sum_{i=0}^{f_{\gerp_0}-1}1/p^{i}$, and  $\deg_{\sigma^{-1}\circ \beta_0}(D)\leq \sum_{i=1}^{f_{\gerp_{0}}-1}1/p^i$. Hence, we get 
\[
p\deg_{\sigma^{-1}\circ\beta_0}(D)+\deg_{\beta_0}(D)<\frac{3}{2}\sum_{i=0}^{f_{\gerp_0}-1}1/p^i<\frac{3p}{2(p-1)}<p.
\]
Note that we used $p\geq 3$ in the last step. This verifies \eqref{Equ:can-test} for $\beta=\sigma^{-1}\circ\beta_0$.

(d) $\deg_{\sigma^{-1}\circ\beta_0}(H)=0$.  Since $\deg_{\beta_0}(D)<1$, the proof of Lemma \ref{lemma:case1}, as modified in part (b) above, allows us to verify  \eqref{Equ:can-test} for $\beta \neq\sigma^{-1}\circ\beta_0, \beta_0$. In the course of doing so, we prove that $\deg_\beta(D)=0$ whenever $\deg_\beta(H)=0$. In particular, we have $\deg_{\sigma^{-1}\circ\beta_0}(D)=\deg_{\sigma\circ\beta_0}(D)=0$. Since  $\deg_{\beta_0}(D)<1$, we conclude that \eqref{Equ:can-test} holds for $\beta=\sigma^{-1}\circ\beta_0, \beta_0$ as well.

The proof of Proposition~\ref{Prop:Sigma-U_p} is now complete.
\end{proof}

The following Lemma will be used to prove certain connectedness results in the next section. It also justifies the definition of $]W^\gen['$ in Case 2c of \ref{Defn:Sigma}. 

\begin{lemma}\label{L:BK-module}
Let $W$ be a bad stratum of codimension $1$ as in Case2(c) of \ref{Defn:Sigma},  $\beta_0\in \BB$, $j\in \Z_{\geq 1}$ and $\delta_j$ be as defined there. Let $Q=(\uA,H)\in \spe^{-1}(W^\gen)$ be a rigid point, $(\uA,H,D)\in \pi_{1,p}^{-1}(Q)$. Then, we have
  \begin{enumerate}
\item $\deg_{\beta_0}(D)=\delta_j$, if $\deg_{\beta_0}(H)>\delta_j$;

\item $\deg_{\beta_0}(D)=\deg_{\beta_0}(H)$, if $\deg_{\beta_0}(H)<\delta_j$;

\item $\deg_{\beta_0}(D)\geq \delta_j$, if $\deg_{\beta_0}(H)=\delta_j$. 
\end{enumerate}
\end{lemma}

\begin{proof}
We will prove this Lemma using Breuil-Kisin modules.  Let $K'/K$ be a finite extension  with ramification index $e$ and residue field $\kappa'$ such that $(\uA,H)$ and $D$ can be defined over $\cO_{K'}$. Let $\gerS_1=\kappa'[[u]]$, and $\varphi: \gerS_1\ra \gerS_{1}$ be the Frobenius endomorphism.

Let $\gerp_0\in \SS$ with $\beta_0\in \BB_{\gerp_0}$. Since  the Lemma concerns only the $p$-divisible group $A[\gerp_0^\infty]$, to simplify the notation, we may assume that $p$ is inert in $F$. Let 
$(\gerM,\varphi)$  be
 the contravariant Breuil-Kisin module of $A[p]$ as considered in \cite[Sectoin 3]{Ti}. Recall that the action of $\cO_L$ on $A[p]$ induces a canonical decomposition $\gerM=\bigoplus_{\beta\in \BB}\gerM_{\beta}$, where $\gerM_{\beta}$ is a free $\gerS_1$-module of rank 2 on which $\cO_L$ acts via $\beta: \cO_L\ra W(\kappa)\ra \kappa'$. The Frobenius semi-linear endomorphism  $\varphi$ on $\gerM$ sends $\gerM_{\sigma^{-1}\circ\beta}$ to $\gerM_{\beta}$ for $\beta\in \BB$. 
By the main theorem of Kisin modules, the $\cO_L$-equivariant quotient $A[p]\twoheadrightarrow A[p]/H$ corresponds  to an $\cO_{L}$-equivariant Kisin submodule $\gerL=\bigoplus_{\beta\in \BB} \gerL_{\beta}\subseteq \gerM$ such that, for each $\beta\in \BB$, both $\gerL_\beta$ and $\gerM_{\beta}/\gerL_{\beta}$ are both free $\gerS_1$-module of rank $1$. For each $\beta$, we choose a basis $(\ee_{\beta}, \ff_{\beta})$
for $\gerM_{\beta}$ over $\gerS_1$  such that $\gerL_\beta$ is generated by $\ee_{\beta}$. Then the endomorphism $\varphi$ is given by   
\[
\varphi(\ee_{\sigma^{-1}\circ\beta}, \ff_{\sigma^{-1}\circ\beta})=(\ee_{\beta},\ff_{\beta})\begin{pmatrix} a_{\beta}&b_{\beta}\\
0&d_{\beta}\end{pmatrix},
\]
with $a_{\beta}, b_{\beta},d_{\beta}\in\gerS_1$. By abuse of notation, we denote  by $v_p$ the valuation on $\gerS_1$ normalized by $v_p(u)=1/e$.  
We list the properties of $a_{\beta}, b_{\beta}, d_{\beta}$ as follows:
\begin{itemize}
\item By \cite[Lemma 3.9]{Ti}, we have $v_{p}(d_{\beta})=\deg_{\beta}(H)$, and $v_p(a_{\beta})=\deg_{\beta}(A[p]/H)=1-\deg_{\beta}(H)$.  
\item If $\deg_{\beta}(H)=0$, i.e. $d_{\beta}\in \gerS_1^\times$, then we may assume $b_{\beta}=0$ and $d_{\beta}=1$ up to replacing  $\ff_{\beta}$ by $\varphi(\ff_{\sigma^{-1}\circ\beta})$.

\item If $Q\in \spe^{-1}(W^{\gen})$, then $v_p(b_{\sigma^{j+1}\circ\beta_0})=0$ by \cite[Lemma~3.12]{Ti}.

\item The fact that $\deg_{\beta_0}(H)\in (0, 1)$ implies that $v_p(b_{\beta_0})=0$, since the cokernel of the linearized map $1\otimes \varphi: \varphi^*(\gerM_{\sigma^{-1}\circ\beta})\ra \gerM_{\beta}$ is free of rank $1$ over $\gerS_1/u^e$ (See \cite[3.10]{Ti}).

\end{itemize}

 Let $D\subseteq A[p]$ be an $\cO_F$-stable subgroup scheme  such that $(\uA,H,D)\in \pi^{-1}_{1,p}(\uA,H)$, and $\gerN=\bigoplus_{\beta\in \BB}\gerN_\beta\subseteq \gerM$ be the Kisin submodule corresponding to $A[p]/D$. Assume that $\gerN_{\sigma^{-1}\circ\beta_0}=\gerS_1(\ee_{\sigma^{-1}\circ\beta_0}+y_{\sigma^{-1}\circ\beta_0}\ff_{\sigma^{-1}\circ\beta_0})$. Since $\gerN_{\sigma^{-1}\circ\beta_0}$ is stable under $\varphi^{g}$, the variable $y_{\sigma^{-1}\circ\beta_0}$ has to satisfy the equation:
\begin{align*}
&\ a_{\sigma^{-1}\circ\beta_0}\cdots a_{\sigma^{-g+1}\circ\beta_0}^{p^{g-2}}a_{\beta_0}^{p^{g-1}}+\bigl(\sum_{i=0}^{g-1}a_{\sigma^{-1}\circ\beta_0}\cdots a_{\sigma^{-i}\circ\beta_0}^{p^{i-1}}b_{\sigma^{-i-1}\circ\beta_0}^{p^i}d_{\sigma^{-i-2}\circ\beta_0}^{p^{i+1}}\cdots d_{\sigma^{-g}\circ\beta_0}^{p^{g-1}}\bigr)y^{p^{g}}_{\sigma^{-1}\circ\beta_0}\\
&=d_{\sigma^{-1}\circ\beta_0}\cdots d_{\sigma^{-g+1}\circ\beta_0}^{p^{g-2}}d_{\beta_0}^{g-1} y_{\sigma^{-1}\circ\beta_0}^{p^{g}-1}.
\end{align*}
Note that the constant term has valution 
\[
v_p(a_{\sigma^{-1}\circ\beta_0}\cdots a_{\sigma^{-g+1}\circ\beta_0}^{p^{g-2}}a_{\beta_0}^{p^{g-1}})
=\sum_{i=1}^gp^{i-1}(1-\deg_{\sigma^{-i}\circ\beta_0}(H)),
\]
and the coeffecient of $y_{\sigma^{-1}\circ\beta_0}^{p^g-1}$ has valuation 
\[
v_p(d_{\sigma^{-1}}\cdots d_{\sigma^{-g+1}\circ\beta_0}^{p^{g-2}}d_{\beta_0}^{g-1})=\sum_{i=1}^g p^{i-1}\deg_{\sigma^{-i}\circ\beta_0}(H)
\]
Let $c_{p^g}=\sum_{i=0}^{g-1}a_{\sigma^{-1}\circ\beta_0}\cdots a_{\sigma^{-i}\circ\beta_0}^{p^{i-1}}b_{\sigma^{-i-1}\circ\beta_0}^{p^i}d_{\sigma^{-i-2}\circ\beta_0}^{p^{i+1}}\cdots d_{\sigma^{g}\circ\beta_0}^{p^{g-1}}
$ denote the coefficient of $y^{p^{g}}_{\sigma^{-1}\circ\beta_0}$, $j$ and $\delta_j$  be as defined in the Lemma. We distinguish three cases:

(1) If $\deg_{\beta_0}(H)>\delta_j$, we have 
\begin{align*}
&v_p(c_{p^g})=v_p(a_{\sigma^{-1}\circ\beta_0}\cdots a_{\sigma^{-g+1}\circ\beta_0}^{p^{g-2}}b_{\beta_0}^{p^{g-1}})\\
&=\sum_{i=1}^{g-1}p^{i-1}(1-\deg_{\sigma^{-i}\circ\beta_0}(H))
\end{align*}
By the method of Newton polygon, all the $p^{g}$-solutions of $y_{\sigma^{-1}\circ\beta_0}$ have  valuation 
\[
v_{p}(y_{\sigma^{-1}\circ\beta_0})=\frac{1}{p}(1-\deg_{\beta_0}(H)).
\]
Since $\varphi(\ee_{\sigma^{-1}\circ\beta_{0}}+y_{\sigma^{-1}\circ\beta_0}\ff_{\sigma^{-1}\circ\beta_0})
=(a_{\beta_0}+b_{\beta_0}y_{\sigma^{-1}\circ\beta_0}^p)\ee_{\beta_0}
+y_{\sigma^{-1}\circ\beta_0}^pd_{\beta_0}\ff_{\beta_0}$, a detailed computation shows that 
\begin{align*}
&\quad\min\{v_p(a_{\beta_0}+b_{\beta_0}y_{\sigma^{-1}\circ\beta_0}^p), v_p(y_{\sigma^{-1}\circ\beta_0}^pd_{\beta_0})\}\\
&=v_p(a_{\beta_0}+b_{\beta_0}y_{\sigma^{-1}\circ\beta_0}^p)\\
&=1-\delta_j.
\end{align*}
Since $\gerM_{\beta_0}/\gerN_{\beta_0}$ is free of rank $1$ over $\gerS_1$, $\gerN_{\beta_0}$ must be generated by $(\ee_{\beta_0}+\frac{y_{\sigma^{-1}\circ\beta_0}^p}{a_{\beta_0}+b_{\beta_0}y_{\sigma^{-1}\circ\beta_0}^p}\ff_{\beta_0})$.
Hence, by \cite[Lemma~3.9]{Ti}, we have  
\begin{align*}
\deg_{\beta_0}(D)&=1-\deg_{\beta_0}(A[p]/D)=\deg(\gerN_{\beta_0}/(1\otimes \varphi)\varphi^*(\gerN_{\sigma^{-1}\circ\beta_0}))\\
&=1-v_p(a_{\beta_0}+b_{\beta_0}y_{\sigma^{-1}\circ\beta_0}^p)=\delta_j
\end{align*}
for all $D\subseteq A[p]$ with $D\neq H$.

(2) If $\deg_{\beta_0}(H)<\delta_j$, then we have 
\begin{align*}
&v_{u}(c_{p^g})= v_u(a_{\sigma^{-1}\circ\beta_0}\cdots a_{\sigma^{-(g-j-2)}\circ\beta_0}^{p^{g-j-3}}b_{\sigma^{-(g-j-1)}\circ\beta_0}^{p^{g-j-2}}d_{\sigma^{-(g-j)}\circ\beta_0}^{p^{g-j-1}}\cdots  d_{\sigma^{-g}\circ\beta_0}^{p^{g-1}})\\
&=\sum_{i=1}^{g-j-2}p^{i-1}(1-\deg_{\sigma^{-i}\circ\beta_0}(H))+p^{g-1}\deg_{\beta_0}(H).
\end{align*}
Hence, all the $p^g$-solutions of $y_{\sigma^{-1}\circ\beta_0}$ have valuation
\[
v_p(y_{\sigma^{-1}\circ\beta_0})=\frac{1}{p}(1-\deg_{\beta_0}(H))+\frac{1}{p}(\frac{1}{p}+\cdots +\frac{1}{p^j}-\deg_{\beta_0}(H)).
\]
Hence,  we have 
\begin{align*}
&\quad\min\{v_p(a_{\beta_0}+b_{\beta_0}y_{\sigma^{-1}\circ\beta_0}^p), v_p(y_{\sigma^{-1}\circ\beta_0}^pd_{\beta_0})\}\\
&=v_p(a_{\beta_0}+b_{\beta_0}y_{\sigma^{-1}\circ\beta_0}^p)=1-\deg_{\beta_0}(H)
\end{align*}
As in Case (1) above, $\gerN_{\beta_0}$ is generated by $(\ee_{\beta_0}+\frac{y_{\sigma^{-1}\circ\beta_0}^p}{a_{\beta_0}+b_{\beta_0}y_{\sigma^{-1}\circ\beta_0}^p}\ff_{\beta_0})$, and 
\[
\deg_{\beta_0}(D)=1-v_p(a_{\beta_0}+b_{\beta_0}y_{\sigma^{-1}\circ\beta_0}^p)=\deg_{\beta_0}(H)
\]
 for all $D\subseteq \uA[p]$ with $H\neq D$.

(3) If $\deg_{\beta_0}(H)=\delta_j$, 
then 
\begin{align*}
v_p(c_{p^g})&\geq \min_{0\leq i\leq g-1}\{v_p(a_{\sigma^{-1}\circ\beta_0}\cdots a_{\sigma^{-i}\circ\beta_0}^{p^{i-1}}b_{\sigma^{-i-1}\circ\beta_0}^{p^i}d_{\sigma^{-i-2}\circ\beta_0}^{p^{i+1}}\cdots d_{\sigma^{g}\circ\beta_0}^{p^{g-1}})\}\\
&=\sum_{i=1}^{g-j-2}p^{i-1}(1-\deg_{\sigma^{-i}\circ\beta_0}(H))+p^{g-1}(\delta_j),
\end{align*}
and all the $p^g$-solutions of $y_{\sigma^{-1}\circ\beta_0}$ have valuation
\[
v_p(y_{\sigma^{-1}\circ\beta_0})\leq \frac{1}{p}(1-\delta_j).
\]
Hence, we have
\[
\deg_{\beta_{0}}(A[p]/D)=\min\{v_p(a_{\beta_0}+b_{\beta_0}y_{\sigma^{-1}\circ\beta_0}^p), v_p(y_{\sigma^{-1}\circ\beta_0}^pd_{\beta_0})\}\leq 1-\delta_j,
\]
and thus $\deg_{\beta_0}(D)\geq \delta_j$ for all $D$. This finishes the proof of the Lemma.

\end{proof}

\begin{prop}\label{P:connectness-W-Q}
 Let $W_{\varphi,\eta}$ be a bad stratum of codimension $1$ as in Case 2(c) of Definition~\ref{Defn:Sigma}, and $\Qbar\in W_{\varphi,\eta}^\gen$ be a closed point. Then $]W^\gen_{\varphi,\eta}['\cap \spe^{-1}(\Qbar)$ is  geometrically connected. 
\end{prop}
\begin{proof}
Let $\ell(\varphi)\cap \eta=\{\beta_0\}$. 
By Stamm's theorem \cite{St} (see also \cite[Theorem~2.4.1]{GK}), we have an isomorphism 
\[
\widehat{\cO}_{\gerY,\Qbar}\simeq \cO_K[[x_{\beta_0},y_{\beta_0}, z_{2},\cdots, z_g]]/(x_{\beta_0}y_{\beta_0}-p).
\]
Let $\gerS=\Spf(\widehat{\cO}_{\gerY, \Qbar})$. Then,  the  rigid generic fiber (in the sense of Berthelot)  $S=\gerS_{\rig}$ is  identified with $\spe^{-1}(\Qbar)\subseteq \tYrig$. According to \cite[Proposition~4.8]{Ti}, we have $\deg_{\beta_0}(H)=\min\{1, v_p(y_{\beta_0}(Q))\}$ for every rigid point $Q=(A,H)\in S$. Let $T$ be the inverse image of $S$ via the map $\pi_{1,p}: \tYrig^{(p)}\ra \tYrig$. Then $\pi|_T: T\ra S$ is a finite \'etale morphism of rigid analytic spaces.  Let $g$ be an analytic function on $T$ such that $\deg_{\beta_0}(D)=\min\{1, v_p(g(Q'))\}$ for any rigid point $Q'=(\uA,H,D)\in T\subseteq \tYrig^{(p)}$. By definition of $]W^\gen_{\varphi,\eta}['$, we have  
\[
]W^\gen_{\varphi,\eta}['\cap\ \spe^{-1}(\Qbar)=S^{|g|\geq |p|^{\delta_j}},
\]
in the notation of Lemma~\ref{L:rigid-Lemma} below. Moreover, Lemma~\ref{L:BK-module} implies that the assumptions in Lemma~\ref{L:rigid-Lemma} are satisfied with $u=y_{\beta_0}$ and $a=b=\delta_j$. Hence, the Proposition follows immediately from  Lemma~\ref{L:rigid-Lemma}.

\end{proof}

\begin{lemma}\label{L:rigid-Lemma}
Let $\gerS$ be the formal scheme 
\[
\gerS=\Spf( \cO_K [[u, v, z_{2}, \cdots, z_d]]/(uv-p)),
\]
 $S=\gerS_{\rig}$ be the associated rigid analytic space. Let $\pi: T\ra S$ be a finite \'etale morphism of rigid analytic spaces. Let $g$ be an analytic function on $T$. Assume that there exist $a,b\in \Q_{>0}$ with $a\in (0,1)$ such that  we have
\begin{itemize}
\item[(a)] $|g|\leq |p|^b$ on  $\pi^{-1}(S^{|p|<|u|\leq |p|^a})$, where  
$$
S^{|p|<|u|\leq |p|^a}\colon =\{ Q\in S: |p|<|u(Q)|\leq |p|^a\},
$$ 
\item[(b)] $|g|=|p|^{b}$ on  $\pi^{-1}(S^{|p|<|u|<|p|^a})$, 

\item[(c)] $|g|>|p|^b$ on $\pi^{-1}(S^{|p|^a<|u|< 1})$. 
\end{itemize}

Then, the admissible open subdomain
\[
S^{|g|\geq |p|^b}\colon =\{Q\in S:|g(Q')|\geq |p|^b \; \text{for all }Q'\in \pi^{-1}(Q)\}.
\]
 of $S$ is geometrically connected.  
\end{lemma}
\begin{proof}

For every integer $N\geq 2$, let $S_N$ denote the affinoid subdomain of $S$ defined by $|p|^{1-1/N}\leq |u|\leq |p|^{1/N}$ and $|z_i|\leq |p|^{1/N}$ for $2\leq i\leq d$.  Since $\{S_{N}: N\geq 3\}$ form an admissible open covering of $X$, it suffices to show that $S_{N}^{|g|\geq |p|^b}\colon =S_{N}\cap S^{|g|\geq |p|^{b}}$ is geometrically connected for $N$ sufficiently large.

Let $A=\cO_K[[u,v, z_2, \cdots, z_d]]/(uv-p)$. Up to replacing $K$ by a finite extension, we may assume that there exists $\varpi\in \cO_K$ with $|\varpi|=|p|^{1/N}$. 
We put 
\[
A_N=A\langle u_N, v_N, z_{2,N},\cdots, z_{d,N}\rangle /(u-\varpi u_N, v-\varpi v_N,  z_{2}-\varpi z_{2,N},\cdots, z_{d}-\varpi z_{d,N}),
\]
and $\gerS_{N}=\Spf(A_N)$. Then  $\gerS_N$ is an admissible formal scheme over $\cO_K$, and $\gerS_{N,\rig}=S_{N}$.   
Up to replacing $N$ by  a multiple, we may assume that $a=\frac{m}{N}, b=\frac{n}{N}$ with $m,n\in \Z_{\geq 1}$ and $1<m<N-1$. Let  $\tilde{\gerS}_N$ be the $p$-adic completion of the  maximal $\cO_K$-flat closed subscheme of  the admissible blow-up $\Proj(A_N[X_1,X_2]/(X_1u_N-X_2\varpi^{m-1}))$ of $\gerS_N$ along the ideal $I=(u_N, \varpi^{m-1})$. By Raynaud's theory, the natural map $\tilde{\gerS}_{N}\ra \gerS_{N}$ induces an isomorphism on rigid generic fibers.
Let $\gerU_{N,1}$ be the open formal scheme  of $\tilde{\gerS}_N$ defined by $X_1\neq 0$ and $\gerU_{N,2}$ be that defined by $X_2\neq 0$. Then $\gerU_{1, \rig}$ (resp. $\gerU_{2, \rig}$) is the admissible open subset of $\fX_{N,\rig}$ defined by $|p|^{1-1/N}\leq |u|\leq |p|^a$ (resp. $|p|^a\leq |u|\leq |p|^{1/N}$). 
Then $x_2=X_2/X_1$ and  $x_1=X_1/X_2$  are respectively defined on $\gerU_{N,1}$ and $\gerU_{N,2}$, and $x_2=x_1^{-1}$ on $\gerU_{N,1}\cap \gerU_{N,2}$. We have $u_N=\varpi^{m-1}x_2$ on $\gerU_{N,1}$, and $v_N=\frac{p}{\varpi^{m+1}}x_1$ on $\gerU_{N,2}$. Then, the special fiber of $\tilde{\gerS}_N$, denoted by $\tilde{S}_{N,0}$, is geometrically reduced, and it consists of $3$ geometrically irreducible components: $C_0$, $C_{\infty}$ and $C_{1}$, which are all smooth of dimension $d$ and  defined  respectively by $x_2=0$, $x_1=0$ (i.e. $x_2=\infty$), and $u_{N}=v_N=0$.  Let $U_{N,i}$ for $i=1,2$ denote the open subset of $\tilde{S}_{N,0}$ corresponding to $\gerU_{N,i}$. Then we have $U_{N,1}=\tilde{S}_{N,0}-C_{\infty}$ and $U_{N,2}=\tilde{S}_{N,0}-C_{0}$.

  Let $\tilde{\gerT}_N$ be the normalization of $\tilde{\gerS}_N$ in $T$. Then $\tilde{\gerT}_N$ is a formal model of $T_N:=\pi^{-1}(\gerS_{N,\rig})$, and   $\pi$ extends to a finite map $\pi: \tilde{\gerT}_N\ra \tilde{\gerS}_N$ of formal schemes.  We put $\gerV_{N,i}=\pi^{-1}(\gerU_{N,i})$. Consider the rigid analytic function $\tilde{g}=g/\varpi^n$. We have $|\tilde{g}|\leq 1$ on $\gerV_{N,1,\rig}$ by condition (a). Hence, $\tilde{g}$ descends to a global section on $\gerV_{N,1}$. Let $V_{N,1}$ be the reduced special fiber of $\gerV_{N,1}$, and $\bar g$ denote the image of $\tilde{g}$ in $\Gamma(V_{N,1},\cO_{V_{N,1}})$ by reduction.  Let $Z\subseteq V_{N,1}$ denote the closed subscheme defined by the vanishing of $\bar{g}$. Note that $]Z[$ is the subdomain of $Y_{N}$ such that $|g|<|p|^b$ by our construction of $Z$ and condition (c). Therefore, we have 
$$
S_{N}^{|g|\geq |p|^b}=\gerS_{N,\rig}-\pi(]Z[)=]\tilde{S}_{N,0}-\pi(Z)[.
$$
By \cite[Proposition 4.3]{AS}, $S_{N}^{|g|\geq |p|^b}$ is geometrically connected if and only if $\tilde S_{N,0}-\pi(Z)$ is geometrically connected.
  Since $\pi$ is finite, $\pi(Z)$ is a closed subscheme of $U_{N,1}$. By condition (b), $\pi(Z)$ has no intersection with $C_0$. Thus $\pi(Z)$ is a closed subset contained in $U_{N,1}-C_{0}=C_{1}-C_0-C_{\infty}$. Since $C_{1}-C_0-C_{\infty}$ is smooth  and irreducible of dimension $d$, $\pi(Z)$ must have dimension $<d$. Hence, $\tilde S_{N,0}-\pi(Z)$ is  geometrically connected.  This finishes the proof of the Lemma. 
 
\end{proof}

 \section{Gluing}\label{section: gluing}

 Let $\uA$ be a Hilbert-Blumenthal abelian variety. For any point $P$ of $A[p]$ and $\gerp\in\SS$, write  $P=P^\gerp \times P_\gerp$, where $P^\gerp \in A[p/\gerp]$ and $P_\gerp \in A[\gerp]$.   Throughout this section, for a point of $\HMV_K$, we will replace our usual notation of $(\uA,P)$ by $(\uuA,P_\gerp)$, where $P_\gerp$ is as above, and $\uuA=(\uA,P^\gerp)$. For a point $Q$ of $\uA$ of finite order, we denote by $(Q)$ the $\ol$-module generated by $Q$ inside $A$.
 
We begin by introducing some auxiliary varieties.  Let us fix $\gern|N$, an ideal of $\calO_L$, as well as an $\ol$-generator $\zeta_\gern$ of $(\GG_m\otimes\gerd_L^{-1})[\gern]$. Recall the choice of $\zeta_\gerp$ in Section  \ref{sect-AL-inv}. For a Hilbert-Blumenthal abelian variety $\uA$, consider the $\ol$-linear Weil pairing $\langle -,-\rangle_\gern: A[\gern] \times A[\gern] \arr (\GG_m\otimes\gerd_L^{-1})[\gern]$.

 Let $\HMV^{\underline{\gern}}_{K}$ be the scheme which classifies tuples $(\uuA,P_\gerp,P_\gern, Q_\gern)$ over $K$-schemes, where $(\uuA,P_\gerp)$ is classified by $\HMV_K$,  and $P_\gern, Q_\gern$ are points of maximal order in $A[\gern]$  satisfying $(P_\gern)=(P_N) \cap \uA[\gern]$, where $P_N$ is the point giving the level $N$ structure included in the notation $\uuA$,  and such that   $\langle P_\gern,Q_\gern \rangle_\gern =\zeta_\gern$. Let $\tHMVAN^{ \underline{\gern}}$ denote the toroidal compactification of $\HMV^{ \underline{\gern}}_K$ obtained using our fixed choice of collection of cone decompositions, with the same notation used for the associated rigid analytic variety.

Given $\gerp \in \SS$,  let $\HMV^{\underline{\gerp},\underline{\gern}}_{K}$ be the scheme which classifies tuples $(\uuA,P_\gerp,Q_\gerp, P_\gern, Q_\gern)$ over $K$-schemes, where $(\uuA,P_\gerp,  P_\gern,Q_\gern)$ is classified by $\HMV^{\underline{\gern}}_K$,  and $Q_\gerp$ is a point of maximal order in $A[\gerp]$  such that $\langle P_\gerp,Q_\gerp\rangle_\gerp =\zeta_\gerp$. Let $\tHMVAN^{\underline{\gerp},\underline{\gern}}$ denote the toroidal compactification of $\HMV^{\underline{\gerp},\underline{\gern}}_K$  with the same notation used for the associated rigid analytic variety.

 For any $j_\gerp \in \calO_L/\gerp$,  there is a finite flat morphism
\[
\pi_{2,j_\gerp}:\tHMVAN^{\underline{\gerp},\underline{\gern}} \arr \tHMVAN^{\underline{\gern}},
\]
defined by
$\pi_{2 ,j_\gerp}(\uuA,P_\gerp,Q_\gerp,P_\gern,Q_\gern)=(\uuA/(j_\gerp P_\gerp+Q_\gerp),\overline{P_\gerp},  \overline{P_\gern}, \overline{Q_\gern})$  on the non-cuspidal part, where overline denotes  image in the quotient. We also define
\[
\pi_{1,\gerp}:\tHMVAN^{\underline{\gerp},\underline{\gern}} \arr \tHMVAN^{\underline{\gern}},\]
via  $\pi_{1,\gerp}(\uuA,P_\gerp, Q_\gerp, P_\gern,Q_\gern)=(\uuA, P_\gerp,P_\gern,Q_\gern)$.  Similarly, for any $j_\gern \in \calO_L/\gern$, we define
\[
\pi_{2,j_\gern}:\tHMVAN^{\underline{\gern}} \arr \tHMVAN,
\]
by
$\pi_{2,j_\gern}(\uuA,P_\gerp,P_\gern,Q_\gern)=(\uuA/(j_\gern P_\gern+Q_\gern),\overline{P_\gerp})$. We also define
\[
\pi_{1,\gern}:\tHMVAN^{\underline{\gern}} \arr \tHMVAN,
\]
via  $\pi_{1,\gern}(\uuA,P_\gerp,P_\gern, Q_\gern)=(\uuA, P_\gerp)$. 
Finally, we define 
\[
w_\gerp: \tHMVAN^{\underline{\gern}} \arr \tHMVAN^{\underline{\gern}},
\] given by sending $(\uuA, P_\gerp,P_\gern,Q_\gern)$ to $(\uuA/(P_\gerp),\overline{Q_\gerp},\overline{P_\gern},\overline{Q_\gern})$, for any choice of a point of maximal order $Q_\gerp\in A[\gerp]$ satisfying $\langle P_\gerp,Q_\gerp\rangle_\gerp=\zeta$.

\begin{defn}\label{Definition: intervals} Let $\gerp\in \SS$.   We  define $\II^\ast_\gerp$ to be the interval $(\sum_{i=1}^{f_\gerp-1} 1/p^i,1)$ if $f_\gerp>1$, and $(0,1)$ if $f_\gerp=1$. 
\end{defn}



\begin{defn} Recall $\tilde{\gerY}_\rig^{|\tau|\leq 1}$ which was introduced in \cite[\S]{Ka}. It is an admissible open in $\tilde{\gerY}_\rig$ whose points are those $Q=(\uA,C)$ such that $|\tau(\Qbar) \cap \BB_\gerp|\leq 1, \forall \gerp$, where $\tau(\Qbar)$ is defined in \eqref{Defn:tau-Qbar}. We define $\tHMVANtau=\prtoY^{-1}(\tilde{\gerY}_\rig^{|\tau|\leq 1})$.

\end{defn}

  The following lemma is crucial for the gluing process and was essentially proved in \cite{Ka}.
 \begin{lemma} \label{Lemma: saturated}
Let $\underline{I}$ be a multiset of intervals as in Definition \ref{Definition: admissible domains}, such that $I_\gerp=\II^\ast_\gerp$.  
 \begin{enumerate}
 \item  If $(\uA,P) \in \prtoY^{-1}(\Sigma\underline{I} )$, then for any nonzero point $Q_\gerp$ of $A[\gerp]$, we have
\[
(A,P^\gerp \times Q_\gerp)\in \prtoY^{-1}(\Sigma\underline{I} ).
\]

 \item  For all $j_\gerp \in \ol/\gerp$, we have the following inclusion inside $\tHMVAN^{\underline{\gerp},\underline{\gern}}$:
 \[
 \pi_{1,\gerp}^{-1}\pi_{1,\gern}^{-1}\prtoY^{-1}(\Sigma) \subset \pi_{2,j_\gerp}^{-1}\pi_{1,\gern}^{-1}\prtoY^{-1}(\Sigma).
 \]
 \end{enumerate}
 \end{lemma}
 \begin{proof}
 

The first statement  follows directly from \cite[Lemma 5.9, Corollary 7.9]{Ka} (keeping in mind Remark \ref{Remark: change of notation}). 
The second statement follows from Proposition \ref{Prop:Sigma-U_p}.   \end{proof}

In the following, we will prove two connectivity results which are essential for the gluing process. Given $\calU \subset \tYrig$, $I \subset [0,1]$, and $\beta \in \BB$, we define an admissible open subset of $\calU$
\[
\calU I_\beta:=\{Q\in \calU: \deg_\beta(Q)\in I\}.
\] 

We first prove a preliminary result which follows directly from the study of the geometry of $\tilde{\Ybar}$ (the special fiber of $\tilde{Y}$) conducted  in \cite{GK}. For $W \subset \tilde{\Ybar}$, we denote its Zariski closure by $W^{cl}$.
 
\begin{lemma} \label{Lemma: intersect}
 Let $\Wpe$ be a stratum in $\tilde{\Ybar}$, and $(\varphi^\prime,\eta^\prime)$ be another admissible pair with  $\supseteq \varphi, \eta^\prime\supseteq \eta$. 

\begin{enumerate}
\item Let $W$ be an irreducible component of  $\Wpe$. There exists an irreducible component $V$ of $W_{\varphi^\prime,\eta^\prime}$, such that $V \subset W^{cl}$.

\item Let $V $ be an irreducible component of $W_{\varphi^\prime,\eta^\prime}$. There exists an irreducible component $W$ of $\Wpe$ such that $V \subset W^{cl}$.
\end{enumerate}
\end{lemma}

\begin{proof}
 Part (2) is immediate: since $V \subset W_{\varphi^\prime,\eta^\prime} \subset W^{cl}_{\varphi,\eta}$, it follows that  $V^{cl}$ lies inside an irreducible component of $W^{cl}_{\varphi,\eta}$. 

For part (1), we argue as follows.  Let $W$ be an irreducible component of $\Wpe$. By \cite[Theorem 2.6.13]{GK}, $W^{cl}$ contains a point in $W_{\BB,\BB}$. Since $W_{\BB,\BB}$ lies inside every closed stratum, it follows that $W^{cl}$ intersects $V^{cl}$, where $V$ is an irreducible component of $W_{\varphi^\prime,\eta^\prime}$. By part (2) of the lemma, there is an irreducible component $W_1$ of $\Wpe$ such that $ V \subset W^{cl}_1$. The nonsingularity of the strata implies that no two distinct irreducible components of  $W^{cl}_{\varphi,\eta}$ meet, and, hence, we must have $W_1=W$. This ends the proof.

\end{proof}


The following general definition includes, as examples, two regions in $\tilde{\gerY}_\rig$ for which we want to prove connectivity results.

\begin{defn} \label{Def: R}Assume that  for  any nowhere-\'etale stratum $\Wpe$ of  $\tilde{\Ybar}$ of codimension 0 or 1, we are given a Zariski open dense subset ${\Wpe^0}\subset \Wpe$, and an admissible open subset $\calR_{\Wpe} \subset \spe^{-1}(W^0_{\varphi,\eta})$. Given such a collection $\calR$, we define
\[
\Sigma_\calR\colon=\bigcup_{\pe} \calR_{\Wpe},
\]
where $\pe$ runs over all admissible pairs with $\Wpe$ nowhere-\'etale and of codimension 0 or 1. For any irreducible component $W \subset \Wpe$, we set 
\[
W^0\colon=W^0_{\varphi,\eta} \cap W,
\]
 \[
 \calR_W\colon=\calR_{\Wpe} \cap \spe^{-1}(W).
 \]
 We assume the following conditions hold:
\begin{enumerate}
  \item If $W \subset \Wpe$ is an irreducible component of a nowhere-\'etale stratum of codimension 0, then $\calR_W=\spe^{-1}(W^0)$.
 
 \item Let $V \subset \Wpe$ be  an irreducible component of a nowhere-\'etale stratum of codimension 1, and $\ell(\varphi)\cap \eta=\{\beta\}\subset \BB_\gerp$. Then, if $\eta_{\gerp}:=\eta\cap \BB_{\gerp}\neq \BB_\gerp$, we have 
 \[
 \calR_V(0,\delta)_\beta=\spe^{-1}(V^0)(0,\delta)_\beta,
 \]
 \[
  \calR_V(\epsilon,1)_\beta=\spe^{-1}(V^0)(\epsilon,1)_\beta,
 \]
 for some $0<\epsilon,\delta<1$. If $\eta_{\gerp}=\BB_\gerp$, then we only require the second condition.
 
  \item Let $V$ be as in (2),   $Q \in \calR_V$, and $\Qbar=\spe(Q)$. Then, $\spe^{-1}(\Qbar) \cap \calR_V$ is connected and $\deg_\beta$ takes values arbitrarily close to $0$ and $1$ on it (If $\eta_{\gerp}=\BB_{\gerp}$, we only require that $\deg_{\beta}$ can take values arbitrarily close to $1$).
 \end{enumerate}

\end{defn}

\begin{lemma} \label{Lemma: Connected} Fix $\gerp \in \SS$. Assume that the collection $\calR$ is as in Definition \ref{Def: R}. 

\begin{enumerate}
 \item Let $Q \in \Sigma_\calR \cap \spe^{-1}(\Wpe)$, where $\Wpe$ is a nowhere-\'etale stratum of codimension 0. Then, there is a connected subset $C \subset \Sigma_\calR$ which contains $Q$, and such that $C \cap \spe^{-1}(W_{\varphi,\eta\cup \{\beta\}})\neq \emptyset$, for any $\beta \not\in\eta_\gerp$.
 
  \item Let $Q \in \Sigma_\calR \cap \spe^{-1}(\Wpe)$, where $\Wpe$ is a nowhere-\'etale stratum of codimension 0 such that $\varphi_\gerp \neq \BB_\gerp$. Then, there is a connected subset $C \subset \Sigma_\calR$ which contains $Q$, and such that $C \cap \spe^{-1}(W_{\varphi\cup\{\beta\},\eta})\neq \emptyset$, for any $\beta \not\in\varphi_\gerp$.
  
   \item Let $Q \in \Sigma_\calR \cap \spe^{-1}(\Wpe)$, where $\Wpe$ is a nowhere-\'etale stratum of codimension 1 such that $\eta_\gerp \neq \BB_\gerp$ (resp., $\varphi_\gerp \neq \BB_\gerp$). Then, there is a connected $C \subset \Sigma_\calR$ containing $Q$, such that $C \cap \spe^{-1}(W_{\varphi-\{\sigma \circ \beta\},\eta})\neq \emptyset$ (resp. $C \cap \spe^{-1}(W_{\varphi,\eta-\{\beta\}})\neq \emptyset$), where $\{\beta\}=\ell(\varphi)\cap \eta$.

\end{enumerate}

\end{lemma}

The Lemma immediately implies the following.

\begin{cor}\label{Cor: connected} Let $\calR$ be as in Definition \ref{Def: R}. 
Every connected component of $\Sigma_\calR$ intersects $\spe^{-1}(W_{\BB,\emptyset})$, as well as $\spe^{-1}(W_{\emptyset,\BB})$. Furthermore, given $\gerp \in \SS$, every connected component of $\Sigma_\calR$  intersects  $\spe^{-1}(W_{\varphi,\eta})$, where $\Wpe$ is a codimension-1 stratum satisfying  $\varphi_\gerq=\BB_\gerq$, $\eta_\gerq=\emptyset$, for $\gerq\neq\gerp$,  and  $\eta_\gerp=\BB_\gerp$, $|\varphi_\gerp|=1$ (respectively, $\varphi_\gerp=\BB_\gerp$, $|\eta_\gerp|=1$).
\end{cor}

\begin{proof} (of  Lemma \ref{Lemma: Connected}) We begin by proving (1). Let $W_1$ be an irreducible component of $\Wpe$ containing $\Qbar$. By Definition \ref{Def: R}, $Q \in \spe^{-1}(W_1^0)$. By Lemma \ref{Lemma: intersect} there is an irreducible component $V$ of $W_{\varphi,\eta\cup\{\beta\}}$ which lies  inside $W^{cl}_1$, and  an irreducible component of $W_{\varphi-\{\sigma\circ\beta\},\eta\cup\{\beta\}}$, say $W_2$, such that $V$ lies inside  $W^{cl}_2$. Note that $\beta\in \eta^c\subset \ell(\varphi)$, and hence $W_2$ has codimension $0$. We set  $U=W_1 \cup V \cup W_2$. Then, 
\[
\tilde{\Ybar}-U=\bigcup_{Z} Z,
\]
where $Z$ runs over all irreducible components of all closed codimension-0 strata different from of $W_1^{cl}$ and $W_2^{cl}$.  

It follows that $U$ is Zariski open in $\tilde{\Ybar}$. Hence, so is $U^0:=W^0_1 \cup V^0 \cup W^0_2$. In particular, for any  rational $0<\epsilon<1$, the region 
\[
U^0[\epsilon,1]_\beta=\spe^{-1}(W^0_1) \cup \spe^{-1}(V^0)[\epsilon,1)_\beta
\]
 is quasi-compact (noting that $\deg_\beta$ is $1$ on $\spe^{-1}(W_1)$ and $0$ on $\spe^{-1}(W_2)$, and in $(0,1)$ on $\spe^{-1}(V)$).  Note that $U^0[1,1]_\beta=\spe^{-1}(W^0_1)$ is connected since $W^0_1$ is dense in $W_1$. We claim that $U^0[\epsilon,1]_\beta$ is connected as well.  Assume not. Then,  it must have a connected component $D$ which does not intersect $U^0[1,1]_\beta$. Since $D$ is quasi-compact, we can use the maximum modulus principle to deduce 
 \[
 D \subset U^0[\epsilon,\epsilon_0]_\beta,  \quad\qquad (\dagger)
 \]
 for some $\epsilon_0<1$.  Let $P \in D$ with specialization $\Pbar$. Let $\calA:=\spe^{-1}(\Pbar)[\epsilon,1)_\beta$. Since $\Pbar \in W_{\varphi,\eta\cup\{\beta\}}$, we deduce by Stamm's Theorem presented in \cite[Theorem 2.4.1]{GK}, that 
 \[
 \spe^{-1}(\Pbar)\cong \{(x_\gamma)_{\gamma \in \BB}: \nu(x_\gamma) \geq 0\ \ \  \forall  \gamma \in \BB-\{\beta\};\  0 \leq \nu(x_\beta) \leq 1\},
 \]
where, for any $R \in  \spe^{-1}(\Pbar)$, we have $\deg_\beta(R)=1-\nu_\beta(Q)=1-\nu(x_\beta(R))$ by  Remark \ref{Remark: change of notation}. Therefore, $\calA$ is a connected subset which contains $P$ and lies entirely in $U^0[\epsilon,1]_\beta$. It follows that $\calA \subset D$, which contradicts $(\dagger)$, as $\deg_\beta$ takes values arbitrarily close to $1$ on $\calA$. 
 Now, taking $\epsilon$ as in part (2) of Definition \ref{Def: R} (and remembering $Q \in \spe^{-1}(W_1^0)$), we find that $C=U^0[\epsilon,1]_\beta$ serves as the desired connected subset. 
 
 Part (2) of the Lemma can be proven exactly as above. We now turn to part (3) of the lemma. Let $V$ be an irreducible component of $\Wpe$ which contains $\Qbar$. As above, let $W_1$ be an irreducible component of $W_{\varphi,\eta-\{\beta\}}$ whose Zariski closure contains $V$. Repeating the above argument, we find that for an appropriate choice of $\epsilon$, $U^0[\epsilon,1]_\beta:=\spe^{-1}(W_1^0) \cup  \spe^{-1}(V^0)[\epsilon,1)_\beta$ is connected and lies entirely inside $\calR_\Sigma$. By assumptions (2,3) in Definition \ref{Def: R}, $\spe^{-1}(\Qbar) \cap \calR_\Sigma$ is connected and $\deg_\beta$ takes values arbitrarily close to $1$ on it. Since $\spe^{-1}(\Qbar) \subset \spe^{-1}(V^0)$, it follows that $\spe^{-1}(\Qbar) \cap \calR_\Sigma$ intersects $U^0[\epsilon,1]_\beta$, and hence $Q$ lies in $C$, the connected component of $\calR_\Sigma$ containing $U^0[\epsilon,1]_\beta$ which clearly intersects $\spe^{-1}(W_{\varphi,\eta-\{\sigma\circ\beta\}})$. The other statement in part (3) of the lemma can be proved in exactly the same way.

 \end{proof}

\begin{prop} \label{Prop: connectivity} Let $\Sigma$ be the region defined in  Definition \ref{Defn:Sigma}.   Fix $\gerp\in\SS$.

\begin{enumerate}

\item Every connected component of $\Sigma$ contains $\spe^{-1}(W)$, where $W$ is an irreducible component of $W_{\BB,\emptyset}$, as well as $\spe^{-1}(V)$, where $V$ is an irreducible component of $W_{\BB,\{\beta\}}$, for some $\beta\in\BB_\gerp$.

\item  Let $T_1,T_2 \subset \SS$. Every connected component of $w_{T_1}^{-1}(\Sigma) \cap w_{T_2}^{-1}(\Sigma)$ intersects a region of the form $\tilde{\gerY}^{|\tau|\leq 1}_\rig \underline{I}$, where $\underline{I}$ is a multiset of intervals satisfying $I_\gerp=\II^\ast_\gerp$ given in Definition \ref{Definition: intervals}. 
\end{enumerate}

\end{prop}

\begin{proof} The result follows from Lemma \ref{Cor: connected} as follows. Taking $\calR:=\{\calR_V\}$ as in Definition \ref{Def: R}, given by 
\[
W^0_{\varphi,\eta}\colon=\Wpe^{\gen}  \qquad  \qquad \calR_{\Wpe}\colon=]\Wpe^{\gen}[^\prime,
\] implies that $\Sigma_\calR=\Sigma$. On the other hand, given a collection $\calR\colon=\{\calR_V\}$ defined by 
\[
W^0_{\varphi,\eta}\colon=w_{T_1}^{-1}(W^{\gen}_{w_{T_1}(\varphi,\eta)})\cap w_{T_2}^{-1}(W^{\gen}_{w_{T_2}(\varphi,\eta)})
\]
\[
\calR_{\Wpe}\colon=w_{T_1}^{-1}(]W^{\gen}_{w_{T_1}(\varphi,\eta)}[^\prime)\  \cap\  w_{T_2}^{-1}(]W^{\gen}_{w_{T_2}(\varphi,\eta)}[^\prime),
\]
implies that $\Sigma_\calR=w_{T_1}^{-1}(\Sigma) \cap w_{T_2}^{-1}(\Sigma)$. To prove the proposition, we first need to show that the two collections $\calR$ presented above satisfy the conditions given in Definition \ref{Def: R}. This will be done in Lemma \ref{Lemma: R conditions}. Assume this is the case for the rest of this proof. Then, applying Corollary \ref{Cor: connected} to the first choice of $\calR$, we find that every connected component of $\Sigma$ intersects $\spe^{-1}(W_{\BB,\emptyset})$, as well as  as well as $\spe^{-1}(W_{\BB,\{\beta\}})$, for some $\beta\in\BB_\gerp$. Since $\spe^{-1}(W_{\BB,\emptyset}) \subset \Sigma \cap \calV_{\can}$, it follows that every connected component of $\Sigma$ contains a connected component of  $\spe^{-1}(W_{\BB,\emptyset})$. Nonsingularity of strata implies that every  connected component of  $\spe^{-1}(W_{\BB,\emptyset})$ is of the form $\spe^{-1}(W)$ where $W$ is an irreducible component of $W_{\BB,\emptyset}$. An exactly similar argument, using the fact that $\spe^{-1}(W_{\BB,\{\beta\}}) \subset \Sigma \cap \calV_{\can}$, proves the second statement in part (1) of the lemma.

For part (2) of the lemma, we use the second choice of $\calR$ discussed above. By Corollary \ref{Cor: connected},  every connected component of $\Sigma_\calR$ intersects $\spe^{-1}(W_{\varphi,\eta})$, where $\Wpe$ is a codimension-1 stratum satisfying  $\varphi_\gerq=\BB_\gerq$, $\eta_\gerq=\emptyset$, for $\gerq\neq\gerp$,  and  $|\varphi_\gerp|=1$, $\eta_\gerp=\BB_\gerp$. To prove part (2) of the  lemma, it is enough to show that
\[
\spe^{-1}(\Wpe)  \cap \Sigma \subset \tilde{\gerY}^{|\tau|\leq 1}_\rig \underline{I}, \qquad\qquad (\dagger\dagger)
\]
where $\underline{I}$ is a multiset of intervals with $I_\gerp=\II^*_\gerp$, and $I_\gerq=\{n_\gerq\}$, with an integer $n_\gerq \in[0,f_\gerq]$, for all  $\gerq\neq \gerp$.  By \cite[Cor. 2.3.4]{GK}, on $\Wpe$, we have $\tau_\gerq=\emptyset$, and $|\tau_\gerp|=1$, which implies that $\spe^{-1}(\Wpe) \subset  \tilde{\gerY}^{|\tau|\leq 1}_\rig$. Let $\underline{I}$ be the smallest multiset of intervals satisfying $(\dagger\dagger)$. By definition of partial degrees on $\spe^{-1}(\Wpe)$, it follows that for $\gerq\neq \gerp$, $I_\gerq$ must be a singleton consisting of an integer in $[0,f_\gerq]$. Finally, it follows from the definition of $\Sigma$ that $I_\gerp =\II_\gerp^\ast$.

 \end{proof}

\begin{lemma}\label{Lemma: R conditions} The two choices of $\calR$ given in the proof of Proposition \ref{Prop: connectivity} satisfy the  conditions (1)-(3) given in Definition \ref{Def: R}

\end{lemma}

\begin{proof}  Let $\calR$ be the first choice in the proof of \ref{Prop: connectivity}. Condition (1) of \ref{Def: R} is direct from the definition of $\calR_{W_{\varphi,\eta}}=]W^\gen_{\varphi,\eta}['$ in \ref{Defn:Sigma}. For Condition (2), we distinguish the cases of $W_{\varphi,\eta}$:
\begin{itemize}
\item If $W_{\varphi,\eta}$ is a stratum as in Case 2(a) of \ref{Defn:Sigma},  we have $]W^\gen_{\varphi,\eta}['=]W^\gen_{\varphi,\eta}[$. We may take any $\epsilon,\delta\in (0,1)$. 

\item If $W_{\varphi,\eta}$ is in Case2(b),  we  take $\delta=\epsilon =\sum_{i=1}^{f_{\beta}-1}\frac{1}{p^i}$. 

\item If $W_{\varphi,\eta}$ is in Case2(c), it follows from Lemma~\ref{L:BK-module} that  we can take $\delta=\delta_j$ and $\epsilon=1-\delta_j$. 
\end{itemize}  
For Condition (3), we prove first that $\spe^{-1}(\Qbar)\cap\calR_{W_{\varphi,\eta}}=\spe^{-1}\cap\ ]W^\gen_{\varphi,\eta}['$ is connected, for any nowhere \'stale codimension 1 stratum $W_{\varphi,\eta}$. For $W_{\varphi,\eta}$  in Case 2(a) of \ref{Defn:Sigma}, we have $\spe^{-1}(\Qbar)\cap\ ]W^\gen_{\varphi,\eta}['=\spe^{-1}(\Qbar)$, which is clearly connected. If $W_{\varphi,\eta}$ is in Case2(b), we have  
$$
\spe^{-1}(\Qbar)\cap ]W^\gen_{\varphi,\eta}['=\spe^{-1}(\Qbar)(\sum_{i=1}^{f_{\gerp}-1}\frac{1}{p^i},1)_{\beta}
$$
with $\ell(\varphi)\cap\eta=\{\beta\}$.
which is also clearly connected. For $W_{\varphi,\eta}$ in Case 2(c), the connectedness of $\spe^{-1}(\Qbar)\cap\ ]W^\gen_{\varphi,\eta}['$ was proved in Proposition~\ref{P:connectness-W-Q}. The second half of (3) is clear by Lemma~\ref{L:BK-module}.

Consider now the second choice of $\calR$ in the proof of \ref{Prop: connectivity}. Condition (1) is direct from the definition of $]W^\gen_{\varphi,\eta}['$ for codimension $1$ stratum $W_{\varphi,\eta}$. For Condition (2), by \ref{equ-deg-wp} and discussion for the first collection of $\calR$, we may take $\delta=1/p$ and $\epsilon=1-1/p$ regardless of the cases of $W_{\varphi,\eta}$. For Condition (3), the only non-trivial part is to prove the connectness of 
\begin{align*}
\spe^{-1}(\Qbar)\cap\calR_{W_{\varphi,\eta}}\colon &=\spe^{-1}(\Qbar)\cap w_{T_1}^{-1}(]W^\gen_{w_{T_1}(\varphi,\eta)}[')\cap w_{T_2}^{-1}(]W^\gen_{w_{T_2}(\varphi,\eta)}[')\\
&=\{Q\in \spe^{-1}(\Qbar): w_{T_1}(Q)\in ]W^\gen_{w_{T_1}(\varphi,\eta)}[', w_{T_2}(Q)\in ]W^\gen_{w_{T_2}(\varphi,\eta)}['\}
\end{align*} 
for all $\Qbar\in W^0_{\varphi,\eta}=w_{T_1}^{-1}(W_{w_{T_1}(\varphi,\eta)}^\gen)\cap w_{T_2}^{-1}(W_{w_{T_2}(\varphi,\eta)}^\gen)$.
Let $\ell(\varphi)\cap \eta=\{\beta\}$ and $\gerp\in \SS$ such that $\beta\in \BB_\gerp$.  
We have several cases:
\begin{itemize}
\item  $\gerp\notin T_1$ and $\gerp\notin T_2$. In this case, the stratum  $W_{w_{T_1}(\varphi,\eta)}$ and $W_{w_{T_2}(\varphi,\eta)}$ are in the same case as classified in Definition~\ref{Defn:Sigma}.    If $Q=(A,H)\in \spe^{-1}(\Qbar)$ and $w_{T_i}(Q)=(A_i,H_i)$ for $i=1,2$, then  we have a canonical isomorphism of $p$-divisible groups
\begin{equation}\label{Equ:isom-p-divisible}
A[\gerp^\infty]\simeq A_1[\gerp^\infty]\simeq A_2[\gerp^\infty].
\end{equation}
 If both $W_{w_{T_1}(\varphi,\eta)}$ and $W_{w_2(\varphi,\eta)}$ are both in Case2(a), then $\spe^{-1}(\Qbar)\cap \calR_{W_{\varphi,\eta}}=\spe^{-1}(\Qbar)$, which is clearly connected. If both of them are in Case2(b),  we have 
\[
\spe^{-1}(\Qbar)\cap \calR_{W_{\varphi,\eta}}=\spe^{-1}(\Qbar)(\sum_{i=1}^{f_{\gerp}-1}\frac{1}{p^i}, 1)_{\beta}.
\]
If both strata are in Case2(c), using the isomorphisms \eqref{Equ:isom-p-divisible}, we see  easily that, for  $Q\in \spe^{-1}(\Qbar)$,  $w_{T_1}(Q)\in ]W^\gen_{w_1(\varphi,\eta)}['$ and $w_{T_2}(Q)\in ]W^\gen_{w_{T_2}(\varphi,\eta)}['$ hold if and only if $Q\in ]W^\gen_{\varphi,\eta}['$. Hence, we get
\[
\spe^{-1}(\Qbar)\cap \calR_{W_{\varphi,\eta}}=\spe^{-1}(\Qbar)\cap ]W^\gen_{\varphi,\eta}[',
\]
which is connected by Proposition~\ref{P:connectness-W-Q}.

\item $\gerp\in T_1$ and $\gerp\in T_2$. We have then 
\[
\spe^{-1}(\Qbar)\cap \calR_{W_{\varphi,\eta}}=w_{\gerp}^{-1}\biggl(\spe^{-1}(w_{\gerp}(\Qbar))\cap w_{T_1-\{\gerp\}}^{-1}(]W_{w_{T_1}(\varphi,\eta)}^\gen[')\cap w_{T_2-\{\gerp\}}^{-1}(]W_{w_{T_2(\varphi,\eta)}}^\gen[')\biggr).
\]
As $w_{\gerp}^{-1}$ is an isomorphism, $\spe^{-1}(\Qbar)\cap \calR_{W_{\varphi,\eta}}$ is connected if and only if so is $\spe^{-1}(w_{\gerp}(\Qbar))\cap w_{T_1-\{\gerp\}}^{-1}(]W_{w_{T_1}(\varphi,\eta)}^\gen[')\cap w_{T_2-\{\gerp\}}^{-1}(]W_{w_{T_2(\varphi,\eta)}}^\gen[')$. We conclude thus by  the previous case.

\item $\gerp$ lies in one of $T_1$ and $T_2$, but not in the other. In this case, exactly one of $W_{w_{T_1}(\varphi,\eta)}$ and $W_{w_{T_2}(\varphi,\eta)}$ is bad in the sense of Definition~\ref{Defn:goodness}. We may assume thus $W_{w_{T_1}(\varphi,\eta)}$ is bad, hence $W_{w_{T_2}(\varphi,\eta)}$ is good. We have then 
\[
\spe^{-1}(\Qbar)\cap \calR_{W_{\varphi,\eta}}=\spe^{-1}(\Qbar)\cap w_{T_1}^{-1}(]W^\gen_{w_{T_1}(\varphi,\eta)}[')=w_{T_1}^{-1}(\spe^{-1}(w_T(\Qbar))\cap ]W^\gen_{w_{T_1}(\varphi,\eta)}[').
\]
As $w_{T_1}^{-1}$ is an isomorphism, the connectivity of $\spe^{-1}(\Qbar)\cap \calR_{W_{\varphi,\eta}}$ follows from Proposition~\ref{P:connectness-W-Q}.
\end{itemize}

\end{proof}


For any $S \subset \SS$, let $\Sigma_S= \bigcup_{T \subseteq S} w_T^{-1}(\Sigma)$. 

\begin{lemma} \label{Lemma: final connectivity} Let $\gerp\not\in S \subset \SS$, and $\gern$ an ideal of $\ol$ prime to $p$. The following hold true.

\begin{enumerate}

\item  Every connected component of $\prtoY^{-1}(\Sigma_S)$ contains $\prtoY^{-1}(\spe^{-1}(W))$, where $W$ is an irreducible component of $W_{\BB,\emptyset} \subset \tilde{\Ybar}$.

\item Every connected component of $\pi_{1,\gern}^{-1}\prtoY^{-1}(\Sigma_S \cap w_\gerp^{-1}(\Sigma_S))$ intersects  a region of the form $\pi_{1,\gern}^{-1}(\tHMVANtau \underline{I})=\pi_{1,\gern}^{-1}\prtoY^{-1}(\tilde{\gerY}^{|\tau|\leq 1}_\rig \underline{I})$, where $\underline{I}$ is a multiset of intervals satisfying $I_\gerp=\II^\ast_\gerp$ given in Definition \ref{Definition: intervals}. 

\end{enumerate}
\end{lemma}

\begin{proof} Since $\prtoY$ and $\pi_{1,\gern}$ are  finite flat maps, it is enough to prove the statements with all instances of $\pi_{1,\gern}^{-1}$, $\prtoY^{-1}$ removed. For part (1), by Proposition \ref{Prop: connectivity}, and in view of definition of $\Sigma_S$, it is enough to prove that given $T_1, T_2\subset \SS$ with $T_2= T_1\cup \{\gerp\}$ for some $\gerp\in\SS$,  every connected component $D$ of $w_{T_2}^{-1}(\Sigma)$ intersects $w_{T_1}^{-1}(\Sigma)$. Since $w_{T_1}$ is an automorphism, we have $D=w_{T_2}^{-1}(D^\prime)$, where $D^\prime$ is a connected component of $\Sigma$.  It is, therefore, enough to show that $w_\gerp^{-1}(D')$ intersects $\Sigma$. By Proposition 6.8, $D'$ contains  $\spe^{-1}(V)$, where $V$ is an irreducible component of $W_{\BB, \{\beta\}}$ for some $\beta \in \BB_\gerp$. Therefore, $w_\gerp^{-1}(D')$ contains $\spe^{-1}(w_\gerp^{-1}(V))$ which intersects $\Sigma$, as $w_\gerp^{-1}(V)$ is an irreducible component of $w_\gerp^{-1}(W_{\BB, \{\beta\}})$, easily seen to be a nowhere-\'etale stratum of codimension 1.

 For part (2), we write

\[
\Sigma_S \cap w_\gerp^{-1}(\Sigma_S)= \bigcup_{T_1,T_2 \subseteq S} w_{T_1}^{-1}(\Sigma) \cap w_{T_2\cup \{\gerp\}}^{-1}(\Sigma).
\]
By Proposition \ref{Prop: connectivity}, every connected component of every term appearing in the union above intersects 
a region of the form $\tilde{\gerY}^{|\tau|\leq 1}_\rig \underline{I}$, where $\underline{I}$ is a multiset of intervals satisfying $I_\gerp=\II^\ast_\gerp$. This proves the result.
\end{proof}

  \begin{lemma}\label{Lemma: pi_1=pi_2} Let $\gern$ be an ideal of $\ol$ prime to $p$, and $\gerp\in \SS$. Let $\calW$ be any of $\prtoY^{-1}(\Sigma_S)$, $\prtoY^{-1}w_\gerp^{-1}(\Sigma_S)$, $\tHMVANtau {\underline I}$ for any $\underline{I}$. Then, $\pi_{2,\gern}^{-1}(\calW)=\pi_{1,\gern}^{-1}(\calW)=\pi_{2,j_\gern}^{-1}(\calW)$ for all $j_\gern\in \calO_L/\gern$.
\end{lemma}

\begin{proof} All of these regions are defined via the $\gerp$-divisible group of the HBAV, and dividing by a subgroup of $A[\gern]$ leaves  that $\gerp$-divisible group unchanged, as $\gerp\!\!\not|\gern$.

\end{proof}
 Consider a collection  $\{f_T: T \subset \SS\}$ of elements of  $\cM^{\dagger}_{\vk}(\taml(Np);K)$, as in Theorem \ref{Theorem: classical}. Then, by Theorem \ref{Thm:AnaCont} , each $f_T$ extends analytically to $\region$. Recall that

\begin{thm} \label{Theorem: gluing} Let the notation be as above. Each $f_T$ extends from $\region$ to
\[
\prtoY^{-1}(\Sigma_\SS)=\bigcup_{T \subset \SS} w_T^{-1}(\region).
\]
\end{thm}

\begin{proof}

We will assume below that the collection of characters is ramified at all $\gerp\in \SS$ (as explained in condition (1) in the statement of Theorem \ref{Theorem: classical}). The case where the collection of characters is unramified at all $\gerp\in \SS$ is proved in \cite{Ka}, and the intermediate cases can be proved by a straightforward mixing of the the argument in \cite{Ka} and the argument below. 

By symmetry, it is enough to prove the theorem for $f_\emptyset$. We prove by induction that if $S \subseteq \SS$, then all  elements of $\calF_S:=\{f_T: T \subseteq \SS-S\}$ can be extended to $\prtoY^{-1}(\Sigma_S)$.   We have already proven this for $S=\emptyset$ and need to prove it for $S=\SS$.   Assuming this claim holds for $S\subset \SS$, we show that it holds for $S \cup\{\gerp\}$, where $\gerp\not\in S$.

As in the proof of \cite[Theorem 10.1]{Buzzard}, we would ideally like to glue $f_T$ and $w_\gerp(f_{T\cup\{\gerp\}})$ to extend $f_T$ from $\prtoY^{-1}(\Sigma_S)$ to $\prtoY^{-1}(\Sigma_{S\cup \{\gerp\}})$. However, two complications arise. Firstly, it turns out that $w_\gerp(f_{T\cup\{\gerp\}})$ is no longer equal to $f_T$, since, unlike the classical case, the characters $\chi_T$ will have conductor away from $p$  (essentially due to the existence of totally positive units). This necessitates  the introduction of the auxiliary variety $\tHMVAN^{\underline{\gern}}$ over which a candidate for the replacement of $f_{T\cup\{\gerp\}}$ lives. Secondly,   unlike in the elliptic case, the $f_T$'s are not defined over the entire non-ordinary locus and this causes complications in the gluing process in that the overlap regions will have several connected components that need to be controlled. This issue can be resolved using our detailed study of the connectivity properties of certain regions on $\tYrig$ in \S\ref{section: gluing}. For this to work, it is crucial that our analytic continuation result, Theorem \ref{Thm:AnaCont}, provides a ``big enough'' domain of automatic analytic continuation for the $f_T$'s.

 For any fractional ideal $\gera$ of $\calO_L$, we consider a cusp $Ta_\gera$ of $\tHMVAN$ which belongs to the connected component ${\tilde{\gerZ}}_{{\rm rig},\gera}$ of $\tHMVAN$, where the polarization module is isomorphic to $(\gera,\gera^+)$ as a module with a notion of positivity:
 \[
 Ta_\gera=((\underline{\underline{\GG_m\! \otimes\! \gerd_L^{-1})/q}}^{\gera^{-1}},[{\zeta_\gerp}]),
 \]
the notation being as in the beginning of this section. We assume that the subgroup generated by the level $Np$ structure is the image of $(\GG_m\otimes \gerd_L^{-1})[pN]$. Let $\eta^{\vk}$ denote a generator of the sheaf $\omega^{\vk}$ on the base of $Ta_\gera$. Let $\gerp \not\in T \supseteq S$ be as above. For any choice of $\gera$, we write
\[
f_T(Ta_\gera)=\sum_{\xi \in (\gera^{-1})^+}{a_{\xi(\gera)}} q^\xi \eta^{\underline{k}},
\]
\[
f_{T\cup\{\gerp\}}(Ta_\gera)=\sum_{\xi \in (\gera^{-1})^+}{b_{\xi(\gera)}} q^\xi \eta^{\underline{k}},
\]
recalling that they are both normalized in the sense that  $c(\calO_L,f_T)=c(\calO_L,f_{T\cup\{\gerp\}})=1$. By \cite[(2.23)]{Shimura}, we have $a_{\xi(\gera)}=c((\xi)\gera,f_T)$, and similarly for $b_{\xi(\gera)}$.

For simplicity of notation, let us denote $\chi_{T\cup\{\gerp\}}/\chi_T$ by $\Psi$. By assumptions stated in Theorem \ref{Theorem: classical},  and since the collection of characters is assumed ramified at $\gerp$, $\Psi$ has conductor  $\gerp \gern_T$ for some integral ideal $\gern_T|N$ and   $\psi_{T,\gerp}=\psi_{T\cup\{\gerp\},\gerp}^{-1}=\Psi^{-1}|_{\calO_\gerp^\times}\!\!\! \mod 1+\gerp\calO_\gerp$ is a nontrivial character of $(\calO_L/\gerp)^\times$; we denote these common characters by $\psi_\gerp$. We set $\gern:=\gern_T$, and let $\psi_\gern$ be the character of $(\calO_L/\gern)^\times$ which is obtained  as $\Psi^{-1}|_{\hat{\calO}^\times}\!\! \mod \Pi_{q|\gern} (1+q^{\ord_\gerq\gern})$. Let $r_\gerp$ (resp., $s_\gerp$) denote the $U_\gerp$-eigenvalues of $f_T$ (resp., $f_{T\cup\{\gerp\}}$).

 Let $c_\gerp \in \gerp^{-1}\gera^{-1}-\gera^{-1}$ and $c_\gern \in \gern^{-1}\gera^{-1}-\bigcup_{\gerq|\gern} \gerq \gern^{-1}\gera^{-1}$ (where $\gerq$ runs over prime ideals) be such that  $Ta^0_{\gera}=(Ta_\gera,q^{c_\gerp},[\zeta_\gern],q^{c_\gern})$ is a cusp on $\tHMVAN^{\underline{\gerp},\underline{\gern}}$. 
 Fix an isomorphism $\ol/\gerp \times \ol/\gern \arr {\gerp}^{-1}\gern^{-1}\gera^{-1}/\gera^{-1}$ such that its restriction to $\ol/\gerp$ (resp., to $\ol/\gern$) is given by multiplication by $c_\gerp$ (resp., by $c_\gern$), and denote its inverse by  
 \[
 (t_\gerp , t_\gern) :\gerp^{-1} \gern^{-1}\gera^{-1}\!/\ \gera^{-1} \arr \calO_L/\gerp \times \calO_L/\gern.
 \]

 \

 \begin{lemma}\label{Lemma: characters} Let notation be as above. We have the following.
 
 \begin{enumerate}

 \item For any $\xi \in (\gerp^{-1}\gern^{-1}\gera^{-1})^+\!\!-\bigcup_{\gerq|\gern\gerp}\  (\gerq \gerp^{-1}\gern^{-1}\gera^{-1})^+$, we have 
 \[
 a_\xi(\gerp\gern\gera)=C \psi_\gerp(t_\gerp(\xi))\psi_\gern(t_\gern(\xi))b_\xi(\gerp\gern\gera),
 \]
where $C$ is independent of $\xi$.
 
 \item $a_{\xi(\gerp \gern\gera)}=r_{\gerp}a_{\xi(\gern\gera)}$, $b_{\xi(\gerp \gern\gera)}=s_{\gerp}b_{\xi(\gern \gera)}$ for all $\xi\in(\gern^{-1}\gera^{-1})^+$.
 
 \end{enumerate}
 \end{lemma}
 \begin{proof}  Let $\germ$ be a  fractional ideal of $L$ prime to $\gerp\gern$ (the conductor of $\Psi$). By the weak approximation theorem, there is  $\lambda_\germ \in \AA_L^\times$ such that the ideal generated by $\lambda_\germ$ equals $\germ$, and  that for every prime ideal $\gerq|\gerp \gern$, we have $(\lambda_\germ)_\gerq \equiv 1 \ {\rm mod}\ q^{\ord_\gerq(\gerp\gern)}$. In this proof, we will view $\Psi$ as a character of $\AA_L^\times/L^\times$. In particular, we have $\Psi(\germ)=\Psi(\lambda_\germ)$ for any $\germ$ prime to $\gerp\gern$, and any choice of $\lambda_\germ$ as above.
 
 By assumptions, $(\xi)\gerp\gern\gera$ is prime to $\gerp \gern$, and the assumptions in the statement of Theorem \ref{Theorem: classical} tell us that 
 \[
 a_\xi(\gerp\gern\gera)=c((\xi)\gerp\gern\gera,f_T)=\Psi(\lambda_{(\xi)\gerp\gern\gera})c((\xi)\gerp\gern\gera,f_{T\cup\{\gerp\}})=\Psi(\lambda_{(\xi)\gerp\gern\gera})b_\xi(\gerp\gera).
 \]
So to prove (1), it is enough to show that for any $\xi,\xi' $ as in part (1) of the lemma, we have $\psi_\gerp(t_\gerp(\xi')t_\gerp(\xi)^{-1})\psi_\gern(t_\gern(\xi')t_\gern(\xi)^{-1})=\Psi(\lambda_{(\xi'\xi^{-1})})$, noting that $\lambda_{\xi^{-1}\xi'}$ makes sense, as $\xi^{-1}\xi'$ is prime to $\gerp \gern$ by the choices.

 For every place $v$ of $L$, let $i_v:L^\times \ra L_v^\times \subset \AA_L^\times$ be the natural inclusion. We can write
\begin{align*}
\Psi(\lambda_{(\xi'\xi^{-1})})&=\Psi(\lambda_{(\xi'\xi^{-1})}(\xi')^{-1}\xi )=\Psi|_{\calO_\gerp^\times}(i_\gerp((\xi')^{-1}\xi)) \Pi_{\gerq|\gern} \Psi|_{\calO_\gerq^\times}(i_\gerq((\xi')^{-1}\xi))\\
&=\psi_\gerp^{-1}((\xi')^{-1}\xi \!\!\!\! \mod 1+\gerp\calO_\gerp)\psi_\gern^{-1}((\xi')^{-1}\xi \!\!\! \mod \Pi_{q|\gern} (1+q^{\ord_\gerq(\gern)}\calO_\gerq)),
\end{align*}
where we have used that $(\xi')^{-1}\xi$ is totally positive and, hence, $\Psi|_{L \otimes_\QQ \RR}((\xi')^{-1}\xi)=1$. To end the proof of part (1), we use   $t_\gerp(\xi')t_\gerp(\xi)^{-1}\equiv\xi'\xi^{-1}\!\!\mod \gerp$ (which implies that  $\psi_\gerp(t_\gerp(\xi')t_\gerp(\xi)^{-1})=\psi_\gerp(\xi'\xi^{-1} \!\!\!\!\mod 1+\gerp\calO_\gerp)$), as well as a similar statement at $\gern$.

 Part (2) is immediate since $f_T$ and $f_{T\cup\{\gerp\}}$ are $U_\gerp$-eigenforms with eigenvalues $r_\gerp$, $s_\gerp$, respectively.
 \end{proof}
 
  Let $f:=\pi_{2,\gern}^*(f_T)$ and $g:=\sum_{j_\gern \in (\calO_L/\gern)^\times} \psi_\gern^{-1}(j_\gern)\pr^\ast\pi_{2,j_\gern}^\ast(f_{T \cup \{\gerp\}})$ which are both defined on $\pi_{1,\gern}^{-1}\prtoY^{-1}(\Sigma_S)=\pi_{2,\gern}^{-1}\prtoY^{-1}(\Sigma_S)$ (Lemma \ref{Lemma: pi_1=pi_2}).
 
 \begin{lemma} \label{Lemma: twist} We have the following equality over $\pi_{1,\gerp}^{-1}\pi_{1,\gern}^{-1}\prtoY^{-1}(\Sigma_S)\subset \tHMVAN^{\underline{p},\underline{\gern}}$:
\[
\sum_{j_\gerp \in (\ol/\gerp)^\times} \psi_\gerp(j_\gerp) \pr^\ast \pi_{2,j_\gerp}^\ast(f)=CW(\psi_\gern^{-1})^{-1}W(\psi_p)(\pr^\ast \pi_{2,\gerp}^\ast(g)-s_\gerp\pi_{1,\gerp}^\ast (g)),
\]
where $W(.)$ denotes the Gauss sum.
\end{lemma}

\begin{proof} 

 By Proposition \ref{Prop:Sigma-U_p}, both sides are defined over $\pi_{1,\gerp}^{-1}\pi_{1,\gern}^{-1}\prtoY^{-1}(\Sigma_S)$. By definitions, we have, for terms appearing on the right side of the equation:
\begin{align*}
 \pr^\ast\pi_{1,\gerp}^\ast \pi_{2,j_\gern}^\ast(f_{T\cup\{\gerp\}})(Ta^0)&={\rm{pr}}^\ast f_{T\cup\{\gerp\}}(\underline{\underline{(\GG_m\otimes \gerd_L^{-1})/\!\!<\!q^{\gera^{-1}}\!\!\!,q^{c_\gern}\zeta_\gern^{j_\gern}\!>}},[\zeta_\gerp])\\ &=\sum_{\xi\in (\gern^{-1}\gera^{-1})^+} b_{\xi(\gern\gera)} \zeta_\gerp^{j_\gern t_\gern(\xi)} q^\xi \eta^{\vk}.
\end{align*}
  
  \begin{align*}
 \pr^\ast\pi_{2,\gerp}^\ast \pi_{2,j_\gern}^\ast(f_{T\cup\{\gerp\}})(Ta^0)&={\rm{pr}}^\ast f_{T\cup\{\gerp\}}(\underline{\underline{(\GG_m\otimes \gerd_L^{-1})/\!\!<\!q^{\gera^{-1}}\!\!\!,q^{c_\gerp},q^{c_\gern}\zeta_\gern^{j_\gern}\!>}},[\zeta_\gerp])\\ &=\sum_{\xi\in (\gerp^{-1}\gern^{-1}\gera^{-1})^+} b_{\xi(\gerp \gern\gera)} \zeta_\gern^{j_\gern t_\gern(\xi)} q^\xi \eta^{\vk}.
\end{align*}
  
Similarly, for the terms on the left side of the equation, we can write:

  \begin{align*}
 \pr^\ast\pi_{2,j_\gerp}^\ast \pi_{2,\gern}^\ast(f_T)(Ta^0)&=
 {\rm{pr}}^\ast f_T(\underline{\underline{(\GG_m\otimes \gerd_L^{-1})/\!\!<\!q^{\gera^{-1}}\!\!\!,q^{c_\gern},q^{c_\gerp}\zeta_\gerp^{j_\gerp}\!>}},[\zeta_\gerp])\\ &
 =\sum_{\xi\in (\gerp^{-1}\gern^{-1}\gera^{-1})^+} a_{\xi(\gerp \gern\gera)} \zeta^{j_\gerp t_\gerp(\xi)} q^\xi \eta^{\vk}.
\end{align*}
  
  Using the above calculations, Lemma \ref{Lemma: characters}, and the fact that the $U_\gerq$-eigenvalues of $f_T$'s are all assumed zero for $\gerq|N$,  it follows easily that the two sides of the desired equation have the same $q$-expansion at the cusp $Ta^0_\gera$. To prove the lemma, it is enough to show that every connected component of $\pi_{1,\gerp}^{-1}\pi_{1,\gern}^{-1}\prtoY^{-1}(\Sigma_S)\subset  \tHMVAN^{\underline{p},\underline{\gern}}$ contains such a cusp. Every such connected component maps surjectively to a connected component of $\prtoY^{-1}(\Sigma_S)$, since $\pi_{1,\gern}$ and $\pi_{1,\gerp}$ are finite flat. The claim now follows from part (1) of Lemma \ref{Lemma: final connectivity}.
\end{proof}

We continue the proof of Theorem \ref{Theorem: gluing}. By the induction assumption, $f_T$ extends to  $\Sigma_S$, and $w_\gerp(f_{T\cup\{\gerp\}})$ extends to  $w_\gerp^{-1}(\Sigma_S)$. By Lemma  \ref{Lemma: pi_1=pi_2} 
\[
h:=N_{L/\QQ}(\gerp)r_\gerp f-C\psi_\gerp(-1)W(\psi_\gerp)W(\psi_\gern^{-1})^{-1}w_\gerp(g)
\]
is defined over $\pi_{1,\gern}^{-1}(\Sigma_S \cap w_\gerp^{-1}(\Sigma_S))$. By assumption (1) in Theorem \ref{Theorem: classical}, $h$ has (nontrivial) character $\psi_\gerp$ at $\gerp$.  

 In what follows, we will use the shorthand notation $\underline{\uuA}$ for the data $(\uuA,P_\gern,Q_\gern)$.  Let $\underline{I}$ be any multiset of intervals as in Definition \ref{Definition: admissible domains} such that $I_\gerp=\II^\ast_\gerp$ and $ \Yrigtau {\underline I}\subset \Sigma$. Let $(\underline{\uuA},P_\gerp,Q_\gerp) \in \pi_{1,\gerp}^{-1}\pi_{1,\gern}^{-1}( \tHMVANtau {\underline I})= \pi_{1,\gerp}^{-1}\pi_{2,\gern}^{-1}( \tHMVANtau {\underline I})$ (Lemma \ref{Lemma: pi_1=pi_2}). We write:
\begin{align*}
N_{L/\QQ}(\gerp)r_\gerp f(\underline{\uuA},Q_\gerp)-{\rm pr}^\ast f(\underline{\uuA}/(P_\gerp),\overline{Q_\gerp})&=\sum_{j_\gerp \in (\ol/\gerp)^\times} {\rm pr}^\ast f(\underline{\uuA}/(j_\gerp P_\gerp+Q_\gerp),\overline{Q_\gerp})\\
&=\psi_\gerp(-1)\sum_{j_\gerp \in (\ol/\gerp)^\times} \psi_\gerp(j_\gerp) {\rm pr}^\ast f({\underline{\uuA}}/(j_\gerp P_\gerp+Q_\gerp),\overline{P_\gerp})\\
&=\psi_\gerp(-1)\sum_{j_\gerp \in (\ol/\gerp)^\times} \psi_\gerp(j_\gerp) \pr^\ast \pi_{2, j_\gerp}^\ast(f)(\underline{\uuA},P_\gerp,Q_\gerp)\\
&=C_0(\pr^\ast g(\underline{\uuA}/(Q_\gerp),\overline{P_\gerp})-s_\gerp g(\underline{\uuA},P_\gerp)),
\end{align*}
where, $C_0=C\psi_\gerp(-1)W( \psi_\gerp)W(\psi_\gern^{-1})^{-1}$. In the first equality, we have used the fact that $U_\gerp(f)=r_\gerp f $, and, in the last equality, we have used Lemma \ref{Lemma: twist}. We must note that by Lemma \ref{Lemma: saturated}, all the terms in the above calculation are well-defined.

It follows that if $R=(\underline{\uuA},P_\gerp,Q_\gerp)\in \pi_{1,\gerp}^{-1}\pi_{1,\gern}^{-1}(\tHMVANtau {\underline I})$, we have
\[
h(\underline{\uuA},Q_\gerq)=\pr^\ast f(\underline{\uuA}/(P_\gerp),\overline{Q}_\gerp)-C_0s_\gerp g(\underline{\uuA},P_\gerp).
\]
For  $j_\gerp \in (\ol/\gerp)^\times$, let $j_\gerp^*$ denote its inverse. Lemma \ref{Lemma: saturated} implies that  both points $R_1=(\underline{\uuA},j_\gerp^*P_\gerp-Q_\gerp,P_\gerp)$, and $R_2=(\underline{\uuA},P_\gerp-j_\gerp Q_\gerp,Q_\gerp)$ belong to $\pi_{1,\gerp}^{-1}\pi_{1,\gern}^{-1}(\tHMVANtau {\underline I})$. Applying the above equality to $R_1$ and $R_2$, we deduce that 
\[
h(\underline{\uuA},P_\gerp)=\psi_\gerp(j_\gerp)h(\underline{\uuA},Q_\gerp),
\]
for any $(\underline{\uuA},P_\gerp,Q_\gerp)\in \pi_{1,\gerp}^{-1}\pi_{1,\gern}^{-1}(\tHMVANtau {\underline I})$. Since  $\psi_\gerp$ is nontrivial, it follows that $h=0$ on $ \pi_{1,\gern}^{-1}(\tHMVANtau {\underline I})$. Choosing $\underline{I}$ as in part (2) of Lemma \ref{Lemma: final connectivity} allows us to  deduce  that $h=0$ on the bigger domain $\pi_{1,\gern}^{-1}\prtoY^{-1}(\Sigma_S \cap w_\gerp^{-1}(\Sigma_S))=\pi_{2,\gern}^{-1}\prtoY^{-1}(\Sigma_S \cap w_\gerp^{-1}(\Sigma_S))$ (Lemma \ref{Lemma: pi_1=pi_2}). Therefore, $N_{L/\QQ}(\gerp)r_\gerp f$ defined on $\pi_{2,\gern}^{-1}\prtoY^{-1}(\Sigma_S)$ and $C_0w_\gerp(g)$ defined on $w_\gerp^{-1}\pi_{2,\gern}^{-1}\prtoY^{-1}(\Sigma_S)$ glue to extend $f$ to $\pi_{2,\gern}^{-1}\prtoY^{-1}(\Sigma_S \cup w_\gerp^{-1}(\Sigma_S))=\pi_{2,\gern}^{-1}\prtoY^{-1}(\Sigma_{S \cup \{\gerp\}})$. Since $f=\pr^\ast\pi_{2,\gern}^\ast(f_T)$, it follows that $f_T$ can be analytically extended to $\prtoY^{-1}(\Sigma_{S \cup \{\gerp\}})$.  This end the proof.

\end{proof}

\section{Further analytic continuation.} \label{Subsection: second step} 

The aim of this section is to show that the overconvergent modular forms obtained in previous sections can be further analytically continued to the entire Hilbert modular variety; i.e., they are classical.

 \begin{thm}\label{thm-ac}
  Let $f$ be an overconvergent Hilbert modular form of level $\taml(Np)$ and weight $\vk$. Assume that $f$ extends analytically to $\prtoY^{-1}(\Sigma_\SS)$, where 
\[
\Sigma_\SS=\bigcup_{T \subset \SS} w_T^{-1}(\Sigma).
\]
 Then,  $f$ is classical.
 \end{thm}

Using Theorems \ref{Theorem: gluing} and  \ref{thm-ac}, one immediately deduces Theorem \ref{Theorem: classical}. In the following, we prove Theorem   \ref{thm-ac}. 

\begin{lemma}[Rigid Koecher principle]\label{lemma:Koecher} Let $V$ be an admissible formal scheme over $\cO_K$, and $\cF$  a locally free $\cO_{V}$-module of finite rank. Let $V_{\rig}$ be the  rigid generic fiber of $V$ and $\cF_{\rig}$ be the rigid analytification of $\cF$. Suppose that the special fiber $V_0$ of $V$ satisfies $S_2$-condition. If $U_0\subset V_0$ is an open subset with complement of codimension $\geq 2$, then the natural restriction map
\[H^0(V_{\rig},\cF_{\rig})\xra{\sim}H^0(\spe^{-1}(U_0),\cF_{\rig})\]
is an isomorphism.
\end{lemma}

This Lemma is a classical result in rigid analytic geometry. For a proof, see \cite[A.3]{Ti}. We now begin the proof of Theorem  \ref{thm-ac}.

\begin{proof}[Proof of Theorem \ref{thm-ac}]

 Let $\prtoY_\ast(\omega^{\vk})$ be the push-forward of the sheaf $\omegab^{\vk}$ on $\fHMV$ via the finite flat map $\prtoY:\tHMVAN\ra \tYrig$.  For any admissible open subset $U\subset \tYrig$, we can identify $H^0(\prtoY^{-1}(U),\omega^{\vk})$ with $H^0(U,\prtoY_{\ast}(\omega^{\vk}))$.  By the rigid GAGA, we just need to prove that $f$, which is \emph{a priori} a section of $\prtoY_*\omegab^{\vk}$ defined over $\Sigma_\SS$,  extends analytically  to  $\tYrig$.  If $T\subset \SS$ is a subset, we put $\gert(T)=\prod_{\gerp\in T}\gerp$ and $\gert^*(T)=(p)/\gert$. We have $w_{T}^{-1}(\vdeg^{-1}(\one))=\vdeg^{-1}(\one_{\gert^*(T)})$, where $\one_{\gert^*(T)}\in [0,1]^{\BB}$ is the vector with $\beta$-component equal to $1$ if $\beta\in \BB_{\gert^*(T)}=\cup_{\gerp|\gert^*(T)}\BB_{\gerp}$, and $0$ otherwise. Since $\Sigma$ contains $\vdeg^{-1}(\one)$ by definition,  the region $\bigcup_{T\subset \SS}w_T^{-1}(\vdeg^{-1}(\one))$ is contained in $\Sigma_\SS$. As $\Yrig$ together with $\{ w_{T}^{-1}(\vdeg^{-1}(\one)): T\subset \SS \}$ form an admissible cover of $\tYrig$, it suffices to show that $f$ can be extended analytically to  $\Yrig$.   The existence of an integral model $\fHMV$  of $\HMV_K$  finite flat over $\gerY$ implies that the sheaf $\prtoY_*(\omegab^{\vk})$ over $\Yrig$ is the rigidification of the locally free $\cO_{\gerY}$-module $\omegab^{\vk}\otimes_{\cO_{\gerY}}\prtoY_*(\cO_{\fHMV})$. Therefore, by the rigid Koecher priciple \ref{lemma:Koecher}, it suffices to show that there exists a closed subset $\overline{V}\subset \Ybar$ of codimension  $\geq 2$, such that  $\Sigma_\SS$ contains the tube $\spe^{-1}(\Ybar-\overline{V})$.

 For $\gerp\in \SS$, let $w_{\gerp}$ be the automorphism of $[0,1]^{\BB}$ given by  $\ua\mapsto \ub$ with $b_{\beta}=1-a_{\beta}$ for $\beta\in\BB_{\gerp} $ and $b_{\beta}=a_{\beta}$ otherwise. For a subset $S\subset \SS$, we define $w_{S}$ to be the composite of the $w_{\gerp}$'s with $\gerp\in S$, and $w=w_{\BB}$.  Let $\vx$ be a vertex point of the hypercube $[0,1]^{\BB}$, and $W_{\vx}$ be the corresponding  stratum of codimension $0$. We put  
\begin{gather*}
T_{0}=\{\gerp \in \SS: x_{\beta}=0, \quad \forall \beta\in \BB_{\gerp}\},\\
T_1=\{\gerp\in \SS: x_{\beta}=1, \quad \forall \beta\in \BB_{\gerp}\},\\ 
T_2=\BB\backslash(T_0\cup T_1).
\end{gather*} 
It's immediate that both the strata $W_{w_{T_0}(\vx)}$ and $W_{w_{T_0\cup T_2}(\vx)}$ are nowhere \'etale (see \ref{subsection:no-etale} for this notion).  Therefore, by definition, $\Sigma$ contains $\spe^{-1}(W_{w_{T_0}(\vx)}^\gen\cup W_{w_{T_0\cup T_2}(\vx)}^\gen)$; hence,  $\Sigma_\SS$ contains the admissible open  $w_{T_0}^{-1}(\spe^{-1}(W_{w_{T_0}(\vx)}^\gen))\cup w_{T_0\cup T_2}^{-1}\spe^{-1} (W_{w_{T_0\cup T_2}(\vx)}^\gen)$, which equals
\[
\spe^{-1}(w_{T_0}^{-1}(W_{w_{T_0}(\vx)}^\gen)\cup w_{T_0\cup T_2}^{-1}(W_{w_{T_0\cup T_2}(\vx)}^\gen))
\]
Note that $w_{T_0}^{-1}(W_{w_{T_0}(\vx)}^\gen)\cup w_{T_0\cup T_2}^{-1}(W_{w_{T_0\cup T_2}(\vx)}^\gen)$ is an open subset of $W_{\vx}$. We denote by $V_{\vx}$ its complement in $W_{\vx}$, and by $\overline{V}_{\vx}$ the  closed closure of $V_{\vx}$ in $\Ybar$. By Proposition \ref{Proposition: codim 2}, $V_{\vx}$ has codimension $\geq 2$ in $W_{\vx}$, so $\overline{V}_{\vx}$ has codimension $\geq 2$ in $\Ybar$.

Let $\bfa$ be an edge of $[0,1]^\BB$. There exists a unique $\beta_0\in \BB$ such that $\bfa$ consisting of the points $\ua=(a_{\beta})\in [0,1]^{\BB}$ such that $a_{\beta_0}\in (0,1)$ and $a_{\beta}\in \{0,1\}$ for $\beta\neq \beta_0$. Let $T_0$ be the subset of $\SS$ such that $a_{\beta}=0$ for all $\beta\in \BB_{\gerp}$. Let $W_{\bfa}$ be the corresponding stratum of codimension $1$. Then both $W_{w_{T_0}(\bfa)}$ and $W_{w_{T_0\cup\{\gerp_0\}}(\bfa)}$ are nowhere \'etale . Hence, by defintion, $\Sigma_\SS$ contains the subset
$$
\calW_{\bfa}=w_{T_0}^{-1}(\,]W_{w_{T_0}(\bfa)}^\gen['\,)\cup w_{T_0\cup\{\gerp_0\}}^{-1}(\,]W_{w_{T_0\cup \{\gerp_0\}}(\bfa)}['\,).
$$
Note that $\calW_{\bfa}\subset \,\spe^{-1}(W_{\bfa})\,$. 
We denote by $W_{w_{T_0}(\bfa)}^{\sing}$ the reduced closed subscheme associated with the complement $W_{w_{T_0}(\bfa)}^{\gen}$ in $W_{w_{T_0}(\bfa)}$. It's obvious that ${W}_{w_{T_0}(\bfa)}^{\sing}$ has codimension $\geq 1$ in $W_{w_{T_{0}}(\bfa)}$. We put 
\[V_{\bfa}=w_{T_0}^{-1}(W_{w_{T_0}(\bfa)}^{\sing})\cup w_{T_0\cup\{\gerp_0\}}^{-1}(W_{w_{T_0\cup\{\gerp_0\}}(\bfa)}^{\sing}).\]
We claim that  $\spe^{-1}(W_{\bfa}-V_{\bfa})\subset \calW_{\bfa}$. Ineed, we have
$$
W_{\bfa}-V_{\bfa}=w_{T_0}^{-1}(W_{w_{T_0}(\bfa)}^{\gen})\cup w_{T_0\cup\{\gerp_0\}}^{-1}(W_{w_{T_0\cup\{\gerp_0\}(\bfa)}}^\gen).
$$ 
Let $Q=(A,H)\in \spe^{-1}(W_{\bfa}-V_{\bfa})$. If $\deg_{\beta_0}(H)>\sum_{i=1}^{f_{\gerp_0}-1}1/p^i$, then $Q$ belongs to  $w_{T_0}^{-1}(\,]W_{w_{T_0}(\bfa)}^\gen['\,)\subset \calW_{\bfa}$ by Definition \ref{Defn:Sigma}.  Now assume $\deg_{\beta_0}(H)\leq \sum_{i=1}^{f_{\gerp_0}-1}1/p^i$. We have necessarily 
$$
1-\deg_{\beta_0}(H)\geq 1-\sum_{i=1}^{f_{\gerp_0}-1}\frac{1}{p^i}>\sum_{i=1}^{f_{\gerp_0}-1}\frac{1}{p^i},
$$
where we used the fact that $p\geq 3 $ in the last step. Hence,  $Q\in w_{T_0\cup\{\gerp_0\}}^{-1}(\,]W_{w_{T_0\cup \{\gerp_0\}}(\bfa)}[')$.  This proves the claim. Now we let $\overline{V}_{\bfa}$ be the closure of $V_{\bfa}$ in $\Ybar$. Since $W_{\bfa}$ has codimension $1$ in $\Ybar$, $\overline{V}_{\bfa}$ has codimension $\geq 2$ in $\Ybar$.  Finally, we set 
\[\overline{V}=(\bigcup_{\vx}\overline{V}_{\vx})\cup (\bigcup_{\bfa}\overline{V}_{\bfa})\cup (\bigcup_{\codim(W)\geq 2}W), \]
where $\vx$ runs through all the vertexes of $[0,1]^{\BB}$,  $\bfa$ though all the edges, and $W$ through all the strata of codimension $\geq 2$.
This is a closed subset of $\Ybar$ of codimension $2$. By the discussion above, we see that $\Sigma_\SS$ contains the tube $]\Ybar-\overline{V}[$. This finishes the proof of Theorem \ref{thm-ac}.

\end{proof}

 \section{Residual Modularity} Let $F$ be a totally real field (on which we will put various conditions). In this section, we prove that certain mod-$p$ representations of $G_F=\Gal(\overline{\QQ}/F)$ are modular.

\subsection{Modularity of icosahedral mod 5 representations of $G_F$} We begin with a Lemma.

\begin{lemma} Let $F$ be a totally real field. Suppose that $\overline\rho: G_F\rightarrow GL_2(\overline{\mathbb{F}}_5)$
is a continuous representation of $G_F=\mathrm{Gal}(\overline{\mathbb{Q}}/F)$ satisfying
\begin{itemize}
\item totally odd,
\item $\overline\rho$ has projective image $A_5$,

\end{itemize}
Then, there is a finite soluble totally real extension $F\subset M\subset \overline{\mathbb{Q}}$ and an elliptic
curve $E$ over $M$ satisfying the following conditions:
\begin{itemize}
\item $\overline\rho_{E, 5}: G_{M}=\mathrm{Gal}(\overline{\mathbb{Q}}/M)\rightarrow \mathrm{Aut}(E[5])$ is equivalent to a twist of $\overline\rho|_{G_M}$ by some character,
\item $\overline\rho_{E, 3}: \mathrm{Gal}(\overline{\mathbb{Q}}/{M(\zeta_3)})\rightarrow \mathrm{Aut}(E[3])$ is absolutely irreducible,
\item $E$ has good ordinary reduction at every place of  $M$ above 3, and potentially good ordinary reduction at every place of $M$ above $5$.
\end{itemize}
\end{lemma}

\begin{proof} 
Suppose $[F(\zeta_5): F]=4$. Then the mod 5 cyclotomic character $\epsilon: \mathrm{Gal}(F(\zeta_5)/F)\rightarrow (\mathbb{Z}/5\mathbb{Z})^\times$ is an isomorphism. Let $F_1$ be the unique quadratic extension $F(\sqrt{5})$ in $F(\zeta_5)$ of $F$ corresponding to the index 2 subgroup of $(\mathbb{Z}/5\mathbb{Z})^\times$. 

Following Taylor \cite{T:02} , since the kernel of the projection $SL_2(\mathbb{F}_5)\rightarrow PSL_2(\mathbb{F}_5)$ is $\{ \pm 1\}$, the obstruction for lifting $\mathrm{Gal}(\overline{F}/F)\rightarrow PSL_2(\mathbb{F}_5)$ to $SL_2(\mathbb{F}_5)$ lies in $H^2(\mathrm{Gal}(\overline{F}/F), \{\pm 1\})=\mathrm{Br}(F)[2]$, and we shall `annihilate' the obstruction by quadratic base-changes. Since $F(\zeta_5)$ is totally ramified at 5, so is $F_1$. Hence the image in the local Brauer group of the obstruction at every place of $F_1$ at 5 is trivial. One may and will choose $F_2$ to be a totally real quadratic extension of $F$ such that all the places of $F$ not dividing 5, at which local images of the obstruction are non-trivial,  remain prime; and every place of $F$ above $5$ splits completely. Let $L$ be the bi-quadratic extension $F_1F_2$ of $F$.  

Suppose $[F(\zeta_5):F]=2$. Choose a totally real quadratic extension $F_1$ of $F$ in which every finite place of $F$ not dividing 5, at which local images of the obstruction are non-trivial, splits completely while every place above 5 remains prime. Choose a quadratic extension $F_2$ of $F$ as above. Let $L=F_1F_2$. \\

In either case, the restriction to $\mathrm{Gal}(F(\zeta_5)/L)$ of the mod 5 cyclotomic character defines an isomorphism to the order 2 subgroup $\{\pm 1\}$ of $(\mathbb{Z}/5\mathbb{Z})^\times$.

The image of the obstruction for $\mathrm{Gal}(\overline{F}/L)\rightarrow PSL_2(\mathbb{F}_5)$ to $SL_2(\mathbb{F}_5)$ therefore lies in $H^2(\mathrm{Gal}(\overline{F}/L), \{\pm 1\})$, and it has local image trivial everywhere except at the infinite places of $L$.

Since the kernel of the square $(\mathbb{Z}/5\mathbb{Z})^\times\rightarrow \{\pm 1\}$ is $\{\pm 1\}$, the obstruction, for lifting the cyclotomic character $\mathrm{Gal}(\overline{F}/L)\rightarrow \{\pm 1\}$ to a character $\mathrm{Gal}(\overline{F}/L)\rightarrow (\mathbb{Z}/5\mathbb{Z})^\times$ whose square \emph{is} the character, lies in $H^2(\mathrm{Gal}(\overline{F}/L), \{\pm 1\})$. Evidently, the image of the obstruction in $H^2(\mathrm{Gal}(\overline{F}/L), \{\pm 1\})$ is trivial everywhere except at the infinite places of $L$. 

Combining, there is no obstruction for lifting  $\mathrm{proj}\, \overline{\rho}: \mathrm{Gal}(\overline{F}/L)\rightarrow A_5$ to $\overline\rho_L: \mathrm{Gal}(\overline{F}/L)\rightarrow GL_2(\mathbb{Z}_5)$ with its determinant mod 5 cyclotomic character. \\

Choose, by class-field
theory, a finite soluble totally real extension $L'$ of $L$ such that $\overline{\rho}_L$ is trivial when
restricted to the decomposition group of every place in $L'$ above $3$ and $5$. The construction follows from class field theory; see Lemma 2.2 in \cite{T:02} of ``Chevalley's lemma'', for example.
Let $M$ be the Galois closure of $L'$ over $F$ (note that $M$ is soluble over $F$), and $\overline{\rho}_M$ denote the restriction of $\overline\rho_L$ to $\mathrm{Gal}(\overline{\mathbb{Q}}/M)$.

As in section 1 of \cite{ST}, let $Y_{\overline\rho_M}/M$ (resp., $X_{\overline\rho_M}/M$)
denote the twist
of the (resp., compactified) modular curve $Y_5$ (resp., $X_5$) with full
level 5 structure. As proved in Lemma 1.1 \cite{ST}, the ``twist''
cohomology
class is in fact trivial, and therefore $X_{\overline\rho_M}\simeq X_5$ and
$Y_{\overline\rho_M}$ is isomorphic over $M$ to a Zariski open subset of
the projective 1-line $\mathbb{P}^1$. In particular, $Y_{\overline\rho_M}$
has infinitely many rational points.

Let $Y_{\overline\rho_M,0}(3)$
denote the degree 4 cover over $Y_{\overline\rho_M}$ which parameterises the isomorphism classes $(E, \phi_5, C)$ of elliptic curves
$E$ equipped with an isomorphism $\phi_5: E[5]\simeq \overline\rho_M$  taking the
Weil pairing on $E[5]$ to $\epsilon: \wedge^2\overline\rho_M\rightarrow \mu_5$,
and a finite flat subgroup scheme $C\subset E[3]$ of order 3.

Let $Y_{\overline\rho_M,\mathrm{split}}(3)$
denote the \'etale cover over $Y_{\overline\rho_M}$ which parameterises the isomorphism classes $(E, \phi_5, C, D)$ where $(E, \phi_5)$ is as in $Y_{\overline\rho_M}$, and where $(C, D)$ is an unordered pair, fixed by $G_{M}$, of finite flat subgroup schemes of $E[3]$ of order 3 which intersect trivially. Then, it follows from Lemma 12 in \cite{R} that $Y_{\overline\rho_M,\mathrm{split}}(3)$
and $Y_{\overline\rho_M,0}(3)$ have only finitely many rational points.

For every prime $\mathfrak{p}$ of $M$ above 3 (resp. 5), the elliptic curve $y^2=x^3+x^2-x$ (resp. a twist of the CM elliptic curve $y^2=x^3+x$) defines an element of $Y_{\overline\rho_M}(M_\mathfrak{p})$ with good ordinary reduction,
and we let $\mathcal{U}_\mathfrak{p}\subset Y_{\overline\rho_M}(M_\mathfrak{p})$ denote a (non-empty) open neighbourhood, for the 3-adic  (5-adic) topology, of the point, consisting of elliptic curves with good (resp. potentially) ordinary reduction at $\mathfrak{p}$.

By Hilbert irreducibility theorem (Theorem 1.3 in \cite{E}; see also Theorem
3.5.7 in \cite{S:92}), we may then find a rational point in $Y_{\overline\rho_M}(M)$
which lies in $\mathcal{U}_\mathfrak{p}$ for every $\mathfrak{p}$ above 15 and does \emph{not}
lie in the images of $Y_{\overline\rho_M, 0}(3)(M)\rightarrow Y_{\overline\rho_M}(M)$
and $Y_{\overline\rho_M,\mathrm{split}}(3)(M)\rightarrow Y_{\overline\rho_M}(M)$. The
elliptic curve over $M$ corresponding to the rational point is what we
are looking for.

\end{proof}

\begin{thm} \label{Theorem: mod 5 modularity} Let $F$ be a totally real
field. Let $\overline{\rho}: G_F\rightarrow GL_2(\overline{\mathbb{F}}_5)$ be a continuous representation of the absolute Galois group $G_F=\mathrm{Gal}(\overline{\mathbb{Q}}/F)$
of $F$ which satisfies the following conditions.

\begin{itemize}
\item totally odd
\item $\overline{\rho}$ has projective image $A_5$
\end{itemize}
Then $\overline{\rho}$ is modular.
\end{thm}

\begin{proof} This can be proved exactly as in \cite[\S 2]{S}.  Choose an elliptic curve $E$ over a finite soluble totally real extension $M$ of $F$ as in the preceding lemma. Replace $M$ by its finite totally real soluble extension, if necessary, to assume that the mod 3 representation $\overline\rho_{E,3}$ is unramified at every prime of $M$ above 3. As argued in the proof of Theorem
3.5.5 in \cite{K:09A}, one can do this with the absolute irreducibility of
$\overline\rho_{E,3}|_{\mathrm{Gal}(\overline{\mathbb{Q}}/M)}$ intact.

By the Langlands-Tunnell theorem, there exists a weight 1 cuspidal Hilbert eigenform $f_1$ which gives rise to $\overline\rho_{E,3}$. By 3-adic Hida theory, we may find a cuspidal Hilbert eigenform $f_2$ of weight 2 and of level prime to 3, ordinary at every prime of $M$ above 3, which gives rise to $\overline\rho_{E,3}$.  As $E$ is ordinary at 3, $f_2$ renders $\rho_{E,3}: G_{M}\rightarrow GL(T_3E)$ `strongly residually modular' in the sense of Kisin \cite{K:09A}, and it follow from Theorem 3.5.5 in \cite{K:09A} that $T_3E$ is modular. By Falting's isogeny theorem, $E$ is therefore modular. As $\overline\rho_{E,5}$ is modular, $\overline\rho|_{G_M}$ is modular. It then follows from Theorem 3.2.1 of \cite{BLGG} that $\overline\rho$ is modular.
\end{proof}

\begin{rem} When the first draft of this paper was written, our theorems about modularity of mod 5 representations of $\mathrm{Gal}(\overline{F}/F)$ came with conditions at $5$; they also appeared in our main theorem about the strong Artin conjecture. At that time, \cite{BLGG} was not written up, and, in order to establish modularity of  the mod 5 representation $r$ of the absolute Galois group of a totally real field $L$ with the image of $\mathrm{proj}\, r$ being $A_5$ (e.g. $\overline{\rho}_L$ in the proof of the lemma above), it was necessary to assume either $r$ is distinguished (with a view to making appeal to Ramakrishna/Taylor lifting argument), or  the kernel of $\mathrm{proj}\, r$ does not fix $L(\zeta_5)$ (with a view to using Khare-Wintenberger `finiteness of deformation rings' argument to lift $r$ to a characteristic zero lifting that is modular).  The main theorem of Barnet-Lamb--Gee--Geraghty \cite{BLGG} completely  solve this inconvenience for us. 
\end{rem}

\section{Hida theory and $\Lambda$-adic companion forms}

\subsection{Hilbert modular forms}

Let $L$ be a totally real field. In the following, we will define Hilbert modular forms as true automorphic forms for the group $Res_{L/\QQ}\GL_{2,F}$. See Remark \ref{Remark: HMF} for the relationship between these forms and the geometric Hilbert modular forms defined in \S \ref{Remark: HMF}.

As before, $\mathcal{O}_L$
denotes the integers of $L$, and $\mathfrak{d}_L$ the different of $L$. Let $\mathbb{A}_L=\mathbb{A}^\infty_L\times L_\infty$ denote the adeles of $L$.
By $\infty$, we shall also mean the product of the
infinite primes of $L$. For an ideal $M$ of $\mathcal{O}_L$, let $L_M$ denote the strict ray class filed of conductor $M\infty$.

Let $U_1(M)$ denote the open compact subgroup
of matrices $\begin{pmatrix}a&b\\c&d\end{pmatrix}\in GL_2(\mathcal{O}_L\otimes_\mathbb{Z}\widehat{\mathbb{Z}})$ such that $c\equiv 0$ mod $M$, and $d\equiv 1$ mod $M$. Let $C_{L,
M}$ denote the strict ray class group mod $M\infty$, i.e., $ \mathbb{A}_L^\times/L^\times(\mathbb{A}_L^{\infty, \times}\cap U_1(M))L_\infty^{+, \times}$. Let $p$ be a rational prime and fix an algebraic closure $\overline{\mathbb{Q}}_p$
and an isomorphism $\overline{\mathbb{Q}}_p\rightarrow \mathbb{C}$.\\

If $k\in \mathbb{Z}$, let $S_{k, w=1}(U_1(M);\mathbb{C})$ denote the $\mathbb{C}$-vector
space, in the sense of Hida \cite{H}, of parallel weight $k=\sum_{\tau\in \mathrm{Hom}_\mathbb{Q}(F, \mathbb{R})}k \tau$ cusp forms of level
$U_1(M)$. Let $S_k(U_1(M))$ (resp., $S_k(U_1(M);R)$ for a ring $R\subset
\overline{\mathbb{Q}}_p$) denote the subspace of forms $f$ in $S_{k,w=1}(U_1(N);\mathbb{C})$ whose Fourier coefficients $c(\mathfrak{n},
f)\in \mathbb{Z}$ (resp., $ R$) for all integral ideals $\mathfrak{n}$ of $\mathcal{O}_L$.
These spaces
come equipped with an action of $C_{L,M}$ via the diamond operators $\mathfrak{q}\mapsto \langle
\mathfrak{q}\rangle$ for a prime $\mathfrak{q}\nmid M$, $T_\mathfrak{q}$ for a prime $\mathfrak{q}\nmid M$, and $U_\mathfrak{q}$ for a prime $\mathfrak{q}|M$. Let $h_k(M)$ denote the sub $\mathbb{Z}$-algebra of $\mathrm{End}(S_k(U_1(M)))$
generated over $\mathbb{Z}$ by all these operators. For a prime $\mathfrak{q}\nmid
M$, define $S_\mathfrak{q}$ by $(\mathbf{N}_{L/\mathbb{Q}}\,
\mathfrak{q})^{k-1}\langle
\mathfrak{q}\rangle$.\\

Fix an integer $N$  prime to $p$. For the ring $\mathcal{O}$
of integers of a finite extension $K$ of $\mathbb{Q}_p$, Hida \cite[\S 3]{H} defines an idempotent $e$ and we set
$$h^0_\mathcal{O}(N)=\lim_{\leftarrow r} e(h_2(Np^r)\otimes_\mathbb{Z}\mathcal{O}).$$

We have a natural map (induced by the diamond operators)
$$\langle\ \rangle: C_{L,Np^\infty}\stackrel{\mathrm{def}}{=}\lim_{\leftarrow r} C_{L, Np^r}= \mathbb{A}_L^\times/\overline{L^\times(\mathbb{A}_L^{\infty, \times}\cap U_1(Np^\infty))L_\infty^{+, \times}}\longrightarrow
h^0_\mathcal{O}(N)^\times$$ where by $\mathbb{A}_L^{\infty, \times}\cap U_1(Np^\infty)$, we mean the set of elements in $\mathbb{A}_L^{\infty, \times}\cap U_1(N)$
which are 1 at every prime $\mathfrak{p}$ of $L$ above $p$.\\

We let $\mathrm{Tor}_{L, Np^\infty}$ (resp. $\mathrm{Fr}_{L, Np^\infty}$)
denote the torsion subgroup (resp., the maximal $\mathbb{Z}_p$ free subgroup
of rank $1+\delta$ with $\delta=0$ if the Leopoldt conjecture holds) of $C_{L,
Np^\infty}$, and let $\Lambda_\mathcal{O}$ denote the completed group algebra over $\mathcal{O}$ of $\mathrm{Fr}_{L, Np^\infty}$.
Note that $h^0_\mathcal{O}(N)$ is a $\Lambda_\mathcal{O}$-module via $\langle \ \rangle$. In \cite{H}, Hida proves that $h^0_{\mathcal{O}}(N)$ is a torsion free module of finite type over $\Lambda_\mathcal{O}$.\\

We will let $$\mathrm{Art}: \mathbb{A}_L^\times/\overline{L^\times L_\infty^{+,\times}}\simeq
\mathrm{Gal}(\overline{L}/L)^\mathrm{ab}$$
denote the (global) Artin map, normalized compatibly with the local Artin
maps which are normalized to take uniformizers to arithmetic Frobenius elements. By abuse of notation, we shall let $\mathrm{Art}$ also denote the induced homomorphism
$C_{L, Np^\infty}\rightarrow \mathrm{Gal}(L_{N}(\mu_{p^\infty})/L)$.

Let $\epsilon$ denote the cyclotomic character
$$\epsilon: \mathrm{Gal}(L_N(\mu_{p^\infty})/L)\rightarrow
\mathbb{Z}_p^\times\hookrightarrow \overline{\mathbb{Q}}_p^\times.$$

We will let $\epsilon^\mathrm{cyclo}$ denote the character
$$G_L=\mathrm{Gal}(\overline{L}/L)\twoheadrightarrow\mathrm{Gal}(\overline{L}/L)^\mathrm{ab}\twoheadrightarrow C_{L,
Np^\infty}\hookrightarrow
\mathcal{O}[C_{L,Np^\infty}]^\times=\Lambda_\mathcal{O}[\mathrm{Tor}_{L,
Np^\infty}]^\times.$$

Note that $\mathfrak{q}\mapsto S_\mathfrak{q}$ extends to $$S: \Lambda_\mathcal{O}[\mathrm{Tor}_{L,
Np^\infty}]\rightarrow h^0_\mathcal{O}(N)^\times$$

If $\mathfrak{m}$ is a maximal ideal of $h^0_{\mathcal{O}}(N)$ with residue field $k(\mathfrak{m})$,
there is (Taylor \cite{T:98}, Carayol \cite{C:86}, Wiles \cite{W}, Rogawski-Tunnell \cite{RT}) a continuous
representation $$\overline{\rho}_\mathfrak{m}:
G_L\rightarrow GL_2(k(\mathfrak{m})),$$ such that, for all prime ideals $\mathfrak{q}$
not dividing $Np$, the representation is unramified at $\mathfrak{q}$ and $\mathrm{tr}\overline{\rho}_\mathfrak{m}(\mathrm{Frob}_\mathfrak{q})=T_\mathfrak{q}$.\\

Let $\mathcal{L}$ be a finite extension of the field of fractions of $\Lambda_\mathcal{O}$, and $\mathcal{O}_\mathcal{L}$ the integral closure of $\Lambda_\mathcal{O}$ in $\mathcal{L}$. We call a $\Lambda_\mathcal{O}$-algebra homomorphism 
$$F_\mathrm{H}: h^0_\mathcal{O}(N)\rightarrow
\mathcal{O}_\mathcal{L}$$ a $\Lambda$-adic
\emph{eigenform} (`H' for Hida). If the unique maximal ideal $\mathfrak{m}\subset h^0_{\mathcal{O}}(N)$ above $\mathrm{ker}\,F_\mathrm{H}$
is non-Eisenstein, i.e., $\overline{\rho}_\mathfrak{m}$ is absolutely irreducible,
it follows from results of Nyssen \cite{N} and Rouquier \cite{R:96} that there is a continuous representation
$$\rho_{F_\mathrm{H}}: G_L\rightarrow GL_2(h^0_{\mathcal{O}}(N)_\mathfrak{m})\stackrel{F_\mathrm{H}}{\rightarrow}
GL_2(\mathcal{O}_\mathcal{L}),$$ which is unramified for
all primes $\mathfrak{q}\nmid Np$, and satisfies $\mathrm{tr}\rho_{F_\mathrm{H}}(\mathrm{Frob}_\mathfrak{q})=T_\mathfrak{q}$,
and $\mathrm{det}\rho_{F_\mathrm{H}}=S\circ \epsilon^\mathrm{cyclo}$. Moreover, it follow from work of Wiles \cite{W}
that, for every $\mathfrak{p}|p$,
$$\rho_{F_\mathrm{H}}|_{D_\mathfrak{p}}\sim \begin{pmatrix} \ast&\ast\\0&\phi_{\mathfrak{p}}\end{pmatrix},$$
where $\phi_{\mathfrak{p}}$ is the unramified character of the decomposition group $D_\mathfrak{p}$ at $\mathfrak{p}$ sending $\mathrm{Frob}_\mathfrak{p}$ to $F(U_\mathfrak{p})$, and the product of the diagonal characters is $(F_\mathrm{H}\circ
S)\circ \epsilon^\mathrm{cyclo}|_{D_\mathfrak{p}}$ .\\

\subsection{$\Lambda$-adic companion forms} Let $F$ be a totally real field with ring of integers $\mathcal{O}_F$. Let  $\mathbb{S}_F$ be the set of prime ideals of $\mathcal{O}_F$ above $p$. Assume that $p$ is unramified in $F$. In the following, we construct a finite totally real soluble extension
$L$ of $F$ and $2^{|\SS_L|}$ overconvergent Hilbert modular forms on $Res_{L/\mathbb{Q}}\mathrm{GL}_{2,L}$
whose various twists by characters of finite order give rise to the Galois representation in question.

\begin{thm}\label{thm: Lambda-adic}

Let $\mathcal{O}$ be the ring of integers of a inite extension of $\mathbf{Q}_p$, with maximal ideal $\mathfrak{m}$. Let $\rho: \mathrm{Gal}(\overline{F}/F)\rightarrow  GL_2(\mathcal{O})$ be a continuous representation satisfying:

\begin{itemize}
\item $\rho$ is unramified at only finitely many places of $F$,
\item if $\overline{\rho}\stackrel{\mathrm{def}}{=}(\rho\ \mathrm{mod}\ \mathfrak{m})$, there exists a cuspidal Hilbert modular eigenform $f$ such that $\rho_f$ is potentially Barsotti-Tate and nearly ordinary at every place of $F$ above $p$ and such that $\overline\rho_f\simeq \overline\rho$,
\item  $\overline\rho$ is absolutely irreducible when restricted to $\mathrm{Gal}(\overline{F}/F(\zeta_p))$, 
\item for every place $\mathfrak{p}$ of $F$ above $p$, the restriction of $\rho$ to the decomposition group at $\mathfrak{p}$ is the direct sum of characters $\alpha_\mathfrak{p}$ and $\beta_\mathfrak{p}$, such that the images of the inertia subgroup at $\mathfrak{p}$ are both finite, and $(\alpha_\mathfrak{p}\ \mathrm{mod}\ \mathfrak{m})\neq (\beta_\mathfrak{p}\ \mathrm{mod}\ \mathfrak{m})$. 
\end{itemize}

Then, there exists a finite totally real soluble extension $L$ of $F$ in which every prime of $F$ above $p$ splits completely, a finite set $\calS$ of finite places of $L$ (where  $N$ is an integer prime to $p$ and divisible once by all  the places in $\calS$), and, for any subset $T$ of the set of places of $L$ above $p$,  a character 
$$\chi_T:\mathrm{Gal}(\overline{F}/L)\rightarrow \mathcal{O}^\times$$ 
and a $\Lambda$-adic eigenform
$$F_{\mathrm{H}, T}: h^0_\mathcal{O}(N)\rightarrow \mathcal{O}_\mathcal{L},$$
where $\mathcal{O}_\mathcal{L}$ is the integral closure of $\Lambda_\mathcal{O}$ in a finite extension $\mathcal{L}$ of $\mathrm{Frac} \Lambda_\mathcal{O}$, together with a height one prime $\wp_T$ such that 
$$(\rho_{F_{\mathrm{H}, T}}\ \mathrm{mod}\ \wp_T)\simeq \rho|_{\mathrm{Gal}(\overline{F}/L)}\otimes\chi_T^{-1}.$$
Furthermore, we have $(F_{\mathrm{H}, T}(U_\mathfrak{p})\ \mathrm{mod}\ \wp_T)=\alpha_\mathfrak{p}/\chi_T$ if $\mathfrak{p}$ in $T$, and $(F_{\mathrm{H}, T}(U_\mathfrak{p})\ \mathrm{mod}\ \wp_T)=\beta_\mathfrak{p}/\chi_T$, if $\gerp\not\in T$. \end{thm}

\begin{proof} As in the proof of Theorem 3.5.5, \cite{K:09A}, we may and will choose  $L$ to be a finite totally real soluble extension of $F$ such that 
\begin{itemize}
\item every prime $\mathfrak{p}$ of $F$ above $p$ splits completely in $L$, 

\item if $\rho_L\stackrel{\mathrm{def}}{=}\rho|_{\mathrm{Gal}(\overline{F}/L)}$, then $\rho_L$ is totally odd, and is ramified precisely at a finite set $\calS$ of finite places of $L$ and possibly at places above $p$; at every place in $\calS$, the image of the inertia subgroups is unipotent,

\item $\overline\rho_L=(\rho_L\ \mathrm{mod}\ \mathfrak{m})$ is unramified outside $p$ and is absolutely irreducible when restricted to $L(\zeta_p)$, 

\item there exists a cuspidal automorphic representation $g=\mathrm{BC}(f)$ of $GL_2(\mathbb{A}_L)$ such that $\overline\rho_g\sim \overline{\rho}$, and $g$ is nearly ordinary at every place of $L$ above $p$ and is special at  every place in $\calS$.
\end{itemize}

By class field theory, one can choose a character $\chi_T$  satisfying $\chi_T|_{I_\mathfrak{p}}=\alpha_\mathfrak{p}|_{I_\mathfrak{p}}$ for $\mathfrak{p}\in T,$ and $\chi_T|_{I_\mathfrak{p}}=\beta_\mathfrak{p}|_{I_\mathfrak{p}}$ for $\mathfrak{p}\not\in T$.  Let $\SS_L^{ur}$ be the set of primes $\gerp|p$ such that $\alpha_\gerp/\beta_\gerp$ is unramified. We can and do arrange for the characters $\chi_T$ to satisfy $\chi_T=\chi_{T-\SS_L^{ur}}$ for all $T$. In other words, if $\gerp \in \SS_L^{ur}-T$, then $\chi_T=\chi_{T\cup\{\gerp\}}$.

Let $\rho_T=\rho_L\otimes\chi_T^{-1}$ and $\overline{\rho}_T=(\rho_T\ \mathrm{mod}\ \mathfrak{m})$.  Since $\rho_T$ is $p$-distinguished, we may appeal to the Ramakrishna/Taylor lifting argument (see Gee's proof of Theorem 3.4 of \cite{G}) to construct a potentially Barsotti-Tate $p$-ordinary lifting $\widetilde\rho_T$ of $\overline{\rho}_T$ such that 
$$\widetilde\rho_T|_{D_\mathfrak{p}}\simeq\begin{pmatrix}\ast&\ast\\ 0&\phi_T\end{pmatrix}$$
where $\phi_T$ is an unramified lifting of $(\alpha_\mathfrak{p}/\chi_T\ \mathrm{mod}\ \mathfrak{m})$ if $\mathfrak{p}\in T$ and $(\beta_\mathfrak{p}/\chi_T\ \mathrm{mod}\ \mathfrak{m})$ if $\mathfrak{p}\not\in T$. It is indeed strongly residually modular in the sense of Kisin, since the twist of $\rho_{g}$ by the Teichmuller lifting of $\overline\chi_T$ defines a modular lifting of $\overline\rho_T$, and (since $\overline{\rho}_T$ is ordinary) it follows from Jarvis's level lowering result \cite{J:04} and the Fontaine-Laffaille theory that the twist is ordinary and Barsotti-Tate at every place of $L$ above $p$ .  It then follows from Theorem 3.5.5 of \cite{K:09A} that there exists a $p$-ordinary Hilbert modular form $g_T$ whose associated Galois representation is isomorphic to $\widetilde\rho_T$.

Let $\mathfrak{m}_T$ denote the maximal ideal of $e h_2(N)$ corresponding to $\overline{\rho}_{g_T}$; in other words, $\overline{\rho}_{\mathfrak{m}_T}\sim \overline{\rho}_T$ and $(T_\mathfrak{p}\ \mathrm{mod}\ \mathfrak{m}_T)=(\alpha_\mathfrak{p}/\chi_T)$ for $\mathfrak{p}\in T$ and $(T_\mathfrak{p}\ \mathrm{mod}\ \mathfrak{m}_T)=(\beta_\mathfrak{p}/\chi_T)$ for $\mathfrak{p}\not\in T$. 

It follows from Section 5 of \cite{S} that there exists $F_{\mathrm{H}, T}: h^0(N)_{\mathfrak{m}_T}\rightarrow \mathcal{O}_\mathcal{L}$ and a height one prime $\wp_T$ of $\mathcal{O}_\mathcal{L}$ such that $(\rho_{F_{\mathrm{H}, T}}\ \mathrm{mod}\ \wp_T)\sim \rho_T$ and $(F_{\mathrm{H}, T}(U_\mathfrak{p})\ \mathrm{mod}\ \wp_T)=\alpha_\mathfrak{p}/\chi_T$ for  $\mathfrak{p}$ in $T$, and $(F_{\mathrm{H}, T}(U_\mathfrak{p})\ \mathrm{mod}\ \wp_T)=\beta_\mathfrak{p}/\chi_T$, for $\gerp \not\in T$.

\end{proof}
For every $T \subset \SS_L$, we define $f_T=(F_{\mathrm{H}, T}\ \mathrm{mod}\ \wp_T)$. It is clear that the $f_T$'s are $p$-adic eigenforms of parallel weight one, and it is standard that they are indeed overconvergent.

\begin{lemma}\label{lemms: relations} The assumptions stated in Theorem \ref{Theorem: classical} hold for the collection of eigenforms $\{f_T: T \subset \SS_L\}$, and the collection of Hecke characters $\{\chi_T: T \subset \SS_L\}$.

\end{lemma} 

\begin{proof} Let $N$ be an integer prime to $p$ and divisible by the conductor of $\rho_T|_{G_L}$ for all $T \subset \SS_L$. Then, the conductor of $\chi_{T\cup\{\gerp\}}/\chi_T$ divides $p^\infty N$. Condier $\gerp \in \SS_L$, and $T\subset \SS_L$ not containing $\gerp$. First assume that $\alpha_\gerp/\beta_\gerp$ is unramified. By construction of $\chi_T$'s at the beginning of the proof of Theorem \ref{thm: Lambda-adic},  we have $\chi_{T\cup\{\gerp\}}/\chi_T=1$ as desired.  Now, assume that $\alpha_\gerp/\beta_\gerp$ is not unramified.  If $\gerq \in \SS_L-\{\gerp\}$, then, viewed as characters of $\AA_L^\times/L^\times$, we have $\chi_{T\cup\{\gerp\}}|_{\calO_\gerq^\times}=\chi_T|_{\calO_\gerq^\times}$, as, by construction, they are either both $\alpha_\gerp|_{\calO_\gerq^\times}$ or both $\beta_\gerq|_{\calO_\gerq^\times}$. Similarly, we can see that $(\chi_{T\cup\{\gerp\}}/\chi_T)|_{\calO_\gerp^\times}=(\alpha_\gerp/\beta_\gerp)|_{\calO_\gerp^\times}\neq 1$.  Since $\alpha_\gerp/\beta_\gerp$ is tamely ramified, it follows that  $\chi_{T\cup\{\gerp\}}/\chi_T$ has conductor $\gerp\gern_T$, for some $\gern_T |N$.

For any $T \subset \SS$, let $\Psi_T$ be the diamond character of $f_T$ which we will view interchangeably as both a character of $\AA_L^\times/L^\times$ and $G_L$.  The determinant of $\rho_{f_T}\cong \rho|_{G_L}\otimes \chi_T^{-1}$ is $\Psi_T$.  From the shape of $\rho$ at $\gerp$, it follows that  $\Psi_T|_{D_\gerp}$ is $\alpha_\gerp\beta_\gerp\chi_T^{-2}|_{D_\gerp}$. Viewing $\Psi_T$ as a character of $\AA_L^\times/L^\times$, we find
${\Psi_T}|_{\calO_\gerp^\times} =(\alpha_\gerp/\beta_\gerp)|_{\calO_\gerp^\times}$ since $\gerp \not \in T$ and, hence, $\chi_T|_{\calO_\gerp^\times}=\beta_\gerp|_{\calO_\gerp^\times}$. Similarly, we find that $\Psi_{T\cup\{\gerp\}}|_{\calO_\gerp^\times} =(\beta_\gerp/\alpha_\gerp)|_{\calO_\gerp^\times}$.
It follows that ${\Psi_T}|_{\calO_\gerp^\times}=(\Psi_{T\cup\{\gerp\}}|_{\calO_\gerp^\times})^{-1}=(\chi_{T\cup\{\gerp\}}/\chi_T)|_{\calO_\gerp^\times}$. Therefore, we have 
\[
\psi_{\gerp,T}^{-1}=\psi_{\gerp,T\cup\{\gerp\}}=(\chi_{T\cup\{\gerp\}}/\chi_T)|_{\calO_\gerp^\times}\mod 1+\gerp\calO_\gerp.
\]
In particular, it follows that all $f_T$'s are of level $\Gamma_1(pN)$, where $N$ is an integer prime to $p$ and divisible by the Artin conductor of $\rho|_{G_L}$.

Let $\gerl$ be a prime ideal prime to $Np$. Comparing ${\rm tr}
(\rho_T(\Fr_\gerl))$ and ${\rm tr}(\rho_{T\cup\{\gerp\}}(\Fr_\gerl))$, it follows that $c(f_T,\gerl)=(\chi_{T\cup\{\gerp\}}/\chi_T)(\gerl) c(f_{T\cup\{\gerp\}},\gerl)$. On the other hands, comparing the determinants of $\rho_T(\Fr_\gerl)$ and $\rho_{T\cup\{\gerp\}}(\Fr_\gerl)$, implies that $\Psi_T/\Psi_{T\cup\{\gerp\}}=(\chi_{T\cup\{\gerp\}}/\chi_T)^2$. Using these facts, and the formula for the Hecke operators $T_{\gerl^r}$ in terms of $T_\gerl$ and $\Psi_T$ (for a positive integer $r$), it is easy to see that 
\[
c(f_T,\germ)=(\chi_{T\cup\{\gerp\}}/\chi_T)(\germ) c(f_{T\cup\{\gerp\}},\gerl)
\]
for any ideal $\germ \subset \calO_L$ which is prime to $Np$. If $\gerq|p$ is a prime ideal different from $\gerp$, then comparing $c(\gerq,f_T)$ and $c(\gerq,f_{T\cup\{\gerp\}})$ implies that the above relation also holds for $\germ=\gerq$, and hence for all $\germ$ prime to $\gerp N$.

Also, using a standard trick, we can assume $c(f_T,\gerq)=0$ for all $T \subset \SS$, and all $\gerq|N$. This would require increasing $N$, but doesn't require changing the prime ideal factors of $N$.  Finally, the formula for $c(\gerp,f_T)$ shows that it is nonzero, and the assumption on the distinguishability of $\rho$ implies that $c(\gerp,f_T)\neq c(\gerp,f_{T\cup\{\gerp\}})$ if $\alpha_\gerp/\beta_\gerp$ is unramified.
\end{proof}

\section{Modularity of Artin representations}

 \begin{thm} Let $F$ be a totally real field in which $5$ is unramified.
Let $$\rho: G_F=\mathrm{Gal}(\overline{\mathbb{Q}}/F)\rightarrow GL_2(\mathbb{C})$$
be a continuous representation satisfying the following conditions:
\begin{itemize}
\item $\rho$ is totally odd,
\item $\rho$ has projective image isomorphic to $A_5$,
\item for every place $\mathfrak{p}$ of $F$ above $5$, the projective image of the decomposition group at $\mathfrak{p}$ has order 2.
\end{itemize}
Then, there exists a holomorphic Hilbert cuspidal eigenform $f$ of parallel weight 1 such that $\rho$ arises from $f$ in the sense of Rogawski-Tunnell, and, hence, the Artin $L$-function
$L(\rho, s)$ is entire.
\end{thm}

\begin{proof}  Since $\rho$ is an Artin representation, it has finite image and hence
ramifies at only finitely many places of $F$; if we choose an isomorphism
$\iota: \overline{\mathbf{Q}}_5\rightarrow \mathbb{C}$, so is
$\rho_\iota:=\iota^{-1}\circ \rho$. Conjugate $\rho_\iota$ if necessary to assume that the image of
$\rho_\iota$ lies in $GL_2(\overline{\mathbb{Z}}_5)$, and let
$\overline{\rho}_\iota$ denote the residual representation $G_F\rightarrow
GL_2(\overline{\mathbb{F}}_5)$.

While $\mathrm{Im}\,(\mathrm{proj}\, \rho_\iota)$ reduces mod 5 to
$\mathrm{Im}\, (\mathrm{proj}\,
\overline{\rho}_\iota)$, $\mathrm{Im}\,(\mathrm{proj}\, \rho_\iota)\simeq
A_5$ is simple, and,
hence,  $\mathrm{Im}\,(\mathrm{proj}\, \rho_\iota) \simeq \mathrm{Im}\,
(\mathrm{proj}\, \overline{\rho}_\iota)$. In particular, $\mathrm{Im}\,
(\mathrm{proj}\, \overline{\rho}_\iota)$ is isomorphic to
$PSL_2(\mathbb{F}_5)$. Because of the assumption on the order of the projective image of the
decomposition group $D_\mathfrak{p}$ at $\mathfrak{p}$ above $5$,
$\rho_\iota|_{D_\mathfrak{p}}$ is a direct sum of characters
$\alpha_\mathfrak{p}$ and $\beta_\mathfrak{p}$, and the image of
$\mathrm{Im}\, (\mathrm{proj}\, \rho_\iota|_{D_\mathfrak{p}})=\mathrm{Im}\,
(\alpha_\mathfrak{p}/\beta_\mathfrak{p})$ assumes a cyclic subgroup of order
2 in $PSL_2(\mathbb{F}_5)$. Furthermore, the image of the wild inertia
subgroup at $\mathfrak{p}$ above 5 cannot have order 2, hence 1, i.e.,
$\alpha_\mathfrak{p}/\beta_\mathfrak{p}$ is tamely ramified at
$\mathfrak{p}$.


By Theorems \ref{Theorem: mod 5 modularity} and \ref{thm: Lambda-adic}, there is a finite soluble totally real
field extension $L$ of $F$, and $2^{|\mathbb{S}_L|}$ overconvergent cusp
eigenforms $\{f_T\}_{T\subset\mathbb{S}_L}$ and $2^{|\mathbb{S}_L|}$
characters $\chi_T$ of $\mathrm{Gal}(\overline{F}/L)$ of finite order such
that 
$\rho_{f_T}=\rho_\iota|_{\Gal({\overline{\QQ}/L})}\otimes\chi^{-1}_T$. By
Theorem \ref{Theorem: classical}, the $f_T$ are
indeed classical cusp eigenforms. Untwisting,
$\rho_\iota|_{\Gal({\overline{\QQ}/L})}$ arises from a classical weight
one form on $Res_{L/\mathbb{Q}}GL_{2, L}$. By automorphic descent, it 
follows that $\rho_\iota$ arises from a classical weight one form on
$Res_{F/\mathbb{Q}}GL_{2, F}$.

\end{proof}

\end{document}